\documentclass[12pt]{amsart}
\usepackage{amssymb}
\usepackage{hyperref}
\usepackage[margin=1.25in]{geometry}
\usepackage[utf8]{inputenc}
\usepackage{chngcntr}
\usepackage{thm-restate}
\usepackage{url}
\usepackage{filecontents}

\newtheorem{theorem}{Theorem}[]
\newtheorem{corollaryy}{Corollary}[]
\newtheorem{proposition}{Proposition}[section]
\newtheorem{lemma}[proposition]{Lemma}
\newtheorem{corollary}[proposition]{Corollary}

\newtheorem{claim}[proposition]{Claim}
\theoremstyle{definition}
\newtheorem{definition}[proposition]{Definition}
\theoremstyle{remark}
\newtheorem{remark}[proposition]{Remark}

\theoremstyle{definition}
\newtheorem{definitionn}{Definition}[]

\newcommand{\bb}{b}
\newcommand{\cc}{c}
\newcommand{\uu}{u}

\newcommand{\Gaf}{\mathcal{G}_{\mathsf{AF}}}
\newcommand{\Gk}{\mathcal{G}_{k}}
\newcommand{\origin}{\textbf{\textit{o}}}
\newcommand{\nut}{g_{\mathsf{TNUT}}}
\newcommand{\ode}{\mathsf{J}_{1}}
\newcommand{\oded}{\mathsf{J}_{2}}
\newcommand{\R}{\mathbb{R}}  % The real numbers.

\begin{document}
\title{Convergence of Ricci flow solutions to Taub-NUT}
\author{Francesco Di Giovanni}
\address{Department of Mathematics, University College London, Gower Street, London, WC1E 6BT, United Kingdom} 
\email{francesco.giovanni.17@ucl.ac.uk}
\begin{abstract}
%We study the Ricci flow evolving from an SU(2) cohomogeneity-1 metric $g_{0}$ on $\R^{4}$ with monotone warping coefficients and whose restriction to any hypersphere is a Berger metric. If $g_{0}$ is ALF, then either the solution develops a Type-II(b) singularity and converges in the Cheeger-Gromov sense to the Ricci-flat Taub-NUT metric $\nut$ or encounters a Type-III singularity, depending on whether the Hopf-fiber is bounded or not respectively. In fact, we prove that the Ricci flow still converges to $\nut$ in if $g_{0}$ has bounded Hopf-fiber, curvature controlled by the size of the principal orbits and opens faster than a paraboloid along the $\mathbb{CP}$-direction. As a by-product we show that there exists $g_{0}$ with $\text{sec}(g_{0}) \geq 0$ flowing to $\nut$ in infinite-time. Finally, we obtain a uniqueness result for $\nut$ which characterizes possible (collapsed) infinite-time singularity models.% by proving that the only complete warped Berger ancient solution with bounded curvature, monotone coefficients, bounded Hopf-fiber and opening Moreover, we prove that the 

We study the Ricci flow starting at an SU(2) cohomogeneity-1 metric $g_{0}$ on $\R^{4}$ with monotone warping coefficients and whose restriction to any hypersphere is a Berger metric. If $g_{0}$ has bounded Hopf-fiber, curvature controlled by the size of the orbits and opens faster than a paraboloid in the directions orthogonal to the Hopf-fiber, then the flow converges to the Taub-NUT metric $\nut$ in the Cheeger-Gromov sense in infinite time. We also classify the long-time behaviour when $g_{0}$ is asymptotically flat. In order to identify infinite-time singularity models we obtain a uniqueness result for $\nut$. 

%If $g_{0}$ is ALF, then either the solution develops a Type-II(b) singularity and converges in the Cheeger-Gromov sense to the Ricci-flat Taub-NUT metric $\nut$ or encounters a Type-III singularity, depending on whether the Hopf-fiber is bounded or not respectively. In fact, we prove that the Ricci flow still converges to $\nut$ in if $g_{0}$ has bounded Hopf-fiber, curvature controlled by the size of the principal orbits and opens faster than a paraboloid along the $\mathbb{CP}$-direction. As a by-product we show that there exists $g_{0}$ with $\text{sec}(g_{0}) \geq 0$ flowing to $\nut$ in infinite-time. Finally, we obtain a uniqueness result for $\nut$ which characterizes possible (collapsed) infinite-time singularity models.

%Finally we prove that a Type-I warped Berger Ricci flow with bounded curvature at infinity contains minimal 3-spheres for times close to the maximal time of existence. 

\end{abstract}
\maketitle
%\newcommand{\ode1}{$\mathsf{J}_{1}$}
%Setting: we let $g_{0}\in\Gin$ be asymptotically flat, with cubic volume growth, mass $m^{-1} > 0$ and both $\bb_{s}$ and $\cc_{s}$ positive. Then we have the following estimates (to double-check):
%\begin{lemma}
%There exists $\alpha > 0$ such that the following estimates are satisfied in $\R^{4}\times [0,+\infty)$:
%\begin{itemize}
%\item[(i)] $\left\vert\frac{\cc_{s}}{\cc} - \frac{\bb_{s}}{\bb}\right\vert \leq \alpha$.
%\item[(ii)] $\frac{1}{\cc}\left(\bb_{s}^{2} - 4\right)\leq \alpha$.
%\item[(ii)] $\frac{1}{\bb}\left(\frac{\bb}{\cc}-1\right)\leq \alpha$.
%\item[(iv)] $\lvert \bb_{s}\rvert + \lvert \cc_{s}\rvert \leq \alpha$.
%\end{itemize}
%\end{lemma}
\section{Introduction}
%\textbf{Should I write about the soliton? Should I specify solutions are immortal? Should I mention that sec positive converge to Ricci flat to emphasize that assumptions are quite weak? Title? Organization of intro, with an eye on mentioning the result for the ALF-setting although the proof still follows from the uniqueness result for \nut? Ask about the mass maybe? Should I specify that Taub-NUT is Ricci-flat? Is it clear what I mean by "uniqueness result" for \nut?}
In \cite{threemanifolds}, Hamilton introduced a method to study the topology of manifolds by flowing Riemannian metrics in the direction of the Ricci tensor: a family of metrics $g(t)$ solve Hamilton's Ricci flow on a manifold $M$ with initial condition $g_{0}$ if 
\[
\frac{\partial g}{\partial t} = -2\text{Ric}_{g(t)}, \,\,\,\,\,\,\, g(0) = g_{0}.
\]

If the initial metric $g_{0}$ is complete and has bounded curvature, then from celebrated works of Shi \cite{shi} and Chen-Zhu \cite{uniqueness} we derive that there exists a solution to the problem and that such solution is unique in the class of complete, bounded curvature solutions respectively: from now on, we always consider maximal complete, bounded curvature solutions to the Ricci flow. By \cite{shi} we know that a solution to the Ricci flow exists smoothly for all positive times if and only if the curvature is bounded on any time slice. Since the flow is a heat-type evolution problem for Riemannian metrics, it is tempting to expect immortal solutions to approach more regular structures in infinite time. In this regard, the behaviour of a solution existing for all positive times has been classified as follows, depending on whether the curvature decays at least as fast as $t^{-1}$ \cite{formationsingularities}:
\begin{align*}
\emph{Type-II(b)}&:\,\,\,\,\,\limsup_{t\nearrow \infty}\left(\sup_{M}\,t\lvert \text{Rm}_{g(t)}\rvert_{g(t)}\right) = \infty, \\
\emph{Type-III}&:\,\,\,\,\,\limsup_{t\nearrow \infty}\left(\sup_{M}\,t\lvert \text{Rm}_{g(t)}\rvert_{g(t)}\right) < \infty.
\end{align*}

Several examples of Type-III singularities for the Ricci flow have been found, both in the compact setting \cite{lottsesum} and in the non-compact one \cite{woolgar}. In fact, some of these cases have been shown to be occurrences of more general phenomena related to either the dimension or the existence of many symmetries: Bamler proved that any closed 3-dimensional immortal Ricci flow encounters a Type-III singularity in infinite time \cite{bamler}, while B\"ohm showed that the same conclusion applies to any immortal \emph{homogeneous} Ricci flow \cite{bohmIII}. 

If a solution develops a Type-III singularity and converges smoothly \emph{without rescaling} in the Cheeger-Gromov sense, meaning that some control on the injectivity radius is available, then the limit is flat. Since Ricci-flat metrics constitute fixed points for the flow, it is natural to search for immortal solutions converging to Ricci-flat \emph{non}-flat metrics in infinite time, thus encountering Type-II(b) singularities. In this sense, only few results are known and most of them are stability properties: the initial condition needs to be \emph{sufficiently close} to the Ricci-flat metric for the Ricci flow solution to converge. For such results, whether the underlying topology is compact or not plays a key role in the analysis. Using Perelman's $\lambda$-functional, Haslhofer and M\"uller \cite{haslhofer} proved stability properties for closed Ricci-flat spaces, generalizing earlier work of \cite{sesumstability}. In the non-compact setting, Deruelle and Kr\"oncke derived a stability result for a class of ALE Ricci flat manifolds \cite{alix}. 

Since the Ricci flow preserves isometries, one might consider looking for solutions converging to Ricci-flat fixed points when symmetries are present. In this direction, Marxen recently generalized earlier work of Hamilton to prove that if $(N,g_{N})$ is closed and Ricci-flat, then a class of warped product solutions to the Ricci flow $(\R\times N,g(t))$, of the form $g(t) = k^{2}(x,t)dx^{2} + f^{2}(x,t)g_{N}$, converge to $(\R\times N,dx^{2} + c^{2}g_{N})$, for some $c > 0$, whenever the initial condition is asymptotic to the target Ricci-flat metric \cite{marxen}. On the other hand, in the maximally symmetric case of homogeneous Ricci flows, convergence to Ricci-flat non-flat spaces is not possible due to a classic result of Alekseevskii and Kimelfeld \cite{classichom}. One of the main contributions of this work consists in proving that a large family of cohomogeneity-1 metrics on $\R^{4}$ converge to the Ricci-flat Taub-NUT metric in infinite time along the Ricci flow. 

We briefly describe the class of metrics we use as initial data for the flow. Any complete metric $g$ which is both invariant under the cohomogeneity-1 left-action of SU(2) on $\R^{4}$ and under rotations of the Hopf-fibres can be diagonalized with respect to a fixed Milnor frame and hence be written, away from the origin, as:
\[
g = ds^{2} + \bb^{2}(s)\,\pi^{\ast}g_{S^{2}(\frac{1}{2})} + \cc^{2}(s)\,\sigma_{3}\otimes \sigma_{3},
\]
\noindent where $s$ is the $g$-distance from the origin, $\pi^{\ast}g_{S^{2}(\frac{1}{2})}$ is the pull-back of the Fubini-Study metric under the Hopf map, and $\sigma_{3}$ is the one-form dual to the vector field tangent to the Hopf-fibres. 

\begin{definitionn}
An SU(2)U(1) invariant metric $g$ on $\R^{4}$ is a \emph{warped Berger metric} if $\cc/\bb \leq 1$ on $\R^{4}$. Moreover $g$ has \emph{monotone coefficients} if $\bb_{s}\geq 0$ and $\cc_{s}\geq 0$ on $\R^{4}$.
\end{definitionn}

We note that the notion of warped Berger metrics is consistent with the notations adopted in \cite{IKS1}. The Ricci flow problem in this symmetry class has been studied on different topologies and examples of non-rotationally symmetric Type-I and Type-II(a) singularities have been constructed in \cite{IKS1} and in \cite{appleton2},\cite{work2} respectively. A well-known warped Berger metric on $\R^{4}$ is given by the Taub-NUT metric $\nut$, which can be written explicitly as 
\[
\nut = \frac{1}{16}\left(1 + \frac{2m^{-1}}{x}\right)\,dx^{2} + \frac{x^{2}}{4}\left(1+ \frac{2m^{-1}}{x}\right)\,\pi^{\ast}g_{S^{2}(\frac{1}{2})} + \frac{m^{-2}}{1 + \frac{2m^{-1}}{x}}\,\sigma_{3}\otimes\sigma_{3},
\]
\noindent for some parameter $m$ which we call the \emph{mass} of $\nut$ and which measures the inverse of the \emph{finite} size of the Hopf-fiber at spatial infinity. The Taub-NUT metric is a gravitational instanton found on $\R^{4}$ by Hawking \cite{hawking}: it is a \emph{hyperk\"ahler} and thus \emph{Ricci-flat} asymptotically flat metric.  We point out that the stability result in \cite{alix} does \emph{not} apply to the Taub-NUT metric which is not ALE being $(\R^{4}\setminus \{\origin\},\nut)$ the total space of a circle fibration with fibres approaching constant length at spatial infinity.

In \cite{work2} we proved that if $g_{0}$ is a complete warped Berger metric with monotone coefficients and curvature decaying at spatial infinity, then the maximal Ricci flow solution starting at $g_{0}$ is immortal. In light of such result and being $\nut$ an asymptotically flat metric, we first focus on the following family of initial data:
\begin{definitionn}\label{ALFmetrics} The class $\Gaf$ consists of all complete warped Berger metrics $g$ on $\R^{4}$ with monotone coefficients satisfying 
\begin{equation}\label{epsilondecay}
\sup_{p\in\R^{4}}\,\left(d_{g}(\origin,p)\right)^{2+\epsilon}\lvert\text{Rm}_{g}\rvert_{g}(p) < \infty, 
\end{equation}
\noindent for some $\epsilon > 0$. A metric $g\in\Gaf$ is called \emph{asymptotically flat}.
\end{definitionn}
The class $\Gaf$ divides in two categories (see Lemma \ref{characterizationALF}): metrics with cubic volume growth, for which $\bb$ opens up linearly and the Hopf-fiber approaches a positive finite quantity $m^{-1}$, and metrics with Euclidean volume growth. Consistently with the Taub-NUT construction we say that a metric $g\in\Gaf$ has positive \emph{mass} $m$ in the first case and zero mass in the second case respectively. We prove that for asymptotically flat warped Berger metrics with monotone coefficients the long-time behaviour of the flow only depends on the mass. 

In the following we say that a Ricci flow solution converges to a Ricci-flat metric $g_{\infty}$ on $\R^{4}$ in the pointed Cheeger-Gromov sense as $t\nearrow \infty$ if for any $t_{j}\nearrow \infty$ the sequence $(\R^{4},g_{j}(t),\origin)$, defined by $g_{j}(t) = g(t_{j} + t)$, converges to $(\R^{4},g_{\infty},\origin)$ in the pointed Cheeger-Gromov sense. In particular, there is no rescaling of the solution.
\begin{theorem}\label{maintheoremconvergenceALF}
Let $(\R^{4},g(t))_{t\geq 0}$ be the maximal solution to the Ricci flow starting at some $g_{0}\in\Gaf$. Either one of the following conditions is satisfied:
\begin{itemize}
\setlength\itemsep{0.6em}
\item[(i)] If $g_{0}$ has positive mass $m$, then $g(t)$ encounters a Type-II(b) singularity. Moreover, $g(t)$ converges to the Taub-NUT metric of mass $m$ in the pointed Cheeger-Gromov sense as $t\nearrow \infty$.
\item[(ii)] If $g_{0}$ has zero mass, then the solution encounters a Type-III singularity. In particular, $g(t)$ converges to the Euclidean metric in the pointed Cheeger-Gromov sense as $t\nearrow \infty$.
% converges to the Euclidean metric in the pointed Cheeger-Gromov sense as $t\nearrow \infty$.
\end{itemize}
\end{theorem}
We note that an analogous Type-III result for SO$(n)$-invariant Ricci flows without minimal hyperspheres was obtained in \cite{woolgar}. Moreover, a numerical investigation on the stability of the Taub-NUT metric for warped Berger Ricci flows in $\bar{B}^{4}$ was conducted in \cite{numerics}: Theorem \ref{maintheoremconvergenceALF} and its generalization below provide a rigorous frame for addressing the questions raised in \cite{numerics} on the $\R^{4}$-topology.

In \cite{appleton}, Appleton proved that on $\R^{4}$ there exists a warped Berger gradient steady soliton with monotone coefficients, bounded Hopf-fiber and coefficient $\bb$ in the directions orthogonal to the Hopf-fiber opening as fast as a paraboloid in $\R^{3}$. Namely, the soliton satisfies the asymptotics:
\[
\cc(s) \sim \text{constant}, \,\,\,\,\,\, \bb(s)\sim \sqrt{s}.
\]

Consequently, we cannot expect initial data opening with arbitrary speed to converge to $\nut$ along the flow. The paraboloid growth rate plays a role in \cite{Ivey}, where Ivey found a family of positively curved, pinched SO(3)-invariant immortal Ricci flows on $\R^{3}$ opening (at least) as fast as a paraboloid that do converge in subsequences in the pointed Cheeger-Gromov sense as $t\nearrow\infty$ \cite{Ivey}. Partly motivated by such analysis, we investigate whether a similar convergence property holds for warped Berger Ricci flows opening \emph{faster} than a paraboloid, thus ruling out Appleton's soliton, without restricting to positively curved pinched solutions. With that in mind, we give the following:
\begin{definitionn}\label{definitionGk}
For all $0 \leq k < 1$, the class $\Gk$ consists of all complete warped Berger metrics $g$ with monotone coefficients satisfying:
\begin{itemize}
\item[(i)] $0 < \liminf_{s \rightarrow \infty}\,\bb_{s}\,\bb^{k}(s) \leq \limsup_{s\rightarrow \infty}\,\bb_{s}\,\bb^{k}(s) < \infty$, 
\item[(ii)]$\sup_{p\in \R^{4}}\,\left(\bb^{2}\lvert\text{Rm}_{g}\rvert_{g}\right)(p) < \infty$, 
\item[(iii)] $\sup_{p\in\R^{4}}\cc(p) < \infty$. 
\end{itemize}
\end{definitionn} 
We note that the assumptions in Definition 2 are independent and that properties (i), (ii) cannot be replaced by requiring a suitable rate of curvature decay (a thorough discussion is provided in Section 2.6). In particular, we point out that metrics in $\Gaf$ with positive mass belong to $\mathcal{G}_{0}$. By integrating (i) we see that if $g\in\Gk$, then the warping coefficient $\bb$ grows like $s^{\frac{1}{k+1}}$, meaning that the projection of $g$ on the base space via the Hopf-map opens faster than a paraboloid in $\R^{3}$. The first order constraints in (i) and the decay in (ii) allow us to apply a maximum principle argument to show (we refer to the Outline for more details) that Ricci flow solutions starting in $\Gk$ have a well defined behaviour at spatial infinity on any time-slice, meaning that (i) is preserved - not uniformly though, in fact according to the convergence to Taub-NUT we deduce that there will be a jump in infinite time. 

We still call \emph{mass} the inverse of the size of the Hopf-fiber at spatial infinity. We prove that any maximal Ricci flow solution starting in $\Gk$ develops a Type-II(b) singularity modelled by an ancient solution satisfying the conditions below.
\begin{definitionn}\label{definitionancientpro} Let $m > 0$. The class $\mathcal{A}$ consists of all complete, warped Berger ancient solutions to the Ricci flow on $\R^{4}$ with monotone coefficients and curvature uniformly bounded in the space-time, satisfying
\begin{align*}
\bb_{s}&\geq \frac{f\left(\frac{\bb}{\cc}\right)}{\frac{\bb}{\cc}} \\
\cc &\leq m^{-1}
\end{align*}
%\inf_{\R^{4}\times (-\infty,0]}\,\,\frac{\bb_{s}}{f\left(\frac{\bb}{\cc}\right)}\,\frac{\bb}{\cc} &> 0,  
%\\ \sup_{\R^{4}\times (-\infty,0]}\cc = m, 
\noindent for some continuous positive function $f$ such that $f(z)\rightarrow\infty$ as $z\rightarrow\infty$.
\end{definitionn}
We point out that the class $\mathcal{A}$ describes warped Berger ancient solutions opening faster than a paraboloid in the directions orthogonal to the Hopf-fiber. Our second main result is a rigidity property. %We prove that 
%\\ \emph{The Taub-NUT metric is the only complete warped Berger ancient solution with bounded Hopf-fiber, bounded curvature and opening faster than a paraboloid in the direction orthogonal to the Hopf-fiber.}
%Namely, we show the following: 
\begin{theorem}\label{maintheoremancient}
The only ancient solution in $\mathcal{A}$ is the Taub-NUT metric.%If $(\R^{4},g(t))_{t\leq 0}$ is an ancient solution to the Ricci flow in $\mathcal{A}$, then $(\R^{4},g(t))$ is the Taub-NUT metric.
\end{theorem}
First, we observe that the result is optimal, for the existence of the gradient steady soliton found by Appleton highlights that we cannot drop the requirement on $f$ to diverge in space-time regions where the roundness ratio $\cc/\bb$ becomes degenerate. Moreover, the Euclidean metric would also be included in the class $\mathcal{A}$ if we allowed the size of the Hopf-fiber to be unbounded. 

We emphasize that the rigidity result applies to possible \emph{collapsed} infinite-time singularity models. Indeed, since the Taub-NUT metric is asymptotically flat with bounded Hopf-fiber, we see that for any $\kappa > 0$ there exist $p\in\R^{4}$ and $r > 0$ such that $\nut$ is $\kappa$-strongly collapsed at $p$ for all scales larger than $r$. It is also worth comparing Theorem \ref{maintheoremancient} with a quantization result obtained by Minerbe in \cite{minerbe}, where they proved that a class of hyperk\"ahler 4-manifolds must have cubic volume growth. Our rigidity result may then be interpreted in terms of quantization of the volume growth as well, for in the definition of $\mathcal{A}$ we, \emph{a priori}, allow for ancient solutions with volume growth faster than quadratic.

As a consequence of the previous rigidity result we show the following: 
\begin{theorem}\label{maintheoremconvergence}
Let $(\R^{4},g(t))_{t\geq 0}$ be the maximal solution to the Ricci flow starting at some $g_{0}\in\Gk$ with mass $m > 0$ . Then $g(t)$ converges to the Taub-NUT metric of mass $m$ in the pointed Cheeger-Gromov sense as $t\nearrow \infty$.
\end{theorem}
Again, the result is in some sense optimal because from the existence of the soliton we derive that we cannot extend the convergence to initial data in $\mathcal{G}_{1}$. Theorem \ref{maintheoremconvergence} is not a stability property: metrics in $\Gk$ are, with the exception of a subclass in $\mathcal{G}_{0}$, not asymptotically flat and indeed they have different volume growth and rate of decay of the curvature with respect to the Taub-NUT metric. In fact, we can find initial data with nonnegative sectional curvature flowing to the Ricci-flat Taub-NUT metric. While the fact that positive sectional curvature is not preserved along the flow in dimension higher than three is well known, even in the cohomogeneity-1 setting \cite{bettiol}, in the result below we prove that negative sectional curvature terms not only appear along the solution but also balance out the positive terms to yield a Ricci-flat limit in infinite time.
\begin{corollaryy}\label{maincorollary}
There exists a complete, bounded curvature warped Berger metric $g_{0}$ with $\emph{sec}(g_{0}) \geq 0$ such that the maximal Ricci flow solution starting at $g_{0}$ is immortal and converges to $\nut$ in the pointed Cheeger-Gromov sense as $t\nearrow \infty$.%  satFor all $k\in [0,1)$ there exists $g_{k}\in\Gk$ with $\emph{sec}(g_{k}) > 0$ such that the maximal Ricci flow solution starting at $g_{k}$ converges to the Taub-NUT metric in the pointed Cheeger-Gromov sense as $t\nearrow \infty$.
\end{corollaryy} 
\subsection*{Outline.} In Section 2 we describe the class of initial data and we comment on the assumptions. In particular, we recap a few key properties of $\nut$. In Section 3 we focus on Ricci flows starting in $\Gaf$. In the asymptotically flat setting one can control the solution at spatial infinity in a precise way and hence maximum principle arguments follow. Similarly to other cohomogeneity-1 scenarios \cite{woolgar}, \cite{IKS1}, \cite{appleton2}, \cite{work2}, we prove that the curvature is uniformly controlled whenever the principal orbits are non-degenerate. More importantly, we show that if the Hopf-fiber is bounded, then the solution always opens faster than a paraboloid in $\R^{3}$ in any space-time region where the roundness ratio $\cc/\bb$ gets small. We dedicate Section 4 to extending the analysis for asymptotically flat Ricci flows with positive mass to solutions starting in $\Gk$. In this regard, a few extra-steps are needed to prove that the initial assumptions in the Definition of $\Gk$ do imply that the behaviour of the warping coefficients at spatial infinity along the solution is known: an important ingredient is the preservation of the power law decay of the curvature along Ricci flow solutions derived in \cite{Lott}. Once we can control the solution on the parabolic boundary, we then rely on a maximum principle argument to prove that the \emph{faster than a paraboloid}-growth condition holds uniformly in any space-time region where the roundness ratio $\cc/\bb$ is small. One may then concentrate on compact time-dependent regions where the squashing factor $\cc/\bb$ is non-degenerate and, analogously to the asymptotically flat case, we prove that there are no space-time regions resembling necks in infinite-time: this is the key result to show that the curvature is uniformly bounded in time. In Section 5 we present a compactness result for a class of warped Berger solutions of the Ricci-flow on $\R^{4}$. Such property has an analogous counterpart in \cite{appleton2}, where Appleton formulates the compactness theorem under a different set of assumptions. In particular, they focus on non-collapsed sequences of Ricci flows, being interested in applying the result to the analysis of finite-time singularities. However, in our setting such assumption is not available for we wish to study infinite-time singularity models of (non-rescaled) Ricci-flows. Therefore, we prove that one can still pass to a pointed Cheeger-Gromov limit which not only preserves the symmetries but whose warping coefficients are smooth limits of the warping coefficients along the sequence, provided that the roundness ratio $\cc/\bb$ is non-degenerate at the given origins we center the solutions at. As a first application of the compactness result, we show that the curvature of any Ricci flow solution in $\Gk$ and $\Gaf$ is uniformly bounded in time so that we never need to rescale for obtaining smooth limits at infinite time. Section 6 is devoted to proving that the only complete warped Berger ancient solution with monotone coefficients, bounded curvature, bounded Hopf-fiber and opening faster than a paraboloid along the directions orthogonal to the Hopf-fiber is $\nut$. The argument follows a similar approach used by Appleton in \cite{appleton2} to derive a uniqueness result for the Eguchi-Hanson metric: we rely on the Compactness result in Section 5 to show that relevant geometric quantities always attain their critical values in the space-time, up to passing to a pointed Cheeger-Gromov limit sharing the same features of the given ancient solution. In particular, we prove that one of the hyperk\"ahler first-order quantities which vanishes identically for $\nut$ is always nonnegative on the class of ancient solutions described above: this yields that the ancient solution is Ricci-flat and hence homothetic to $\nut$. We point out that differently from the case discussed by Appleton, we \emph{cannot} use the $\kappa$-non-collapsedness of the ancient solutions, which plays an important role in their analysis. Therefore, one of the main difficulties here consists in ensuring that the roundness ratio $\cc/\bb$ stays positive along any space-time sequence we use to approximate critical values of some given geometric quantity so that the compactness result can indeed be applied. In fact, we know that on the soliton the hyperk\"ahler quantity mentioned before approaches its \emph{negative} infimum in space-time regions where the roundness ratio becomes degenerate. Finally, we rely on the uniqueness result in Section 6 to prove the convergence of immortal Ricci flows in $\Gk$ and $\Gaf$. 
\subsection*{Acknowledgements.} The author would like to thank Jason Lotay for suggesting him to work on this problem as part of his PhD, for the constant support and for many helpful conversations.
\section{Initial data for the Ricci flow}
\subsection{Warped Berger metrics on $\R^{4}$}
Let $(M,g)$ be a non-compact Riemannian manifold and let $\mathsf{G}$ be a compact Lie group acting on $(M,g)$ with cohomogeneity 1. Assume that there exists a singular orbit $\Sigma_{\text{sing}}$, alternatively the orbit space is homeomorphic to $\R$ and $M$ is hence foliated by $\mathsf{G/H}$, with $\mathsf{H}$ the principal isotropy group. Given $q\in\Sigma_{\text{sing}}$ we consider a minimal geodesic $\gamma$ starting at $q$ and meeting all the principal orbits orthogonally. Away from the singular orbit, we can write $g$ along $\gamma$ as
\[
g = ds^{2} + g_{s},
\]
\noindent for some 1-parameter family of homogeneous metrics $g_{s}$ on $\mathsf{G/H}$. We may then use the action to extend such form on any orbit and therefore on the entire principal part of the manifold. 
If there are enough isometries, then the family of metrics $g_{s}$ can be diagonalized along a fixed frame. In this case the diagonal form is preserved along the Ricci flow whenever the solution is unique in the class considered due to the diffeomorphism invariance. Before we concentrate on the cohomogeneity-1 left-action of SU(2) on $\R^{4} = \mathbb{C}^{2}$, we briefly discuss homogeneous metrics on the 3-sphere. We thank Christoph B\"ohm for suggesting the following argument to us.

Once we identify $S^{3}$ with the unit quaternions, we see that $S^{3}\times S^{3}$ acts on $S^{3}$ by conjugation. Consider the finite group $H= \{\pm 1, \pm i, \pm j, \pm k\}$. Then $S^{3}\times H$ still acts on $S^{3}$ with isotropy group at $1$ given by $\mathsf{G}_{1} = \{\mathsf{h} = (h,h)\in H\times H\}$. For any element $\mathsf{h}\in \mathsf{G}_{1}$ the map $d\mathsf{h}:T_{1}S^{3}\rightarrow T_{1}S^{3}$  acts by conjugation on the space of pure imaginary quaternions. In particular, we find that $T_{1}S^{3}$ splits as the direct sum of three inequivalent 1-dimensional representations spanned by $\{i,j,k\}$ respectively, such that none of them is acted on trivially. Since any homogeneous Riemannian metric $g$ on $S^{3}$ must be an $\text{Ad}(\mathsf{G}_{1})$-invariant inner-product on $T_{1}S^{3}$, by Schur's Lemma we deduce that $g$ is diagonal along the frame $\{i,j,k\}$. The same conclusion holds for the Ricci tensor being a bilinear symmetric $\text{Ad}(\mathsf{G}_{1})$-invariant form. We also note that one could derive the orthogonality of $g$ by simply checking that the inner product induced by $g$ on $T_{1}S^{3}$ satisfies
\[
<i,j> = < i\cdot i\cdot \bar{i}, i\cdot j \cdot \bar{i} > = - < i, j >. 
\] 

The same argument can be generalised to the cohomogeneity-1 action of SU(2) on $(\R^{4},g)$, meaning that given a basis $\{I,J,K\}$ in the Lie algebra $\mathfrak{su}(2)$, then the restriction of $g$ to any principal orbit can be diagonalized along the left-invariant extensions of $\{I,J,K\}$. From now on, we denote such extensions by $\{X_{1},X_{2},X_{3}\}$, while we let $\{Y_{1},Y_{2},Y_{3}\}$ be their right-invariant counterparts. Thus, the frame $\{Y_{i}\}$ constitutes a basis of Killing vectors for $(\R^{4},g)$. Thanks to the diffeomorphism between $S^{3}\subset \mathbb{C}^{2}$ and $\text{SU(2)}$, defined in Euler coordinates by 
\[
(e^{i(\theta + \psi)}\cos(\phi),e^{i(\theta-\psi)}\sin(\phi))\mapsto \left[ {\begin{array}{cc}
   e^{i(\theta + \psi)}\cos(\phi) & -e^{-i(\theta-\psi)}\sin(\phi) \\
   e^{i(\theta-\psi)}\sin(\phi) & e^{-i(\theta +\psi)}\cos(\phi) \\
  \end{array} } \right],
\]
\noindent where $\phi\in [0,\pi/2), \psi\in [0,\pi), \theta\in [0,2\pi)$, we may write the left-invariant frame as
\begin{align}\label{leftinvariantframe}
X_{1} &= \sin(2\theta)\partial_{\phi} - \frac{\cos(2\theta)}{\sin(2\phi)}\partial_{\psi} + \cot(\phi)\cos(2\theta)\partial_{\theta}, \notag \\ 
X_{2} &=  \cos(2\theta)\partial_{\phi} + \frac{\sin(2\theta)}{\sin(2\phi)}\partial_{\psi} -\cot(2\phi)\sin(2\theta)\partial_{\theta}, \\ \notag
X_{3} &= \partial_{\theta}.
\end{align}
%use the Maurer-Cartan formalism to write the left-invariant 1-forms $\{\sigma_{i}\}$ dual to $\{X_{1}\}$ as
%\begin{align*}
%\sigma_{1} &= \sin(2\theta)d\phi - \sin(2\phi)\cos(2\theta)d\psi, \\ 
%\sigma_{2} &=  \cos(2\theta)d\phi + \sin(2\phi)\sin(2\theta)d\psi, \\ 
%\sigma_{3} &= \cos(2\phi)d\psi + d\theta,
%\end{align*}
%noindent with dual left-invariant frame 
%\begin{align}\label{leftinvariantframe}
%X_{1} &= \sin(2\theta)\partial_{\phi} - \frac{\cos(2\theta)}{\sin(2\phi)}\partial_{\psi} + \cot(\phi)\cos(2\theta)\partial_{\theta}, \notag \\ 
%X_{2} &=  \cos(2\theta)\partial_{\phi} + \frac{\sin(2\theta)}{\sin(2\phi)}\partial_{\psi} -\cot(2\phi)\sin(2\theta)\partial_{\theta}, \\ \notag
%X_{3} &= \partial_{\theta}.
%\end{align}
\noindent According to the previous discussion, given a metric $g$ invariant under the cohomogeneity-1 action of SU(2) on $\R^{4}$, away from the origin $\origin$ we can represent $g$ by
\begin{align}\label{initialmetric}
\begin{split}
g &= \xi^{2}(x)dx\otimes dx + g_{x} \\
&= \xi^{2}(x)dx\otimes dx + a^{2}(x)\,\sigma_{1}\otimes\sigma_{1} + \bb^{2}(x)\,\sigma_{2}\otimes\sigma_{2} + \cc^{2}(x)\,\sigma_{3}\otimes\sigma_{3},
\end{split}
\end{align}
\noindent where $\xi,a,b,c:(0,+\infty)\rightarrow (0,+\infty)$ are smooth radial functions and $\{\sigma_{i}\}$ is the left-invariant dual coframe induced by $\{X_{i}\}$. If we introduce the geometric quantity $s(\cdot) = d_{g}(\origin,\cdot)$, then we may rewrite \eqref{initialmetric} as 
\begin{equation}\label{initialmetricscoordinate}
g = ds^{2} + a^{2}(s)\,\sigma_{1}\otimes \sigma_{1} +  b^{2}(s)\,\sigma_{2}\otimes \sigma_{2} + \cc^{2}(s)\,\sigma_{3}\otimes\sigma_{3}.
\end{equation}
\noindent If we also assume $g$ to be invariant under the U(1) action on the Hopf-fibres, then the vector field $X_{3}$ is Killing, thus enlarging the Lie algebra of Killing vectors to $\mathfrak{u}(2)$. This is equivalent to requiring $a = \bb$ on $\R^{4}$. Therefore, the Hopf-fibration allows us to write
\begin{equation}\label{fubini}
g = ds^{2} + \bb^{2}(s)\,\pi^{\ast}g_{S^{2}(\frac{1}{2})} + \cc^{2}(s)\,\sigma_{3}\otimes \sigma_{3},
\end{equation}
\noindent where $g_{S^{2}(\frac{1}{2})}$ is the Fubini-Study metric and $\sigma_{3}$ is the one-form dual to the vector field tangent to the Hopf-fibres. In the following we usually refer to the warping coefficient $\bb$ as the coefficient of $g$ along the $S^{2}$-direction. Similarly, we often say that the factor $\cc$ constitutes the size of the Hopf-fiber. According to \eqref{fubini}, an SU(2)U(1)-invariant metric $g$ on $\R^{4}$ is given by the formula:
\begin{equation}\label{metricr3}
g = g_{\R^{3}} + \cc^{2}(s)\,\sigma_{3}\otimes\sigma_{3},
\end{equation}
\noindent where $g_{\R^{3}}$ is the projection of $g$ on the base via the Hopf-fibration $\R^{4}\setminus\{\origin\}\rightarrow \R^{3}\setminus\{\origin\}$. In the analysis below it is important to control how fast the manifold $(\R^{3},g_{\R^{3}})$ opens up. Indeed, in this work we always discuss solutions to the Ricci flow evolving from metrics $g$ of the form \eqref{metricr3} with $g_{\R^{3}}$ opening faster than a paraboloid.
% We thus make the following:
%\begin{definition}
%We say that a metric $g$ on $\R^{4}$ of the form \eqref{metricr3} with $\sup_{\R^{4}}\cc < \infty$, \emph{opens faster than a paraboloid} if 
%\[
%\liminf_{s\rightarrow\infty}\,\bb\bb_{s} = \infty. 
%\]
%\end{definition} 

We finally focus on those SU(2)U(1)-invariant metrics $g$ on $\R^{4}$ satisfying the \emph{warped Berger} condition (see also \cite{IKS1}):
\[
\cc \leq \bb.
\]
\noindent Accordingly, any homogeneous metric $g_{s}$ on the principal orbit $\{s\}\times S^{3}$ is a Berger metric with squashing factor $\cc/\bb \leq 1$. Since such squashing factor plays an important role in the analysis and also appears in the hyperk\"ahler quantities characterizing the Taub-NUT metric, we make the following:
\begin{definition}\label{definitionu}
Given a warped Berger metric $g$, the scale-invariant roundness ratio $\cc/\bb:\R^{4}\rightarrow (0,1]$ is denoted by $\uu$.
\end{definition}

We observe that for any radial map $f$ we think of $f = f(s) = f(s(x))$ as a function of $x$ unless otherwise stated.  From \eqref{initialmetric} and \eqref{initialmetricscoordinate} we have the following relation between the two radial derivatives:
\begin{equation}\label{changevariable}
\partial_{s} = \frac{1}{\xi(x)}\partial_{x}.
\end{equation} 
\noindent The metric $g$ in \eqref{fubini} defines a smooth metric on $\R^{4}$ if and only if $\bb$ and $\cc$ are smooth odd functions of the radial variable $x$ and the condition below holds:
\begin{equation}\label{smoothnessorigin}
\lim_{s\rightarrow 0}\frac{d \bb}{ds}(s)=\lim_{s\rightarrow 0}\frac{d \cc}{ds}(s) = 1. 
\end{equation}  
\noindent It is worth noting that the smoothness conditions reflect the underlying topology and hence lead to significant variations, both in terms of results and approach, when comparing the study of SU(2)U(1)-invariant Ricci flows on different manifolds \cite{IKS1}, \cite{appleton2},\cite{work2}.
%  We note that the underlying topology plays a role in the analysis of the Ricci flow dynamics via the boundary conditions above. 
%\\We point out that, as previously observed, the Lie algebra of Killing vectors for $g$ contains a copy of $\mathfrak{u}(2)$. Indeed, any U(2)-invariant metric on $\R^{4}=\mathbb{C}^{2}$ can be written as in \eqref{initialmetricscoordinate}, up to choosing a suitable Milnor frame (see, e.g., the case $k = 1$ in Section 2.2 of \cite{appleton2}). %represents an arbitrary U(2)-invariant metric on $\R^{4}=\mathbb{C}^{2}$ hence the Ricci flow solutions an

\subsection{Curvature terms.} If $g$ is a warped Berger metric on $\R^{4}$, then from the Koszul formula we derive the vertical sectional curvatures %Given a warped Berger metric $g_{0}$ a simple application of the Koszul formula (see also \cite[Appendix A]{IKS1}) allows to compute the sectional curvatures of the vertical planes
\begin{align}
k_{12} &= \frac{1}{\bb^2}\left(4 -3\uu^{2} -\bb_{s}^{2}\right),\label{sectionalvertical12} \\
k_{13} = k_{23} &= \frac{1}{\bb^{2}}\left(\uu^{2} - \bb_{s}\cc_{s}\uu^{-1}\right), \label{sectionalvertical13} 
\end{align}
\noindent and the mixed sectional curvatures
\begin{align}\label{sectionalhorizontal01}
k_{01} = k_{02} &= -\frac{\bb_{ss}}{\bb}\, ,\\
k_{03} &= -\frac{\cc_{ss}}{\cc}\,. \label{sectionalhorizontal03}
\end{align}
\noindent Moreover, unless the isometry group extends to SO(4), we also have a \emph{non-trivial} curvature term which is not the sectional curvature of a 2-plane:
\begin{equation}\label{rmo123}
\text{Rm}_{0123} = \frac{1}{\bb^{2}}\left(\cc_{s} - \bb_{s}\uu\right) = \frac{\uu_{s}}{\bb}.
\end{equation} We finally report the formula for the scalar curvature:
\begin{equation}\label{scalarcurvature}
R_{g} = 2(k_{01} + k_{02} + k_{03} + k_{12} + k_{13} + k_{23}).
\end{equation}
\subsection{Monotone coefficients.} Since we are interested in studying the long-time behaviour of the Ricci flow, we always consider maximal solutions evolving from warped Berger metrics with coefficients $\bb$ and $\cc$ increasing in space. Namely, we make the following:
\begin{definition}\label{monotonecoefficients}
A warped Berger metric has \emph{monotone coefficients} if 
\begin{equation}\label{monotonecoefficientseq}
\bb_{s} \geq 0,\,\,\, \cc_{s}\geq 0. 
\end{equation}
\end{definition}
The reason we restrict our analysis to this subclass is twofold. We know that there exist spherically symmetric asymptotically flat initial data containing minimal 3-spheres leading to the formation of finite-time Type-I singularities along the Ricci flow \cite{work}. The monotonicity condition is meant to generalise the lack of minimal embedded spheres for the SO($n$)-invariant setting and is hence natural when the emphasis is on investigating the long-time behaviour of the Ricci flow. Indeed, in \cite{work2} we proved that the maximal complete, bounded curvature Ricci flow solution starting at some warped Berger metric with monotone coefficients and curvature decaying at spatial infinity is immortal. In fact, the result holds with assumptions weaker than the spatial monotonicity of both the coefficients $\bb$ and $\cc$. However, the stronger requirement provided in Definition \ref{monotonecoefficients} allows us to control the injectivity radius of the solution only in terms of upper bounds of the curvature. 

Once we know that according to \cite[Theorem 3]{work2} we have a large family of immortal solutions, we wish to determine for which subclass it is possible to classify the infinite-time singularity models. In particular, we aim to identify a class of initial data giving rise to solutions encountering a Type-II(b) singularity at infinite time modelled by the Taub-NUT metric. In order to do that, we first recollect a few properties of the Ricci-flat Taub-NUT metric.
\subsection{The Taub-NUT metric.}
The Taub-NUT metric is a complete gravitational instanton found on $\R^{4}$ by Hawking in \cite{hawking}. Following \cite{jason}, we describe the Taub-NUT metric $\nut$ as the complete, warped Berger metric on $\R^{4}$ of the form \eqref{fubini}, whose warping coefficients $\bb$ and $\cc$ satisfy the differential equations below: 
\begin{equation}\label{ODE1}
\ode \doteq \cc_{s} - \uu^{2} = 0,
\end{equation}
\noindent and
\begin{equation}\label{ODE2}
\oded \doteq \bb_{s} + \uu - 2 = 0.
\end{equation}
\noindent The first-order conditions define a \emph{hyperk\"ahler} structure on $(\R^{4},\nut)$, so that $\nut$ is in particular \emph{Ricci-flat}. One may solve explicitly the equations and write $\nut$ as (see also \cite{jason}):
\begin{equation}\label{explicitnut}
\nut = \frac{1}{16}\left(1 + \frac{2m^{-1}}{x}\right)\,dx^{2} + \frac{x^{2}}{4}\left(1+ \frac{2m^{-1}}{x}\right)\,\pi^{\ast}g_{S^{2}(\frac{1}{2})} + \frac{m^{-2}}{1 + \frac{2m^{-1}}{x}}\,\sigma_{3}\otimes\sigma_{3},
\end{equation}
\noindent where $m$ is a positive parameter quantifying the \emph{mass} of the magnetic monopole giving rise to the Taub-NUT metric \cite{hawking}. Since $m^{-1}$ measures the size of the Hopf-fiber at spatial infinity, we see that $\nut$ has \emph{cubic volume growth}, meaning that there exist $A\geq \alpha > 0$ such that 
\[
\alpha\,r^{3}\leq \text{Vol}_{\nut}\left(B_{\nut}(\origin,r)\right) \leq A\,r^{3},\,\,\,\,\,\, \forall r\geq 1.
\]
\noindent From the formulas of the curvature terms given above we also derive that $\nut$ is an \emph{asymptotically flat} metric satisfying 
\[
\sup_{p\in\R^{4}}\,(d_{\nut}(\origin,p))^{3}\lvert \text{Rm}_{\nut}\rvert_{\nut}(p) < \infty.
\] 
\noindent Since by \eqref{ODE2} the coefficient $\bb$ along the $S^{2}$-direction orthogonal to the Hopf-fiber grows linearly in the distance, we get
\[
\bb^{3}(s)\lvert k_{12}\rvert(s) \geq \delta > 0,
\]
\noindent for all $s\geq 1$. Namely, we find that $\bb^{3}\lvert \text{Rm}_{\nut}\rvert_{\nut}\geq \delta$ away from the unit ball with respect to $\nut$ centred at the origin. By the latter property and the uniform boundedness of the Hopf-fiber we derive that for any $\kappa > 0$ there exist $p\in\R^{4}$ and $r > 0$ such that $\nut$ is strongly $\kappa$-\emph{collapsed} at $p$ for all scales larger than $r$. According to Perelman's analysis, we can rule out the Taub-NUT metric as a possible \emph{finite-time} singularity model for the Ricci flow. One of the main goals of this work consists in showing that $\nut$ can actually appear as an \emph{infinite-time} singularity model for immortal Ricci flow solutions.  
\subsection{Asymptotically flat initial data.}
Since the curvature of the Taub-NUT metric decays at a cubic rate at spatial infinity, it is worth investigating the Ricci flow starting at a warped Berger asymptotically flat metric. It turns out that, as long as we restrict our analysis to asymptotically flat metrics with monotone coefficients, the long-time behaviour of the flow for those initial data can be entirely classified and only depends on the length of the Hopf-fiber at spatial infinity. First, we set the following:
\begin{definition}\label{ALFmetrics} The class $\Gaf$ consists of all complete warped Berger metrics $g$ on $\R^{4}$ with monotone coefficients satisfying 
\begin{equation}\label{epsilondecay}
\sup_{p\in\R^{4}}\,\left(d_{g}(\origin,p)\right)^{2+\epsilon}\lvert\text{Rm}_{g}\rvert_{g}(p) < \infty, 
\end{equation}
\noindent for some $\epsilon > 0$. A metric $g\in\Gaf$ is called \emph{asymptotically flat}.
\end{definition}
% In the following we denote the set of complete asymptotically flat warped Berger metrics on $\R^{4}$ by $\Gaf$. 
% By direct computation one may verify that the curvature of the Taub-NUT metric $g_{TN}^{m}$ decays at a cubic rate at spatial infinity. Accordingly, the class $\Gaf$ contains metrics which have zero asymptotic volume ratio and hence are \emph{not} asymptotically locally Euclidean. 
Below we provide a simple characterization of $\Gaf$. In fact, for the next result we may also drop the assumption on the monotonicity of the warping coefficients.
\begin{lemma}\label{characterizationALF}
Let $g$ be an asymptotically flat warped Berger metric on $\R^{4}$. Then either one of the following conditions is satisfied:
\begin{itemize}
\setlength\itemsep{0.5em}
\item[(i)] There exist the limits
 \[ \lim_{s(p)\rightarrow \infty}\bb_{s}(p) = 2, \,\,\,\,\,\, \lim_{s(p)\rightarrow\infty}\cc_{s}(p) = 0, \,\,\,\,\,\, \lim_{s(p)\rightarrow\infty}\cc(p) \doteq m_{g}^{-1}\in (0,\infty). 
\]
\item[(ii)] There exist the limits
\[
\lim_{s(p)\rightarrow \infty}\bb_{s}(p) = 1, \,\,\,\,\,\, \lim_{s(p)\rightarrow\infty}\cc_{s}(p) = 1.
\]
\end{itemize}

\end{lemma}
\begin{proof}
In the following we always take $s \geq 1$ and we let $\epsilon$ and $\alpha$ be the positive number appearing in \eqref{epsilondecay} and a uniform constant that may change from line to line respectively. We first note that the asymptotic behaviour of the derivatives is a known fact \cite{unnebrink}. Therefore it only remains to show that in case (i) the warping coefficient $\cc$ admits a finite positive limit at spatial infinity. Since $\cc_{s}$ is decaying to zero at spatial infinity there exists $\gamma > 0$ such that $\cc(s) \leq \gamma\,s$. Consider the quantity $\nu = \min\{\epsilon, 3/4\}$. From \eqref{sectionalhorizontal03} and \eqref{epsilondecay} we derive
\[
\lvert s^{1+ \nu}\cc_{ss}\rvert \leq \gamma \lvert s^{2+\nu}\frac{\cc_{ss}}{\cc} \rvert \leq \alpha.
\]
\noindent We can thus apply l'H\^{o}pital formula and conclude that $s^{\frac{2}{3}\nu}\cc_{s}$ is bounded for all $s\geq 1$. It follows that 
\[
\cc(s) \leq \alpha(1 + s^{1-\frac{2}{3}\nu}),
\]
\noindent for all $s\geq 1$. Since $\bb$ grows linearly with respect to the geometric coordinate $s$, we also have
\[
s^{2 + \nu}\frac{\uu^{2}}{\bb^{2}} \leq \alpha s^{-2 + \nu}\cc^{2} \leq \alpha\,s^{-\frac{\nu}{3}}.
\]
\noindent The previous estimate, the condition $\bb_{s}\rightarrow 2$ and formula \eqref{sectionalvertical13} yield 
\[
\left \vert s^{1 + \nu}\,\frac{\cc_{s}}{\cc}\right \vert \leq \alpha.
\]
\noindent By integrating we conclude that there there exist $0 < \delta < M < \infty$ such that 
\[ 
\delta \leq \cc \leq M
\]
\noindent for any $s\geq 1$. The uniform upper bound for $\cc$ and \eqref{sectionalhorizontal03} give $\lvert \cc_{ss}\rvert \leq \alpha s^{-2-\nu}$. Integrating and using that $\cc_{s}\rightarrow 0$ at infinity we obtain $\lvert \cc_{s}\rvert \leq \alpha s^{-1-\nu}$. Therefore $\cc$ admits a limit at infinity, which by the previous analysis needs to be positive and finite.
\end{proof}
\begin{remark}
From the classification result in Lemma \ref{characterizationALF} we see that any metric $g\in\Gaf$ with vanishing asymptotic volume ratio behaves like the Taub-NUT metric at spatial infinity, in the sense that $\text{Vol}_{g}B_{g}(\origin,r) \sim r^{3}$, for $r\geq 1$. In particular, for any such $g$ the Hopf-fiber has a well defined positive and finite length at infinity. In analogy with the magnetic monopole construction of the Taub-NUT metric, we refer to the quantity $(\lim_{s(p)\rightarrow\infty}\cc(p))^{-1} \equiv m_{g}$ as the \emph{mass} of $g$. Accordingly, Lemma \ref{characterizationALF} implies that the class $\Gaf$ is the union of cubic volume growth metrics with bounded Hopf-fiber -i.e. \emph{positive mass} - and of Euclidean volume growth metrics with unbounded Hopf-fiber - i.e. \emph{zero mass}.
% In analogy with the magnetic monopole construction of the Taub-NUT metric, we refer to the inverse of the quantity $\lim_{s(p)\rightarrow\infty}\cc(p) \equiv m_{g}^{-1}$ as to the \emph{mass} of $g$. Consistently, we say that any $g\in\Gaf$ with Euclidean volume growth has zero mass.
\end{remark}
\subsection{Initial data opening faster than a paraboloid.}
%\textbf{write explicitly with equation environment that for any metric in $\Gk$ the curvature decays at some specific rate}
%\textbf{Note that once we have a limsup condition in the definition of the set $\Gk$ then the fact that $\bb_{s}$ is bounded at the initial time is immediate and does not follow from the curvature being controlled by $\bb^{2}$.}
Appleton proved that on $\R^{4}$ there exists a warped Berger gradient steady soliton with monotone coefficients which is characterized by the following asymptotics at spatial infinity \cite{appleton}:
\[
\cc(s) \sim \text{constant}, \,\,\,\,\,\, \bb(s)\sim \sqrt{s}.
\]
\noindent Therefore, the length of the Hopf-fiber approaches a positive finite quantity at spatial infinity, while the projection on the base space $\R^{3}\setminus \{\origin\}$  
\[
g_{\R^{3}}= ds^{2} + \bb^{2}(s)\,\pi^{\ast}g_{S^{2}(\frac{1}{2})}
\]
\noindent opens as fast as a paraboloid on $\R^{3}$. Thus, we derive that initial data opening at spatial infinity with arbitrary speed may fail to converge to the Taub-NUT metric in infinite time, the soliton being an explicit example for that. 

In \cite{Ivey}, Ivey showed that a family of positively curved, pinched SO(3)-invariant immortal Ricci flows on $\R^{3}$ opening (at least) as fast as a paraboloid do converge along subsequences in the pointed Cheeger-Gromov sense as $t\nearrow\infty$ \cite{Ivey}. In line with this result, one is tempted to ask whether an analogous property holds in our setting. Accordingly, we aim to determine whether the soliton provides a sort of lower barrier for the convergence property, in the sense that any solution with bounded Hopf-fiber and warping coefficient along the $S^{2}$-direction growing faster than a paraboloid in $\R^{3}$ does flow to the Taub-NUT metric in infinite time. From a slightly different angle, we  investigate whether the Taub-NUT metric is the only complete, bounded curvature warped Berger ancient Ricci flow with monotone coefficients, bounded Hopf-fiber and opening faster than the soliton.  

By the previous observations we need to characterize the property of a warped Berger metric opening faster than a paraboloid in a way that would be meaningful and hence preserved along a Ricci flow solution. The following definition is equivalent to the one given in the Introduction for the Hopf-fiber is uniformly bounded: we prefer the formulation below because it is invariant under rescaling.
\begin{definition}\label{definitionGk}
For all $0 \leq k < 1$, the class $\Gk$ consists of all complete warped Berger metrics $g$ with monotone coefficients satisfying:
\begin{align}
0 < \liminf_{s \rightarrow \infty}\,(\bb_{s}\uu^{-k})(s) &\leq \limsup_{s\rightarrow \infty}\,(\bb_{s}\uu^{-k})(s) < \infty, \label{openingfasterparaboloid} \\ 
\sup_{p\in \R^{4}}\,\left(\bb^{2}\lvert\text{Rm}_{g}\rvert_{g}\right)(p) & < \infty, \label{decaycurvature1}
\\ \sup_{p\in\R^{4}}\cc(p) &< \infty. \label{hopffiberbounded}
\end{align}
\end{definition} 
Since from \eqref{hopffiberbounded} we see that $\uu^{-1} \sim \bb$ away from the origin, we can integrate \eqref{openingfasterparaboloid} and derive that for any metric $g\in\Gk$ the warping coefficient $\bb$ along the $S^{2}$-direction satisfies $\bb(s) \sim s^{\frac{1}{k+1}}$ for all $s$ large enough, meaning that $g$ opens faster than a paraboloid. In particular, for any warped Berger metric $g\in\Gk$ the volume of geodesic balls of radius $r$ centred at the origin grows as 
\[
\text{Vol}_{g}B_{g}(\origin,r) \sim r^{\frac{2}{k+1} + 1}.
\]
\noindent By combining \eqref{openingfasterparaboloid} and \eqref{decaycurvature1} we find that the curvature of a metric $g\in\Gk$ decays at a rate
\begin{equation}\label{decaycurvaturegk}
\sup_{p\in\R^{4}}\,\left(d_{g}(\origin,p)\right)^{\frac{2}{k+1}}\lvert\text{Rm}_{g}\rvert_{g}(p) < \infty.
\end{equation}
\begin{remark}
It is worth determining which of the previous conditions are related and how.
\begin{itemize}
\setlength\itemsep{0.5em}
\item[(i)] \eqref{openingfasterparaboloid} and \eqref{decaycurvature1} $\Rightarrow$ \eqref{decaycurvaturegk}: this easily follows from integration.
\item[(ii)] \eqref{decaycurvature1} $\nRightarrow$ \eqref{openingfasterparaboloid}: it suffices to consider the following example
\[
g = ds^{2} + \arctan^{2}(s)\left(\sigma_{1}^{2} + \sigma_{2}^{2} + \sigma_{3}^{2}\right) = ds^{2} + \arctan^{2}(s)g_{S^{3}},
\]
\noindent which satisfies the condition in \eqref{decaycurvature1}, yet the metric has cylindrical asymptotics. Indeed, the maximal complete, bounded curvature Ricci flow solution starting at such $g$ encounters a finite-time Type-II singularity \cite[Theorem 1]{work2}.
\item[(iii)] \eqref{decaycurvaturegk} $\nRightarrow$ \eqref{openingfasterparaboloid}: one can take a warped Berger metric with $\cc(s) = \arctan(s)$ and $\bb(s) = s\log(s)$ for all $s\geq 1$ and find that \eqref{decaycurvaturegk} holds with $k = 0$ while the warping coefficient $\bb$ grows faster than a linear function of the distance.
\item[(iv)] \eqref{openingfasterparaboloid} $\nRightarrow$ \eqref{decaycurvature1}: the first-order constraint given by \eqref{openingfasterparaboloid} does not rule out second order terms which are not controlled by the size of the principal orbit $\bb$.
\end{itemize}
\end{remark}

By the existence of the steady soliton found by Appleton we know that Ricci flow solutions starting at initial data as in Definition \ref{definitionGk} with $k = 1$ might in general fail to converge to the Taub-NUT metric. We also note that the Euclidean metric would be included in the class $\mathcal{G}_{0}$ if we dropped the requirement on the size of the Hopf-fiber in \eqref{hopffiberbounded}.

The class of asymptotically flat warped Berger metrics with positive mass - i.e. bounded Hopf-fiber - is contained in $\mathcal{G}_{0}$. The sets $\Gk$ though allow for initial data with geometric features different from the Taub-NUT metric, beyond the rates of both decay of the curvature and growth of the volume of geodesic balls. Indeed, we now describe a metric $g\in\mathcal{G}_{0}$ with nonnegative sectional curvature.
\begin{lemma}\label{existencemetricsec}
There exists $g\in\mathcal{G}_{0}$ satisfying $\emph{sec}(g) \geq 0$.
\end{lemma}
\begin{proof}
Consider the warped Berger metric $g$ on $\R_{+}\times S^{3}$ defined by
\[
g = ds^{2} + s^{2}\,\pi^{\ast}g_{S^{2}(\frac{1}{2})} + \cc^{2}(s)\,\sigma_{3}\otimes \sigma_{3},
\]
\noindent where $c(s) = \int_{0}^{s}\frac{1}{1 + y^{4}}\,dy$. By the smoothness conditions in \eqref{smoothnessorigin} we see that $g$ extends to a complete warped Berger metric on $\R^{4}$ with monotone coefficients. Since $\bb$ is linear and $\cc$ is concave we have $k_{01} = 0$ and $k_{03} \geq 0$. Moreover
\[
k_{12} = \frac{1}{\bb^{2}}\left(4 - 3\uu^{2} - \bb_{s}^{2}\right) = \frac{3}{\bb^{2}}\left(1 - \uu^{2}\right) \geq 0.
\]
\noindent Finally, by direct computation we check that $(c - s(1 + s^{4})^{-1/3})_{s} \geq 0$, hence yielding $k_{13} \geq 0$. Therefore, we have shown that $\text{sec}(g)\geq 0$. In order to prove that $g\in\mathcal{G}_{0}$ it suffices to show that $\bb^{2}\lvert \text{sec}(g)\rvert$ is bounded since \eqref{openingfasterparaboloid} and \eqref{hopffiberbounded} are satisfied with $k = 0$. To this aim, we find that there exists $\alpha > 0$ such that
\begin{align*}
&\bb^{2}(s)\lvert k_{01}\rvert(s) = 0, \\
&\bb^{2}(s)\lvert k_{03}\rvert(s) = s^{2}\left\vert\frac{1}{\cc(s)}\left(-\frac{4s^{3}}{(1 + s^{4})^{2}}\right)\right\vert  \leq \alpha, \\
&\bb^{2}(s)\lvert k_{12} \rvert(s) = 3 \left \vert 1 - \uu^{2}(s) \right \vert \leq \alpha, \\
&\bb^{2}(s)\lvert k_{13}\rvert(s) = \left \vert \uu^{2} - \frac{s}{\cc(s)(1 + s^{4})}\right\vert \leq \alpha. 
\end{align*}
\end{proof}
% the example pages 39/42 of recycled paper of a metric $g_{0}\in\mathcal{G}_{0}$ with $\text{sec}(g_{0}) \geq 0$.
% for any $\Gk$ we can find $g_{k}\in\Gk$ with $\text{sec}(g_{k}) > 0$.  

%We will show below that thanks to a result in \cite{Lott} the power law decay of the curvature, which follows by combining \eqref{openingfasterparaboloid} and \eqref{decaycurvature1}, is preserved along the solution on any time-slice. We also note that 
\subsection{The Ricci flow equations.}
Given a complete, bounded curvature warped Berger metric $g_{0}$, there exists a unique maximal, complete, bounded curvature solution to the Ricci flow $(\R^{4},g(t))_{0\leq t < T}$ starting at $g_{0}$ \cite{shi},\cite{uniqueness}. From now on we omit to specify each time that any Ricci flow solution we consider is meant to be the unique complete, bounded curvature one evolving from some initial metric $g_{0}$.

If $(\R^{4},g(t))_{0\leq t < T}$ is the maximal Ricci flow starting at some complete, bounded curvature warped Berger metric $g_{0}$, then the diffeomorphism invariance and the uniqueness property of the problem in the class of complete, bounded curvature solutions ensure that $g(t)$ is still an SU(2)U(1)-invariant metric for all $t\in [0,T)$. Therefore, we can argue as in Section 2.1 to derive that $g(t)$ can be diagonalized with respect to a time-independent fixed frame. Namely, the solution has the form
\begin{equation}\label{ricciflowsolution}
g(t) = ds\otimes ds + b^{2}(s,t)\,(\sigma_{1}\otimes\sigma_{1} + \sigma_{2}\otimes\sigma_{2}) + \cc^{2}(s,t)\,\sigma_{3}\otimes\sigma_{3},
\end{equation}
\noindent where $s=s(x,t)$ is the distance from the origin with respect to the solution and hence is time-dependent. Such geometric coordinate allows us to write the Ricci flow equations as
\begin{align}\label{Ricciflowpdes1}
\bb_{t} &= \bb_{ss} + \left(\frac{\cc_{s}}{\cc} + \frac{\bb_{s}}{\bb} \right)\bb_{s} + 2\frac{\uu^{2}}{\bb} - \frac{4}{\bb} \\
\cc_{t} &= \cc_{ss} + 2\frac{\bb_{s}\cc_{s}}{\bb} - 2\frac{\uu^{3}}{\bb}. \label{Ricciflowpdes2}
\end{align}
\noindent Since the coordinate $s$ depends on time, there is a non trivial commutator between $\partial_{t}$ and $\partial_{s}$ given by
\begin{equation}\label{commutatorformula}
\left[ \frac{\partial }{\partial t}, \frac{\partial}{\partial s} \right]    = -(\text{ln}(\xi))_{t}\frac{\partial}{\partial s} = - \left(2\frac{\bb_{ss}}{\bb} + \frac{\cc_{ss}}{\cc} \right)\frac{\partial}{\partial s},
\end{equation}
\noindent where $\xi$ is defined as in \eqref{initialmetric}. We finally write down the formula for the time-dependent Laplacian along a solution of the Ricci flow. Given a smooth radial map $f: \R^{4}\rightarrow \R$ we have 
\begin{equation}\label{formulalaplacian}
\Delta f \equiv f_{ss} + \left (2\frac{\bb_{s}}{\bb} + \frac{\cc_{s}}{\cc} \right )f_{s}.
\end{equation}
\begin{remark} Throughout this work we often compute the evolution equation of geometric quantities at stationary points. To this aim, the commutator formula plays an important role. Say that we are interested in studying the sign of $\partial_{t}f$ at a local minimum point, then we report the evolution equation of $f$ at such minimum point after using the conditions $f_{ss} \geq 0$ and $f_{s} = 0$.
\end{remark}
We note that the Ricci flow preserves the class of warped Berger metrics with monotone coefficients. %In particular, if the Hopf-fiber is bounded at the initial time $t = 0$, then it stays bounded as long as the solution exists. 
\begin{lemma}\label{consistencymonotone}
Let $(\R^{4},g(t))_{0\leq t < T}$ be the maximal Ricci flow solution starting at some complete, bounded curvature warped Berger metric $g_{0}$ with monotone coefficients. Then the following properties hold:
\begin{itemize}
\item[(i)] $\uu(\cdot,t) \leq 1$ for all $t\in [0,T)$.
\item[(ii)] $\bb_{s}(\cdot,t) > 0$, $\cc_{s}(\cdot,t) > 0$ for all $t\in (0,T)$.
\item[(iii)] $\uu(\cdot,t) \geq \inf_{\R^{4}}\uu(\cdot,0)$, for all $t\in [0,T)$.%and let $\varepsilon\doteq \inf_{x\geq 0}\cc/\bb(x,0)$. Then for any $(p,t)\in\R^{4}\times [0,T)$ we have
%\[
%\varepsilon \leq \frac{\cc}{\bb}(p,t)\leq 1.$\cc(\cdot,t) \leq \sup_{\R^{4}}\cc(\cdot,0)$, for all $t\in [0,T)$.
%\item[(iii)] For any $T^{\prime} < T$ there exists $\alpha_{T^{\prime}}$ such that
%\[
%\sup_{p\in\R^{4}}\left(s_{0}^{2+\epsilon}\lvert\emph{Rm}_{g(t)}\rvert_{g(t)}(p)\right) \leq \alpha_{T^{\prime}}.
%\]
\end{itemize}
%\noindent In particular $g(t)\in \Gaf$ for all $t\in[0,T)$.
\end{lemma}
\begin{proof}
The proof of (i) and (iii) follows from the same arguments in Lemma 2.6 in \cite{work2}. Similarly, one can easily adapt the proof of Lemma 3.5 in \cite{work2} by replacing the evolution equation of $\cc H$ with that of $\cc_{s}$ to show that (ii) is satisfied as long as the solution exists. %Finally, conditio we note that by \eqref{Ricciflowpdes2} the evolution equation of $\cc$ satisfies:
%%%\]
%\noindent Since $\lvert k_{03} \rvert = \lvert \cc_{ss}/\cc \rvert$ is bounded as long as the solution exists, we see that $\cc$ is exponentially bounded in space and we can hence apply the maximum principle \cite[Theorem 12.14]{ricciflowtechniques2} to deduce (iii).
% 
%\end{lemma} 
%\begin{proof}
%We first note that one can always choose the $g-$distance of hyperspheres from the origin as radial coordinate $x$. Let $\psi:\left [ 0, \infty\right )\rightarrow \mathbb{R}$ be a smooth \emph{non decreasing} function which is sufficiently flat at $0$ (say that $\lim_{x\rightarrow 0}(x^{2}\psi^{(k)}(x)) = 0$ for any $k\in\mathbb{N}$) and is identically 1 for $x\geq x_{0}$ for some sufficiently small $x_{0}$ only depending on $g$. From the boundary conditions \eqref{smoothnessorigin}, the asymptotics \eqref{cubicvolumegrowth} and \eqref{euclideanvolumegrowth}, and the finiteness of $m(g)$ it then follows that the map  $\phi \doteq x\psi + 1$ is a distance-like function satisfying the required lower bound.
%\end{proof}
\end{proof}
\begin{remark}
In the following we always implicitly use that for warped Berger Ricci flows the roundness ratio $\uu$ is bounded by 1.
\end{remark}
From the analysis in \cite{work2} we finally derive that any Ricci flow solution we consider below is in fact immortal.
\begin{corollary}\label{solutionimmortal}
Let $(\R^{4},g(t))_{0\leq t < T}$ be the maximal Ricci flow solution starting at some $g_{0}$ belonging to either $\Gk$ or $\Gaf$. Then the solution is immortal.
\end{corollary}
\begin{proof}
By direct computation one may check that $(\R^{4},g)$ does not contain closed geodesics when $g$ is a warped Berger metric with warping coefficients $\bb$ and $\cc$ strictly increasing in space. Therefore, given a maximal Ricci flow solution as in the statement, by Lemma \ref{consistencymonotone} we see that we can find $t_{0} \in (0,T)$ such that $(\R^{4},g(t_{0}))$ does not contain closed geodesics. Since the curvature is bounded, we deduce that $\text{inj}(g(t_{0})) > 0$. We may then apply \cite[Theorem 3]{work2} to the initial condition $(\R^{4},g(t_{0}))$, being the decay of the curvature preserved along the flow \cite{formationsingularities}, and conclude that the Ricci flow solution exists smoothly for all positive times.%ere are no closed geodesics for all positive times. 

\end{proof}

\section{The Ricci flow in $\Gaf$}
In this section we study the Ricci flow problem in $\Gaf$. According to the characterization of asymptotically flat warped Berger metrics provided in Lemma \ref{characterizationALF}, given $g\in\Gaf$ we refer to the inverse of $\sup_{\R^{4}}\cc$ as the \emph{mass} of $g$. In particular, we recall that the set $\Gaf$ decomposes in the union of metrics with zero mass, or equivalently Euclidean volume growth, and of metrics with positive mass, or equivalently cubic volume growth. The key results of this section consist in showing that for any Ricci flow solution in $\Gaf$ the curvature is controlled by the size of the principal orbits uniformly in time and the spatial derivative $\bb_{s}$ is bounded away from zero in any space-time region where the roundness ratio $\uu$ is not degenerate, once we let the flow start.

We point out that most of the analysis could in fact be performed on a more general level, however we  prefer to focus first on the asymptotically flat case which is easier to deal with - and also includes metrics with vanishing asymptotic volume ratio - before discussing the problem for metrics in $\Gk$, for which the behaviour of the flow at spatial infinity is a priori less rigid.

We first verify that the Ricci flow acts on the class $\Gaf$ preserving the curvature decay of the initial metric \eqref{epsilondecay}. 
\begin{lemma}\label{ALFpreserved}
Let $(\R^{4},g(t))_{t\geq 0}$ be the maximal Ricci flow solution starting at some $g_{0}\in\Gaf$ and let $\epsilon > 0$ be such that $\sup_{\R^{4}}(d_{g_{0}}(\origin,\cdot))^{2 + \epsilon}\lvert\emph{Rm}_{g_{0}}\rvert_{g_{0}}(\cdot) <  \infty$. For any $T^{\prime} < \infty$ there exists $\alpha(T^{\prime})$ such that
\[
\sup_{p\in\R^{4}}\,(d_{g_{0}}(\origin,p))^{2 + \epsilon}\lvert\emph{Rm}_{g(t)}\rvert_{g(t)}(p) \leq \alpha(T^{\prime}),
\]
\noindent for all $t\in [0,T^{\prime}]$.
\end{lemma}
\begin{proof}
We let $s_{0}$ be the geometric coordinate representing the distance function from the origin induced by $g_{0}$ so that we can write the initial metric $g_{0}$ as in \eqref{fubini}. We consider the smooth function $\phi: s_{0}\mapsto \sqrt{s_{0}^{2} + 1}$. From the connection terms, we see that
\[
\lvert \nabla_{g_{0}}\phi\rvert_{g_{0}} = \lvert \partial_{s_{0}}\phi \rvert \leq 1,
\]
\noindent and
\[
\nabla^{2}_{g_{0}}\phi(\partial_{s_{0}},\partial_{s_{0}}) = \partial_{s_{0}}^{2}\phi.
\]
\noindent Moreover, whenever $\bb$ is positive we have
\[
\nabla^{2}_{g_{0}}\phi(X_{1}/\lvert X_{1} \rvert_{g_{0}},X_{1}/\lvert X_{1} \rvert_{g_{0}}) =  \frac{\bb_{s_{0}}}{\bb}\frac{s_{0}}{\sqrt{1 + s_{0}^{2}}}.
\]
\noindent From \cite{work2}[Corollary 3.2] we derive that the previous quantity is bounded away from the origin. Since by the boundary conditions $\bb_{s_{0}}(\origin) = 1$ we may conclude that the bound extends at the origin as well. Similar arguments work when evaluating the Hessian of $\phi$ along $X_{3}$. Therefore we have just shown that $\phi$ is a smooth distance-like function on $(\R^{4},g_{0})$ in the sense of \cite{ricciflowtechniques2}[Lemma 12.30]. Namely, there exists $\alpha > 0$ such that
\begin{align*}
\begin{split}
\alpha^{-1}(s_{0}(p) + 1) &\leq \phi(p) \leq \alpha(s_{0}(p) + 1), \\
\lvert \nabla_{g_{0}}\phi \rvert &\leq \alpha, \\
\nabla_{g_{0}}^{2}(\phi) &\leq \alpha g_{0}.
\end{split}
\end{align*}
\noindent Since $\bb_{s_{0}}\geq 0$ and $\cc_{s_{0}}\geq 0$, the same computations yield
\[
\nabla_{g_{0}}^{2}\phi \geq 0.
\]
\noindent We may finally apply Proposition B.10 in \cite{Lott} to our setting, thus proving that the power law decay of the curvature in \eqref{epsilondecay} persists along the flow. 
\end{proof}
%\begin{lemma}\label{distance-likenice}
%Any asymptotically flat $g\in\mathfrak{WB}(\mathbb{R}^{4})$ admits a distance-like function $\phi$ satisfying 
%\[
%\phi^{-1}\emph{Hess}_{g}(\phi) \geq \gamma g,
%\]
%\noindent for some $\gamma \in \mathbb{R}$.
%Given $g_{0}\in\Gaf$, by Lemma \ref{characterizationALF} and the spatial monotonicity of $\cc(\cdot,0)$ it follows that 
%\[
%\lim_{s_{0}(p)\rightarrow \infty}\cc(p,0) \equiv \sup_{p\in\R^{4}}\cc(p,0) = m_{g_{0}}^{-1} > 0.
%\]
A simple consequence of the power law decay being preserved along the Ricci flow is the \emph{conservation of mass} along the flow.
\begin{corollary}\label{massconserved}
Let $(\R^{4},g(t))_{t\geq 0}$ be the maximal Ricci flow solution starting at some $g_{0}\in\Gaf$ with positive mass $m_{g_{0}}$. Then $m_{g(t)} = m_{g_{0}}$ for any $t\geq 0$.
\end{corollary}
\begin{proof} 
According to Corollary \ref{solutionimmortal} we define $\alpha:[0,\infty) \rightarrow \R$ by $\alpha(t) = \sup_{\R^{4}}\lvert \text{Rm}\rvert(\cdot,t)$. Given $t > 0$, since by Lemma \ref{consistencymonotone} we know that $\cc_{s}$ stays nonnegative along the flow, we deduce that the quantity $m_{g(t)}^{-1} := \lim_{s(t)\rightarrow \infty}\cc(s(t),t)$ exists and is well defined. Moreover, the curvature is uniformly bounded in $\R^{4}\times [0,t]$ by $\alpha(t)$ and we can then rely on standard distortion estimates of the distance function along bounded curvature Ricci flow to derive 
\[
m_{g(t)}^{-1} = \lim_{s(t)\rightarrow \infty}\cc(s(t),t) = \lim_{s_{0}(p)\rightarrow \infty}\cc(p,t),
\]
\noindent where $s_{0}$ is the $g_{0}$-distance from the origin. Suppose for a contradiction that there exists $t_{1} > 0$ such that $m_{g(t_{1})} \neq m_{g_{0}}$. By the Ricci flow equations and Lemma \ref{ALFpreserved} we obtain
\[
\left \vert \partial_{t}\log(\cc(p,\cdot))\right\vert \leq \frac{\alpha(t_{1})}{s_{0}(p)^{2+\epsilon}}
\]
\noindent in $\R^{4}\times [0,t_{1}]$ up to modifying $\alpha(t_{1})$ by a constant factor. Once we integrate we find
\[
\frac{1}{t_{1}}\left\vert \log\left(\frac{\cc(p,t_{1})}{\cc(p,0)}\right)\right\vert \leq \frac{\alpha(t_{1})}{s^{2+\epsilon}_{0}(p)}.
\]
\noindent We may finally let $s_{0}(p)\rightarrow \infty$ and get a contradiction.
\end{proof}
\begin{remark} A consequence of Corollary \ref{massconserved} is given by the fact that the size of the Hopf-fiber stays uniformly bounded in the positive-mass case. 
\end{remark}
We may now focus on first order estimates. From the Ricci flow equations \eqref{Ricciflowpdes1}, \eqref{Ricciflowpdes2} and the commutator formula \eqref{commutatorformula} we derive the evolution equations for the first spatial derivatives:
\begin{equation}\label{equationbs}
\partial_{t}\bb_{s}= \Delta \bb_{s} - 2\frac{\bb_{s}\bb_{ss}}{\bb} + \frac{1}{\bb^{2}}\left (\bb_{s}\left(4 - \bb_{s}^{2} - (\cc_{s}\uu^{-1})^{2} -6\uu^{2}\right ) + 4\cc_{s}\uu\right)
\end{equation}
\noindent and
\begin{equation}\label{equationcs}
\partial_{t}\cc_{s} = \Delta \cc_{s} - 2\frac{\cc_{s}\cc_{ss}}{\cc} + \frac{1}{\bb^{2}}\left(\cc_{s}\left( -6\uu^{2} - 2\bb_{s}^{2}\right ) + 8\bb_{s}\uu^{3}\right),
\end{equation}
\noindent where the Laplacian formula is given in \eqref{formulalaplacian}.
\noindent Similarly to the analysis in the finite-time case \cite{work2}, we show that the derivatives stay uniformly bounded and that the flow becomes rotationally symmetric in space-time regions where the orbits get degenerate. 
\\In the following estimates $\alpha$ always denotes a uniform, space-time independent constant that may change from line to line, unless otherwise stated.
\begin{lemma}\label{firstorderestimates}
If $(\R^{4},g(t))_{t\geq 0}$ is the maximal Ricci flow solution evolving from some $g_{0}\in\Gaf$, then the following conditions hold:% There exists $\alpha > 0$ such that the following conditions are satisfied uniformly in the space-time:
%\begin{itemize}
%\item[(i)] $g_{0}$ is asymptotically flat.
%\item[(ii)] $\cc_{s}(\cdot,0) \geq 0$.
%\item[(iii)] $\bb_{s}(\cdot,0)\leq 2$.
%\end{itemize}
%\noindent Then the following estimates are uniformly satisfied in $\R^{4}\times [0,\infty)$:
\begin{align}
 \sup_{\R^{4}\times [0,+\infty)}&\,\,(\bb_{s} + \cc_{s}) < \infty, \label{firstderivativesbounded} \\
\sup_{\R^{4}\setminus\{\origin\}\times [0,+\infty)}&\,\,\frac{1}{\cc}(\bb_{s}^{2} - 4) < \infty, \label{bbsboundedby4} \\
\sup_{\R^{4}\times [0,+\infty)}&\,\,\frac{1}{\cc}\left(1 - u\right) < \infty, \label{rotationalsymmetryorder0} \\
\sup_{\R^{4}\times [0,+\infty)}&\,\,\frac{1}{\cc}\left\vert\cc_{s} - \bb_{s}\uu\right\vert < \infty.  \label{rotationalsymmetryorder1}
\end{align}
\end{lemma}
\begin{proof} 
\emph{Estimate \eqref{firstderivativesbounded}.} Consider the upper bound for $\bb_{s}$. Since $\bb_{s}(\origin,t) = 1$ and by Lemmas \ref{characterizationALF} and \ref{ALFpreserved} we see that $\bb_{s}(s,t)$ converges to either 1 or 2 at infinity for any $t\geq 0$, we deduce that if $\bb_{s}$ attains a value $\bar{\alpha} > \sup \bb_{s}(\cdot,0)$, then there exists a maximum point $(p_{0},t_{0})$ among prior times where $\bb_{s}(p_{0},t_{0}) = \bar{\alpha}$ for the first time. We can then argue as in \cite[Lemma 4.3]{work2}. The same argument works for $\cc_{s}$ as well. %From the evolution equation \eqref{equationbs} we get
%\[
%(\bb_{s})_{t}(p_{0},t_{0}) \leq \frac{1}{\bb^{2}}\left( 4\bar{\alpha} - \bar{\alpha}^{3} - \bar{\alpha}\frac{\cc_{s}^{2}}{\uu^{2}} - 6\bar{\alpha}\uu^{2} + 4\uu\cc_{s}\right).
%\]
%\noindent The right hand side is always negative for any $\bar{\alpha} > 2$. Therefore $\bb_{s}$ is uniformly bounded from above. We also note that the same argument implies that if $\bb_{s}(\cdot,0) \leq 2$ then $\bb_{s}(\cdot,t)\leq 2$. The same argument works for the upper bound of $\cc_{s}$. The uniform lower bounds follow immediately from the monotonicity condition being preserved according to Lemma \ref{ALFpreserved}.
% the lower bounds actually follow from the monotonicity assumption.

\emph{Estimate \eqref{bbsboundedby4}.} Set $\varphi \doteq \cc^{-1}(\bb_{s}^{2} - 4)$. From the boundary conditions we see that $\varphi$ diverges to minus infinity at the origin uniformly in time. On the other hand, according to Lemmas \ref{characterizationALF} and \ref{ALFpreserved} we also have that $\varphi\rightarrow 0$ at spatial infinity as long as the solution exists. At any positive interior maximum point $(p_{0},t_{0})$ we have
\begin{align*}
\varphi_{t}(p_{0},t_{0}) &\leq \frac{1}{\bb^{2}\cc}\left(-(\cc_{s}\uu^{-1})^{2}\left(\bb_{s}^{2} + 4\right) + \bb_{s}\cc_{s}\left(8\uu^{-1} + 8\uu - 2\bb_{s}^{2}\uu^{-1}\right)\right) \\ &+\frac{1}{\bb^{2}\cc}\left(8\bb_{s}^{2} - 2\bb_{s}^{4} - 10\uu^{2}\bb_{s}^{2} - 8\uu^{2}\right).
\end{align*}
\noindent Since $\bb_{s}^{2} > 4$ at any positive value of $\varphi$, we get
\[
\varphi_{t}(p_{0},t_{0}) \leq \frac{1}{\bb^{2}\cc}\left(-(\cc_{s}\uu^{-1})^{2}\left(\bb_{s}^{2} + 4\right) + 8\uu\bb_{s}\cc_{s} - 10\uu^{2}\bb_{s}^{2} - 8\uu^{2}\right).
\] 
\noindent Being the $\cc_{s}$-quadratic above always negative, we conclude that $\varphi$ is uniformly bounded \emph{from above} in the space-time.

\emph{Estimate \eqref{rotationalsymmetryorder0}} We first prove that $\psi\doteq \cc^{-1/2} - \bb^{-1/2}$ is uniformly bounded in the space-time. Since the curvature is bounded, by \eqref{rmo123} we see that $\bb^{-1}\uu_{s} = O(1)$ as $s\rightarrow 0$ uniformly in time. Therefore $\psi(\origin,t) = 0$ as long as the solution exists. Moreover, from Lemma \ref{ALFpreserved} we also derive that $\psi$ is uniformly bounded at spatial infinity by the inverse of the size of the Hopf-fiber. At any interior maximum point $(p_{0},t_{0})$ we have
\begin{align*}
\psi_{t}(p_{0},t_{0}) &\leq \frac{1}{\bb^{\frac{5}{2}}}\left(\frac{\bb_{s}^{2}}{4}\left(1 - \sqrt{\uu}\right) + \uu^{2} + \uu^{\frac{3}{2}} - 2\right) \\
&\leq \frac{1}{\bb^{\frac{5}{2}}}\left(\frac{\bb_{s}^{2}}{4}\left(1 - \sqrt{\uu}\right) -2\left(1 - \sqrt{\uu}\right)\right) \\ &\leq \frac{1}{\bb^{\frac{5}{2}}}\left(1-\sqrt{\uu}\right)\left(\alpha\cc - 1\right),
\end{align*}
\noindent where $\alpha > 0$ is a uniform constant given by the estimate \eqref{bbsboundedby4}. Say that $\psi(p_{0},t_{0}) = M$. By choosing $M$ large enough we can make $\cc$ as small as we ask. Therefore, the right hand side of the evolution equation becomes strictly negative, hence showing that $\psi$ is uniformly bounded in the space-time. We may now consider $f\doteq \cc^{-1}(1 - u)$. Similarly to the case of $\psi$ above, $f$ is uniformly bounded both at the origin and at spatial infinity. At any maximum point we have (see also \cite{work2}[Lemma 4.5])
\begin{align*}
f_{t} &\leq \frac{1}{\bb^{3}}\left(\bb_{s}^{2}\left(1- \uu\right)+ 2\uu + 2\uu^{2} - 4\right)
\\ &\leq \frac{1}{\bb^{3}}\left((4 + \alpha c)\left(1- \uu\right) + 2\uu + 2\uu^{2} - 4\right) \\ &\leq \frac{1}{\bb^{3}}\left(-2\uu + 2\uu^{2} + \alpha \cc\left( 1 - \uu\right)\right) = \frac{\cc f}{\bb^{3}}\left(-2\uu + \alpha \cc \right).
\end{align*}
\noindent We have shown that $\uu^{1/2} \geq 1 - \alpha \cc^{1/2}$. Therefore, if we pick the value attained by $f$ large enough, we see that $(-2\uu + \alpha\cc)(p_{0},t_{0}) \leq -1$. That completes the proof. 
%\emph{We point out that if we further assume $\bb_{s}(\cdot,0)\leq 2$, which is then preserved along the flow, then the rotational symmetry type of estimate follows without relying on any further bound.}

\emph{Estimate \eqref{rotationalsymmetryorder1}} Again $\cc_{s}/\cc - \bb_{s}/\bb$ is uniformly bounded at the origin and at spatial infinity. Once the quantity is controlled along the parabolic boundary of the space-time, one can then argue as in \cite[Lemma 4.8]{work2}.
%  At any negative minimum point the evolution equation becomes
%\begin{align*}
%h_{t}|_{\text{min} < 0} &\geq \frac{1}{\bb^{2}}\left( \lvert h \rvert \left(\frac{\cc_{s}^{2}}{\uu^{2}} + 8\uu^{2} + 2\bb_{s}^{2}\right) - 16\frac{\bb_{s}}{\bb}\left( 1 - \uu\right)\right) \\ &\geq \frac{2}{\bb^{2}}\left(\lvert h \rvert(4\uu^{2} + \bb_{s}^{2}) - 8\alpha\uu\bb_{s}\right) \\ &\geq \frac{2}{\bb^{2}}\left(\lvert h \rvert(4\uu^{2} + \bb_{s}^{2}) - 4\alpha(\uu^{2} + \bb^{2}_{s}\right) > 0
%\end{align*}
%\noindent for $\lvert h\rvert$ large enough. We note that we have used the estimate \eqref{rotationalsymmetryorder0} in the second line. Similarly one can derive a uniform upper bound for $h$.
\end{proof}
Before we prove analogous second order estimates, we first show that in the cubic volume growth case the spatial derivative $\cc_{s}$ decays at some specific rate in space-time regions where $\uu$ is small. We recall that for the Taub-NUT metric we have $\cc_{s} = \uu^{2}$.
%We first note that the argument above shows that if $(\R^{4},g(t))$ is the maximal Ricci flow solution starting at some $g_{0}\in$ then $g(t)\in$ for any $t\geq 0$. We then prove the following preliminary lemma; we note that for the next bound to be satisfied we crucially\footnote{One might suspect that if we only assume $\lvert\text{Rm}_{g_{0}}\rvert_{g_{0}} = \mathcal{O}(s^{-3})$ then $\bb_{s}$ attains a maximum value larger than 2 at some radial coordinate that approaches infinity as $t\nearrow \infty$ so that at spatial infinity $\bb_{s}$ always converges to 2 from above, hence there is no hope for getting smooth convergence to the Taub-NUT metric.} use both the fact that the curvature of the initial data decays to zero at least at a cubic rate and that $\bb(\cdot,t)\leq 2$.
%Before we prove analogous second order estimates, we first check that we can bound the quantity $\uu^{-k}$, for some $k > 1$, by the spatial derivative $\cc_{s}$. This will allow us to control the curvature in space-time regions where the ratio $\uu$ is small. 
\begin{lemma}\label{bbccsquaredccsboundedbyk}
Let $(\R^{4},g(t))_{t\geq 0}$ be the maximal Ricci flow starting at some $g_{0}\in\Gaf$ and let $\epsilon > 0$ satisfy $\sup_{\R^{4}} (d_{g_{0}}(\origin,\cdot))^{2+\epsilon}\lvert\emph{Rm}_{g_{0}}\rvert_{g_{0}}(\cdot) <  \infty$. For any $1 < k < \min\{1+\epsilon,\sqrt{2}\}$ there exists $\alpha> 0$ independent of $k$ such that  
\[
\sup_{\R^{4}\times [0,\infty)}\cc_{s}\uu^{-k} <  \alpha.
\]
\end{lemma}
\begin{proof}
From \eqref{sectionalvertical13} and Lemma \ref{ALFpreserved} we derive that for any $t\geq 0$ there exists $\alpha = \alpha(t)$ such that 
\[
\bb^{2+\epsilon}\left \vert \frac{\uu^{2}}{\bb^{2}} - \frac{\bb_{s}\cc_{s}}{\bb\cc}\right \vert \leq \alpha.
\]
\noindent Since $\bb$ is linear at infinity, we get that for $s$ large
\[
\lvert\cc_{s}\uu^{-1}\rvert\bb^{\epsilon} \leq \alpha + \uu^{2}\bb^{\epsilon}.
\]
\noindent Therefore, for any $1 < k < \min\{1+\epsilon,\sqrt{2}\}$ we have
\[
\cc_{s}\uu^{-k} \leq  \frac{1}{\cc^{\epsilon}}\left(\alpha + \uu^{2}\bb^{\epsilon}\right)\uu^{1 + \epsilon - k} \leq \frac{\alpha}{\cc^{\epsilon}}u^{1 + \epsilon - k} + \uu^{3-k},
\]
\noindent which is uniformly bounded at spatial infinity by Lemma \ref{ALFpreserved}. In particular, we see that $(\cc_{s}\uu^{-k})(s,t)$ either converges to zero (in the cubic volume growth case) or to 1 (in the Euclidean volume growth case) for any $1 < k < \min\{1+\epsilon,\sqrt{2}\}$. By the boundary conditions we derive that if $\cc_{s}\uu^{-k}$ becomes unbounded as $t\nearrow \infty$ then there exists a sequence of maxima diverging. The evolution equation of $\cc_{s}\uu^{-k}$ at a maximum point is  
\begin{align*}
\partial_{t}\left(\cc_{s}\uu^{-k}\right)|_{\text{max}} &\leq \frac{\cc_{s}\uu^{-k}}{\bb^{2}}\left( \bb_{s}^{2}(k^{2}-2) -4k + \uu^{2}(-6+4k) + (\cc_{s}\uu^{-1})^{2}(k^{2} -2k)\right)\\ &+\frac{1}{\bb^{2}}\left(\cc_{s}\uu^{-k}\left(\bb_{s}\cc_{s}\uu^{-1}(-2k^{2} + 2k)\right) + 8\uu^{3-k}\bb_{s}\right)\\ &\leq \frac{1}{\bb^{2}}\left( -4k(\cc_{s}\uu^{-k}) + \alpha \right) < \frac{1}{\bb^{2}}\left( -4(\cc_{s}\uu^{-k}) + \alpha \right),
\end{align*}
\noindent where we have used that $1 < k < \sqrt{2}$ and the estimate \eqref{firstderivativesbounded}. We conclude that for any $k\in (1, \min\{1 + \epsilon, \sqrt{2}\})$ the function $\cc_{s}\uu^{-k}$ admits a uniform upper bound independent of $k$.

%Given $k\in [1,2)$ we have that $4(k+1)(k-2) = -\delta_{k}$; therefore, by picking $M_{k}$ sufficiently large we obtain $(\bb/\cc)^{k} = M_{k}\cc_{s}^{-1} \geq \alpha M_{k}$ being $\cc_{s}$ bounded from above; therefore 
%\[
%%\]
%\noindent for $M_{k}$ large enough.
\end{proof}
We may now show that the mixed sectional curvatures are controlled in space-time regions where $\cc$ stays positive. One can compute that
\begin{align}\label{evolutionk01}
\begin{split}
(k_{01})_{t} &= \Delta k_{01} + 2k_{01}^{2} + k_{01}\left (\frac{8}{\bb^{2}} -\frac{8\uu^{2}}{\bb^{2}} - \frac{2\cc_{s}^{2}}{\cc^{2}} -\frac{4\bb_{s}^{2}}{\bb^{2}} \right ) + k_{03}\left (\frac{4\uu^{2}}{\bb^{2}} - \frac{2\bb_{s}\cc_{s}}{\bb\cc}\right) \\  &-\frac{4\cc_{s}^{2}}{\bb^{4}} + \frac{24\bb_{s}\cc_{s}\uu}{\bb^{4}} -\frac{2\bb_{s}\cc_{s}^{3}}{\bb\cc^{3}} -\frac{24\bb_{s}^{2}\uu^{2}}{\bb^{4}} + \frac{8\bb_{s}^{2}}{\bb^{4}} -\frac{2\bb_{s}^{4}}{\bb^{4}}
\end{split}
\end{align}
\noindent and
\begin{equation}\label{evolutionk03}
\begin{split}
(k_{03})_{t} &= \Delta k_{03} + 2k_{03}^{2} - 4k_{03}\left (\frac{\bb_{s}^{2}}{\bb^{2}} +\frac{\uu^{2}}{\bb^{2}}\right ) + 4k_{01}\left (\frac{2\uu^{2}}{\bb^{2}} - \frac{\bb_{s}\cc_{s}}{\bb\cc}\right) \\  & + \frac{12\cc_{s}^{2}}{\bb^{4}} +\frac{40\bb_{s}^{2}\uu^{2}}{\bb^{4}} -\frac{48\bb_{s}\cc_{s}\uu}{\bb^{4}} - \frac{4\bb_{s}^{3}\cc_{s}}{\bb^{3}\cc}.
\end{split}
\end{equation}
\noindent Once we control the ratio $\uu$ from below by $\cc_{s}$, we can prove the following:
\begin{lemma}\label{secondderivativeslemma}
If $(\R^{4},g(t))_{t\geq 0}$ is the maximal Ricci flow solution starting at some $g_{0}\in\Gaf$, then
\[
\sup_{\R^{4}\times [0, + \infty)}\,\cc^{2}\left(\lvert k_{01}\rvert + \lvert k_{03}\rvert\right) <  \infty.
\]
% and let $\epsilon > 0$ satisfy $\sup_{\R^{4}}(s_{0}^{2+\epsilon}\lvert\emph{Rm}_{g_{0}}\rvert_{g_{0}})\leq \alpha$. 
\end{lemma}
\begin{proof}
First, we prove that $-\bb\cc k_{01} = \cc\bb_{ss}$ has a uniform lower bound in the space-time. In analogy with \cite{IKS2}, we consider the quantity $f\doteq \cc\bb_{ss} - 2\bb_{s}^{2} - \cc_{s}^{2}$ which we see to be uniformly bounded at the origin and at spatial infinity by the boundary conditions and Lemma \ref{ALFpreserved} respectively. A long yet straightforward computation yields that whenever $f$ attains some negative minimum value its evolution equation becomes
\begin{align*}
\partial_{t}f|_{f_{\text{min}} < 0} &\geq 2\bb^{2}_{ss}(2 - \uu) + 2\cc_{ss}^{2} + \cc_{ss}\left(4\frac{\uu^{2}}{\bb} - 2\frac{\bb_{s}\cc_{s}}{\cc} - 4\frac{\bb_{s}\cc_{s}}{\bb}\right) + \\ &+ \frac{\uu\bb_{ss}}{\bb}\left(4 - 3\bb_{s}^{2} - 8\uu^{2} - \cc_{s}^{2} + 2\bb_{s}\cc_{s}\uu^{-1}\left(1 -4\uu^{-1}\right)\right) 
\\ &+ \frac{1}{\cc^{2}}\left(2\bb_{s}\cc_{s}^{3} + 4\cc_{s}^{2}\uu^{3}(1 + 3\uu) + 4\bb_{s}^{2}\cc_{s}^{2}(1 + \uu^{2}) - 24\bb_{s}\cc_{s}\uu^{4} + 2\bb_{s}^{4}\uu^{2}(2 + \uu)\right) \\ &+  \frac{1}{\cc^{2}}\left(24\bb_{s}^{2}\uu^{4}(1 + \uu) - 8\bb_{s}^{2}\uu^{2}(2 + \uu) -8\bb_{s}\cc_{s}\uu^{3}(2 + 3\uu)\right).
\end{align*}
\noindent Since the first order spatial derivatives are uniformly bounded we may assume that $\cc\bb_{ss} < f_{\text{min}}/2$ provided that $\lvert f_{\text{min}} \rvert$ is large enough. By applying Cauchy-Schwarz to the coefficients of $\cc_{ss}$ and again using \eqref{firstderivativesbounded} we get
\[
\partial_{t}f|_{f_{\text{min}} < 0} \geq 2\bb^{2}_{ss} + \frac{\cc_{ss}^{2}}{2} + \frac{\uu\bb_{ss}}{\bb}\left(4 - 3\bb_{s}^{2} - 8\uu^{2} - \cc_{s}^{2} + 2\bb_{s}\cc_{s}\uu^{-1}\left(1 -4\uu^{-1}\right)\right)  - \frac{\alpha}{\cc^{2}},
\]
\noindent for some uniform constant $\alpha > 0$. By the monotonicity of $\bb_{s}$ and $\cc_{s}$ we finally obtain
\[
\partial_{t}f|_{f_{\text{min}} < 0} \geq \frac{1}{\cc^{2}}\left( 2(\bb_{ss}\cc)^{2} + 4\uu^{2}\bb_{ss}\cc - \alpha\right) \geq \frac{1}{\cc^{2}}\left( \frac{f_{\text{min}}^{2}}{2} + 4f_{\text{min}} - \alpha\right) > 0,
\]
\noindent for $\lvert f_{\text{min}} \rvert$ large enough. The existence of a uniform upper bound for $\cc\bb_{ss}$ follows from the similar arguments.
\\We now show that $-\cc^{2} k_{03} = \cc\cc_{ss}$ has a uniform lower bound as long as the solution exists. We proceed as before. We define $h = \cc\cc_{ss} - 2\cc_{s}^{2} -\bb_{s}^{2}$, which by the boundary conditions and Lemma \ref{ALFpreserved} is uniformly bounded at the origin and at spatial infinity. Suppose that $h$ attains a negative minimum. According to \eqref{firstderivativesbounded} we find that $\cc\cc_{ss}\leq h_{\text{min}}/2$, whenever $\lvert h_{\text{min}} \rvert$ is sufficiently large. At such point we can write the evolution equation of $h$ as
\[
\partial_{t}h|_{h_{\text{min}} < 0}\geq \frac{1}{\cc^{2}}\left(2(\cc\cc_{ss})^{2} + \alpha\cc\cc_{ss} + 2\bb_{ss}^{2}\cc^{2} - \alpha\lvert\bb_{ss}\cc\rvert -\alpha\right), 
\]
\noindent where $\alpha$ is a uniform positive constant given by $\bb_{s}$ and $\cc_{s}$ being positive and bounded along the flow - and by $\uu$ being bounded by 1. Since we have just checked that $\lvert \cc\bb_{ss}\rvert$ is uniformly bounded, the right hand side is positive once we pick $\lvert h_{\text{min}}\rvert$ and hence $\lvert \cc\cc_{ss}\rvert$ large enough. Analogously one may check that $\cc\cc_{ss}$ is uniformly bounded from above in the space-time.
\end{proof}
From the estimates \eqref{firstderivativesbounded} and \eqref{rotationalsymmetryorder0} and the previous Lemma we deduce that the curvature is uniformly controlled in time in regions where the orbits do not become degenerate. 
\begin{corollary}\label{corollarycccurvature}
If $(\R^{4},g(t))_{t\geq 0}$ is the maximal Ricci flow solution starting at some $g_{0}\in\Gaf$, then 
\[
\sup_{\R^{4}\times [0, + \infty)}\,\left(\cc^{2}\lvert\emph{Rm}_{g(t)}\rvert_{g(t)}\right) < \infty.
\]
\end{corollary}
We dedicate the end of this section to proving that the spatial derivative $\bb_{s}$ has a uniform lower bound in the space-time domain where the squashing factor $\uu$ stays positive. %We first focus on the positive-mass case. Below we will then check that the very same arguments work for the zero-mass case as well. 
We start by showing that in the bounded Hopf-fiber setting $\bb_{s}\uu^{-1}$ always diverges when $\uu^{-1}$, and hence $\bb$ by \eqref{rotationalsymmetryorder0}, is large. This estimate will play a key role in characterizing the possible infinite-time singularity models and turns out to be satisfied by Ricci flows in $\Gk$ as well.% cannot approach zero in a space-time region where $\bb$ is large. The following also turns out to be a useful estimate when proving convergence to the Taub-NUT metric.
\begin{lemma}\label{keyquantity}
Let $(\R^{4},g(t))_{t\geq 0}$ be the maximal Ricci flow solution starting at some $g_{0}\in\Gaf$ with bounded Hopf-fiber. There exist $\alpha, \lambda > 0$ such that
\[
\bb^{\lambda}(\bb_{s}\uu^{-1} - \log(\bb))\geq -\alpha,
\]
\noindent uniformly in the space-time.
\end{lemma}
\begin{proof}
We let $\chi \doteq \bb^{\lambda}(\bb_{s}\uu^{-1} - \log(\bb))$ be defined smoothly on $\R^{4}\setminus \{\origin\}\times [0, + \infty)$ and we extend it continuously at the origin. From the boundary conditions and Lemma \ref{ALFpreserved} we see that $\chi(\origin,t) = 0$ and $\chi(s,t)\rightarrow \infty$ as $s\rightarrow \infty$ for all positive times. Assume that $\chi$ attains some large negative value at a minimum point $(p_{0},t_{0})$ among prior times. The evolution of $\chi$ at $(p_{0},t_{0})$ becomes
\begin{align}
\partial_{t}\chi(p_{0},t_{0})|_{\chi_{\text{min}} < 0}& \geq \frac{1}{\bb^{2}}\left(\chi\left(\bb_{s}^{2}(\lambda^{2} + 4\lambda) - 2\lambda\bb_{s}\cc_{s}\uu^{-1} + 2\lambda\uu^{2} - 4\lambda\right) + 4\cc_{s}\bb^{\lambda}\right) \\  &+ \frac{1}{\bb^{2}}\left(\bb^{\lambda}\bb_{s}\uu^{-1}(2\bb_{s}^{2} - 4\bb_{s}\cc_{s}\uu^{-1}-2\uu^{2})+ \bb^{\lambda}(4 - 2\uu^{2} - 4\bb_{s}^{2} + 2\bb_{s}\cc_{s}\uu^{-1}) \right) 
\end{align}
\noindent Since $\lvert \chi_{\text{min}}\rvert = \bb^{\lambda}(\log(\bb) - \bb_{s}\uu^{-1}) \leq \bb^{\lambda}\log(\bb)$, we see that $\bb$ can be taken as large as we want once we pick $\lvert \chi_{\text{min}}\rvert$ large. Similarly, at any negative minimum of $\chi$ the derivative $\bb_{s}$ is small whenever the value of $\bb$ is sufficiently large, being $\cc$ uniformly bounded from above. Thus, whenever $\lvert \chi_{\text{min}}\rvert$ is large enough, we may write the evolution equation of $\chi$ as
\begin{equation}\label{equationuseful2}
\partial_{t}\chi(p_{0},t_{0})|_{\chi_{\text{min}}} \geq \frac{1}{\bb^{2}}\left(\lambda\lvert \chi_{\text{min}}\rvert + \bb^{\lambda}\bb_{s}\uu^{-1}(-4\bb_{s}\cc_{s}\uu^{-1} - 2\uu^{2}) + \bb^{\lambda}\right).
\end{equation}
\noindent Finally, we note that according to Lemma \ref{bbccsquaredccsboundedbyk} we can find $ k > 1$ such that
\[
4\bb^{\lambda}\bb_{s}^{2}\cc_{s}\uu^{-2} \leq 4\alpha\bb^{\lambda}\bb_{s}^{2}\uu^{-2 + k} \leq \alpha \bb^{\lambda}(\log(b))^{2}\uu^{k},
\] 
\noindent where again we have used that $\chi(p_{0},t_{0}) < 0$. We may then choose $\lambda \leq 1$ and conclude that 
\[
\partial_{t}\chi(p_{0},t_{0})|_{\chi_{\text{min}} < 0} \geq \frac{1}{\bb^{2}}\left(\lambda\lvert \chi_{\text{min}}\rvert  + \frac{\bb^{\lambda}}{2}\right) > 0.
\]
\end{proof}
We now show that $\bb_{s}$ cannot become degenerate in space-time regions where the quantity $\uu$ is bounded away from zero. On the one hand this control is necessary for the compactness result we rely on for proving that symmetries are preserved on any pointed Cheeger-Gromov limit. On the other, we see that if we were in a rotationally-symmetric setting, the solution would have positive asymptotic volume ratio. The latter observation will be crucial when showing that any Ricci flow in $\Gaf$ has curvature uniformly bounded in the space-time.%
\\We recall that the mean curvature of the Euclidean 3-sphere $S(\origin,z)$ with respect to the solution $g(t)$ is given by
\[
H(z,t) = \left(2\frac{\bb_{s}}{\bb} + \frac{\cc_{s}}{\cc}\right)(z,t).
\]
%\noindent The strategy consists in proving that $\cc H(\cdot,t)$ is bounded away from zero whenever $\uu$ is not small and then combining that with Lemma \ref{lowerboundforCA}.
% cannot become degenerate in space-time regions where $\uu$ is bounded away from zero.
%since both $\bb_{s}$ ad $\uu$ are scale-invariant, the uniform lower bound passes to any Cheeger-Gromov limit
%We may now prove the following
\begin{lemma}\label{lowerboundforCA}
If $(\R^{4},g(t))_{t\geq 0}$ is the maximal Ricci flow solution starting at some $g_{0}\in\Gaf$, then there exists $\beta > 0$ such that 
\[
\inf_{\R^{4}}\,\left(\bb_{s}\uu^{-1}\right)(\cdot,t) \geq \beta,
\]
\noindent for all times $t\geq 1$.% \babbo $g_{0}\in\Gaf$. There exists $\nu > 0$ such that $\bb_{s} \geq \nu\uu$ in the space-time.
\end{lemma}
\begin{proof}
\emph{Case (i): Positive mass.} We consider the maximal immortal Ricci flow solution evolving from $g_{0}$ for times $t \geq 1$ so that $\bb_{s}(\cdot,t)$ is positive everywhere by the strong maximum principle (see Lemma \ref{consistencymonotone}). Given $\alpha,\lambda > 0$ as in Lemma \ref{keyquantity}, we see that $\bb_{s}\uu^{-1}\geq 1$ in the time-dependent region $V(t) = \{p\in\R^{4}:\,\log(\bb(p,t)) - \alpha/\bb^{\lambda}(p,t) \geq 1\}$. We note that since $\bb$ is monotone we may identify $V(t)$ with the complement of some time-dependent Euclidean ball $B(\origin,r(t))$, with $t\mapsto r(t)$ a continuous function. From the estimate \eqref{rotationalsymmetryorder0} we derive that $\uu(\cdot,t)\geq \varepsilon$ in $B(\origin,r(t))$ for all $t\geq 1$ for some $\varepsilon > 0$, being $\bb(\cdot,t)$ uniformly bounded in $B(\origin,r(t))$. % Since the evolution equation of $\log(\uu)$ is given by
%\[
%\partial_{t}\log(\uu) = \Delta \log(\uu) + \frac{1}{\bb^{2}}\left(1 - \uu^{2}\right),
%\]
%\noindent we see that $\uu$ cannot approach zero along a sequence of interior minima belonging to $B(\origin,r(t))$. Therefore, there exists $\tilde{\varepsilon} > 0$ such that $\uu\geq \tilde{\varepsilon}$ in $B(\origin,r(t))$ for all $t\geq 0$.  
In particular, we deduce that 
\[
\cc H (r(t),t) \geq 2\bb_{s}\uu(r(t),t) \geq 2 \varepsilon^{2} > 0, 
\]
\noindent for all $t\geq 1$. Similarly, $\cc H(\origin,t) = 3$ for all times according to the boundary conditions. Therefore, if $\cc H$ attains some value $\tilde{\varepsilon} > 0$ small enough in $B(\origin,r(t))$ for the first time, then this must happen at an interior minimum point $(p_{0},t_{0})$ and we have
\[
\partial_{t}(\cc H)(p_{0},t_{0})|_{\tilde{\varepsilon}}\geq \frac{1}{\bb^{2}}\left(2\cc H \left(\uu^{2} - \bb_{s}^{2}\right) + 16\bb_{s}\uu(1 - \uu^{2}) \right).
\]
\noindent Since $\uu\geq \varepsilon$ in $B(\origin,r(t))$, we see that if $\tilde{\varepsilon}$ is small enough, then $\bb_{s}\uu^{-1}(p_{0},t_{0}) \leq 1$, which hence yields $\partial_{t}(\cc H)(p_{0},t_{0}) > 0$. We conclude that $\cc H$ is uniformly bounded from below in $B(\origin,r(t))$, for all times $t \geq 1$. Since $\bb_{s}\uu^{-1} \geq 1$ in $V(t)$, if the quantity attains some value $\beta$ sufficiently small for some time $t_{1} > 1$, then there exists an interior minimum point $(p_{0},t_{0})$ in $B(\origin,r(t_{0}))$ among times $t\in (1,t_{1}]$. The evolution equation of $\bb_{s}\uu^{-1}$ at such minimum point is 
\[
\partial_{t}\,(\bb_{s}\uu^{-1})(p_{0},t_{0}) \geq \frac{1}{\bb^{2}}\left(\bb_{s}\uu^{-1}(2\bb_{s}^{2} - 4\bb_{s}\cc_{s}\uu^{-1} - 2\uu^{2}) + 4\cc_{s} \right).
\]
\noindent From the estimate $\cc H \geq \tilde{\varepsilon}$ we conclude that $\cc_{s}(p_{0},t_{0})\geq \tilde{\varepsilon}/2$ whenever $\beta$ is small enough. Therefore, the right hand side of the evolution equation is positive and hence $\bb_{s}\uu^{-1} \geq \beta > 0$ for all times $t\geq 1$.

\emph{Case (ii): Zero mass.} In this case $\cc H(\cdot,t)\rightarrow 3$ at spatial infinity as long as the solution exists. Thus, we can argue as above using that $\uu \geq \delta$ in the space-time, for some $\delta > 0$, as follows from (iii) in Lemma \ref{consistencymonotone}.% if $\cc H$ attains some small value a at any interior minimum point By the same argument above in the Euclidean volume growth setting $\cc H\geq \beta > 0$ everywhere in the space-time, for $\cc H(\cdot,t)\rightarrow 3$ at spatial infinity as long as the solution exists. We may then argue as in the positive-mass setting.

\end{proof}

\section{The Ricci flow in $\Gk$}
 
In this section we extend the analysis of asymptotically flat warped Berger Ricci flows to solutions in $\Gk$. One of the main difficulty consists in controlling the flow in the space-time region where the roundness ratio $\uu$ is small. In the asymptotically flat case the stronger than quadratic decay of the curvature determines the behaviour of the warping coefficients at spatial infinity precisely. Once such decay is preserved along the flow, one can then rely on maximum principle arguments to derive time-independent bounds. On the contrary, for the case of $\Gk$ some extra work is needed to control the solution along the parabolic boundary of the space-time and hence ensure that the condition of opening faster than a paraboloid is indeed preserved. 

First, we note that one can argue as in Lemma \ref{ALFpreserved} to prove that the power law decay of the curvature in $\Gk$ persists along the solution.
\begin{lemma}\label{Gkpreserved}
Let $(\R^{4},g(t))_{t\geq 0}$ be the maximal Ricci flow solution starting at some $g_{0}\in\Gk$. For any $T^{\prime} < \infty$ there exists $\alpha(T^{\prime}) > 0$ such that 
\[
\sup_{p\in\R^{4}}(d_{g_{0}}(\origin,p))^{\frac{2}{k + 1}}\lvert\emph{Rm}_{g(t)}\rvert_{g(t)}(p) \leq \alpha(T^{\prime}),
\]
\noindent for all $t\in [0,T^{\prime}]$.
\end{lemma}
As a simple consequence of Lemma \ref{Gkpreserved} we derive that the volume growth rate of metrics in $\Gk$ is preserved along the solution as well as a conservation mass principle. We recall that given $g\in\Gk$ we call \emph{mass} (of $g$) the quantity $(\lim_{s\rightarrow \infty}\cc(s))^{-1}$ and we denote such positive finite number by $m_{g}$.
\begin{corollary}\label{conservationmassGk}
Let $(\R^{4},g(t))_{t\geq 0}$ be the maximal Ricci flow solution starting at some $g_{0}\in\Gk$. For any $t\geq 0$ there exist $B(t) > \beta(t) > 0$ such that 
\[
\beta(t)s_{0}^{\frac{1}{k + 1}} \leq \bb(s_{0},t) \leq B(t)s_{0}^{\frac{1}{k+1}}.
\]
\noindent Moreover, we have $m_{g(t)} = m_{g_{0}}$ for all $t\geq 0$.
\end{corollary}
\begin{proof}
Suppose for a contradiction that there exist a sequence $p_{j}$ and $t_{0}$ such that $s_{0}(p_{j})\rightarrow \infty$ and $\bb(p_{j},t_{0})(s_{0}(p_{j}))^{-\frac{1}{k + 1}} \rightarrow 0$. Then, from the decay of the curvature in Lemma \ref{Gkpreserved} we get
\[
\frac{\log\left(\frac{\bb(p_{j},t_{0})}{\bb(p_{j},0)}\right)}{t_{0}} \leq \alpha(t_{0})s_{0}^{-\frac{2}{k+1}}.
\]
\noindent Since by integrating \eqref{openingfasterparaboloid} we see that $\bb(s_{0},0) \geq \beta_{0}s_{0}^{\frac{1}{k+1}}$ for $s_{0}$ large enough, the contradiction follows. Similar arguments work for the upper bound while for the conservation of mass the proof is the same as in the asymptotically flat case (see Lemma \ref{massconserved}).
\end{proof}
\begin{remark}\label{remarkbbrm} We point out that according to Lemma \ref{Gkpreserved} and Corollary \ref{conservationmassGk} we deduce that for any $t\geq 0$ there exists some positive constant $\alpha(t)$ such that $\bb^{2}\lvert \text{Rm}\rvert(\cdot,t)\leq \alpha(t)$ on the time-slice $\R^{4}\times \{t\}$.
\end{remark}
Next, we show that the first order derivatives are uniformly bounded in the space-time. Since we cannot a priori control the behaviour of $\bb_{s}$ at spatial infinity on any time-slice, the proof requires an extra step when compared to its asymptotically flat counterpart. We recall that by Lemma \ref{consistencymonotone} the derivatives $\bb_{s}$ and $\cc_{s}$ are positive as soon as the flow starts.
\begin{lemma}\label{firstderivativesboundedGk}
Let $(\R^{4},g(t))_{t\geq 0}$ be the maximal Ricci flow solution starting at some $g_{0}\in\Gk$. Then
\[
\sup_{\R^{4}\times [0,\infty)}\,\left(\bb_{s} + \cc_{s}\right) < \infty.
\]
\end{lemma} 
\begin{proof}
Since by Lemmas \ref{consistencymonotone} and \ref{conservationmassGk} $\cc$ is uniformly bounded and spatially increasing, we see that $\cc_{s}(\cdot,t)$ is integrable for all $t\geq 0$. Moreover, from \eqref{sectionalhorizontal03} we derive that $\lvert \cc_{ss}\rvert \leq \alpha(t) \cc(s,t) \leq \alpha(t) m_{g_{0}}^{-1}$ in the space-time being the flow smooth for all positive times. Therefore $\cc_{s}(s,t)\rightarrow 0$ as $s\rightarrow \infty$ for all $t\geq 0$. We can then argue exactly as in \cite[Lemma 4.3]{work2} to prove that $\cc_{s}$ is uniformly bounded.
\\For what concerns $\bb_{s}$, we note that the evolution equation \eqref{equationbs} can be written as 
\[
\partial_{t}\,\bb_{s} = \Delta \bb_{s} + \frac{1}{\bb^{2}}\left(\bb_{s}(4 - \bb_{s}^{2} - (\cc_{s}\uu^{-1})^{2} - 6\uu^{2} - 2\bb_{ss}\bb) + 4\cc_{s}\uu\right).
\]
\noindent From the boundary conditions and the curvature being bounded we derive that given $t_{0} > 0$ there exist $r_{t_{0}}$ and $\alpha(t_{0})$ positive such that
\[
\partial_{t}\,\bb_{s} \leq \Delta \bb_{s} + \alpha(t_{0}),
\]
\noindent in $B_{g_{0}}(\origin,r_{t_{0}})\times [0,t_{0}]$. Since the curvature is uniformly bounded for all times in $[0,t_{0}]$, in the complement region $\R^{4}\setminus B_{g_{0}}(\origin,r_{t_{0}})\times [0,t_{0}]$ we can rely on \cite[Corollary 3.2, Lemma 3.4]{work2} to bound the evolution equation of $\bb_{s}$ by
\[
\partial_{t}\,\bb_{s}\leq \Delta \bb_{s} + \frac{1}{\bb^{2}}\left(4\bb_{s} + 4\cc_{s}\uu\right) - 2\frac{\bb_{s}\bb_{ss}}{\bb} \leq \Delta \bb_{s} + \frac{\alpha(t_{0})}{\bb} - 2\frac{\bb_{s}\bb_{ss}}{\bb} \leq \Delta \bb_{s} + \frac{\alpha(t_{0})}{\bb} + \alpha(t_{0})\bb_{s},
\] 
\noindent where we have also used that $\lvert k_{01}\rvert$ is uniformly bounded for $0\leq t \leq t_{0}$. Finally, since $\bb$ is monotone and the flow is smooth we can combine the estimates in the two space-time regions and conclude that there exists $\alpha(t_{0}) > 0$ such that
\[
\partial_{t}\,\bb_{s} \leq \Delta \bb_{s} + \alpha(t_{0})(1 + \bb_{s}),
\]
\noindent in $\R^{4}\times [0,t_{0}]$. Thus, since $\bb_{s}$ is exponentially bounded as we derive from $\lvert k_{01}\rvert$ being bounded, we may apply the maximum principle in \cite[Theorem 12.14]{ricciflowtechniques2} to deduce that for any $t_{0} > 0$ there exists $A_{t_{0}} > 0$ such that
\[
\bb_{s}(\cdot,t)\leq \sup_{\R^{4}}\,\bb_{s}(\cdot,0) + A_{t_{0}}
\]
\noindent in $\R^{4}\times [0,t_{0}]$. From Shi's derivative estimates \cite[Theorem 14.13]{ricciflowtechniques2} and the decay of the curvature in Lemma \ref{Gkpreserved} it follows that $\lvert \nabla \text{Rm}\rvert(s_{0},t) = \mathcal{O}(s_{0}^{-2/k+1})$ for all $t > 0$. Therefore, from the commutator formula we get
\[
\left \vert \partial_{t}\,\bb_{s}\right \vert = \left \vert \partial_{s}(-\text{Ric}_{11}\bb) + \text{Ric}_{ss}\bb_{s}\right \vert \leq \alpha\left( \lvert \nabla \text{Rm}\rvert \bb + \lvert \text{Rm}\rvert \bb_{s}\right).
\]
\noindent Since we have previously shown that $\bb_{s}(\cdot,t)$ is bounded on any time-slice we may apply Corollary \ref{conservationmassGk} and derive that for any $T^{\prime} >  1$ there exists $\alpha(T^{\prime})$ such that
\[
\lvert \partial_{t}\,\bb_{s}\rvert \leq \alpha(T^{\prime})(s_{0} + 1)^{-\frac{1}{k + 1}},
\]
\noindent in $\R^{4}\times [1,T^{\prime}]$. Therefore we have proved that for any $\varepsilon > 0$ and for any $t > 1$ there exists $r(t,\varepsilon)$ such that
\[
\bb_{s}(s_{0},t) \leq (\sup_{\R^{4}\times [0,1]}\,\bb_{s}) + \varepsilon,
\]
\noindent whenever $s_{0} \geq r(t,\varepsilon)$. Once we know that $\bb_{s}$ is uniformly bounded at spatial-infinity on any time-slice we can rely on the same argument in \cite[Lemma 4.3]{work2} to prove that in fact $\bb_{s}$ is uniformly bounded everywhere in the space-time.
\end{proof}
Thanks to Corollary \ref{conservationmassGk} and Lemma \ref{firstderivativesboundedGk} we can immediately extend the rotational symmetry type of bounds to Ricci flow solutions starting in $\Gk$. Namely, we have:
\begin{corollary}\label{rotationalsymmetryboundsGk}
If $(\R^{4},g(t))_{t\geq 0}$ is the maximal Ricci flow solution evolving from some $g_{0}\in\Gk$, then the following conditions hold:% There exists $\alpha > 0$ such that the following conditions are satisfied uniformly in the space-time:
%\begin{itemize}
%\item[(i)] $g_{0}$ is asymptotically flat.
%\item[(ii)] $\cc_{s}(\cdot,0) \geq 0$.
%\item[(iii)] $\bb_{s}(\cdot,0)\leq 2$.
%\end{itemize}
%\noindent Then the following estimates are uniformly satisfied in $\R^{4}\times [0,\infty)$:
\begin{align}
\sup_{\R^{4}\setminus\{\origin\}\times [0,+\infty)}&\,\,\frac{1}{\cc}(\bb_{s}^{2} - 4) < \infty, \notag \\
\sup_{\R^{4}\times [0,+\infty)}&\,\,\frac{1}{\cc}\left(1 - u\right) < \infty, \notag \\
\sup_{\R^{4}\times [0,+\infty)}&\,\,\frac{1}{\cc}\left\vert\cc_{s} - \uu\bb_{s}\right\vert < \infty.  \notag
\end{align}
\end{corollary}
Similarly, the decay of the curvature being preserved as in Lemma \ref{Gkpreserved} and the control on the asymptotic behaviour of the warping coefficients as in Corollary \ref{conservationmassGk} ensure that second order estimates analogous to the asymptotically flat case still hold for Ricci flows evolving from initial data in $\Gk$. For example, since $\bb_{ss}/\bb = \mathcal{O}(s_{0}^{-2/k+1})$ we see that $\bb_{ss}\cc$ decays as $s_{0}^{-1/k+1}$ and one can hence apply maximum principle arguments as in Lemma \ref{secondderivativeslemma} once we know that the first order derivatives are uniformly bounded. In particular, the curvature of the solution is again controlled by the size of the Hopf-fiber:
\begin{corollary}\label{curvaturecontrolledbycgk}
If $(\R^{4},g(t))_{t\geq 0}$ is the maximal Ricci flow solution starting at some $g_{0}\in\Gk$, then 
\[
\sup_{\R^{4}\times [0, + \infty)}\,\left(\cc^{2}\lvert\emph{Rm}_{g(t)}\rvert_{g(t)}\right) < \infty.
\]
\end{corollary}
Next, we prove that for Ricci flows starting in $\Gk$ the quantity $\bb_{s}\uu^{-1}$ is controlled from below in any region where $\uu$ becomes degenerate exactly as for the asymptotically flat case. If a Ricci flow solution in $\Gk$ has curvature bounded uniformly in time, then such estimate implies that any infinite-time singularity model must open up along the $S^{2}$-direction faster than a paraboloid in $\R^{3}$. However, differently from the asymptotically flat case, for solutions in $\Gk$ we need a preliminary bound to make sure that $\bb_{s}\uu^{-1}$ does indeed diverge at spatial infinity on any time-slice. 
\begin{lemma}\label{preliminarykeybound}
Let $(\R^{4},g(t))_{t\geq 0}$ be the maximal Ricci flow solution starting at some $g_{0}\in\Gk$ and let $k < \bar{k} < 1$ and $\delta\in (0,1-\bar{k})$.  For any $t\geq 0$ there exists $\alpha(t) > 0$ such that:
\[
\uu^{-\delta}\left(\bb_{s}\uu^{-\bar{k}} - 1\right) \geq -\alpha(t) > -\infty.
\]
\end{lemma}
\begin{proof}
We set $F_{\bar{k},\delta} \doteq \uu^{-\delta}(\bb_{s}\uu^{-\bar{k}} - 1)$. By the boundary conditions we see that $F_{\bar{k},\delta}(\origin,t) = 0$ for all $t\geq 0$. Moreover, from the definition of $\Gk$ it follows that $F_{\bar{k},\delta}(s_{0},0)\rightarrow \infty$ as $s_{0}\rightarrow \infty$, meaning that $\inf F_{\bar{k},\delta}(\cdot,0) > -\infty$. We now argue as for the proof of Lemma \ref{firstderivativesboundedGk}. First, the evolution equation of $F_{\bar{k},\delta}$ is given by
\begin{align*}
\partial_{t}F_{\bar{k},\delta} &= \Delta F_{\bar{k},\delta} - 2\delta\uu^{-\delta}\left(\frac{\bb_{s}}{\bb} - \frac{\cc_{s}}{\cc}\right)\left( \bb_{ss}\uu^{-\bar{k}} + \bar{k}\bb_{s}\uu^{-\bar{k}}\left(\frac{\bb_{s}}{\bb} - \frac{\cc_{s}}{\cc}\right)\right) \\ &+ \uu^{-\delta}\left(-2(\bar{k} + 1)\frac{\bb_{s}\bb_{ss}\uu^{-\bar{k}}}{\bb} + 2\bar{k}\frac{\cc_{s}\bb_{ss}\uu^{-\bar{k}}}{\cc}\right) \\
&+ \frac{\uu^{-\delta - \bar{k}}}{\bb^{2}}\left(\bb_{s}\left(4(1-\bar{k}) -(1 + \bar{k}^{2})(\bb_{s}^{2} + \left(\frac{\cc_{s}}{\uu}\right)^{2}) + 2\bar{k}^{2}\bb_{s}\frac{\cc_{s}}{\uu} + \uu^{2}(4\bar{k}- 6)\right) + 4\cc_{s}\uu\right) \\
&+ \frac{F_{\bar{k},\delta}}{\bb^{2}}\left(-\delta^{2}(\bb_{s} - \cc_{s}\uu^{-1})^{2} - 4\delta(1-\uu^{2})\right).
\end{align*}
\noindent Since the curvature is bounded, from the boundary conditions one can check that for any $t_{0} > 0$ there exists $r_{t_{0}} > 0$ and $\alpha(t_{0}) > 0$ such that
\[
\partial_{t}F_{\bar{k},\delta} \geq \Delta F_{\bar{k},\delta} - \alpha(t_{0}),
\]
\noindent in $B_{g_{0}}(\origin,r_{t_{0}})\times [0,t_{0}]$. For analysing the terms in the evolution equation for radii larger than $r_{t_{0}}$ we first note that by Corollary \ref{conservationmassGk}
\begin{equation}\label{equation1useful}
\frac{\uu^{-\delta -\bar{k}}}{\bb} = \mathcal{O}\left( s_{0}^{-\frac{1}{k+1}(1 - \delta - \bar{k})}\right),
\end{equation}
\noindent which hence decays at spatial infinity on any time-slice because $\bar{k} + \delta < 1$. From \eqref{equation1useful} and Lemma \ref{Gkpreserved} we derive that $\uu^{-\bar{k}-\delta}\lvert\bb_{ss}\rvert$ decays at spatial infinity as long as the solution exists. Since by Lemma \ref{firstderivativesboundedGk} the first derivatives $\bb_{s}$ and $\cc_{s}$ are bounded, we find that all the second order terms in the evolution equation of $F_{\bar{k},\delta}$ decay to zero at the rate given in \eqref{equation1useful} for all $t\geq 0$. Similarly, from Lemma \ref{Gkpreserved} and Corollary \ref{conservationmassGk} we see that $\bb^{2}\lvert \text{Rm}\rvert(\cdot,t)$ is bounded on any time-slice, meaning that $\lvert \bb_{s}\cc_{s}\uu^{-1}\rvert \leq \alpha(t)$. Thus any term of the form $\bb_{s}\cc_{s}\uu^{-\bar{k} - \delta}$ decays at the same rate given by \eqref{equation1useful}. To sum up, in the region $\R^{4}\setminus B_{g_{0}}(\origin,r_{t_{0}})\times [0,t_{0}]$ we can then write the evolution equation of $F_{\bar{k},\delta}$ as 
\[
\partial_{t}F_{\bar{k},\delta} \geq \Delta F_{\bar{k},\delta} - \alpha(t_{0}) + \frac{F_{\bar{k},\delta}}{\bb^{2}}\left(-\delta^{2}(\bb_{s} - \cc_{s}\uu^{-1})^{2} - 4\delta(1-\uu^{2})\right),
\]
\noindent for some $\alpha(t_{0})$. Finally, we note that
\begin{align*}
\frac{F_{\bar{k},\delta}}{\bb^{2}}\left(-\delta^{2}(\bb_{s} - \cc_{s}\uu^{-1})^{2} - 4\delta(1-\uu^{2})\right) &\geq \frac{\bb_{s}\uu^{-\bar{k} - \delta}}{\bb^{2}}\left(-\delta^{2}(\bb_{s} - \cc_{s}\uu^{-1})^{2} - 4\delta(1-\uu^{2}) \right) \\ &\geq \frac{\bb_{s}\uu^{-\bar{k} - \delta}}{\bb^{2}}\left(-\delta^{2}(\bb_{s}^{2} + (\cc_{s}\uu^{-1})^{2}) - 4\delta) \right),
\end{align*}
\noindent and the last terms are again bounded away from the origin on any time-slice as observed above. Therefore, for any  $t_{0} > 0$ there exists $\alpha(t_{0}) > 0$ such that 
\[
\partial_{t}F_{\bar{k},\delta} \geq \Delta F_{\bar{k},\delta} - \alpha(t_{0}) .
\]
\noindent From the maximum principle \cite[Theorem 12.14]{ricciflowtechniques2} we conclude that
\[
F_{\bar{k},\delta}(\cdot,t) \geq \inf_{\R^{4}}\,F_{\bar{k},\delta}(\cdot,0) - \alpha(t) > -\infty,
\]
\noindent for all positive times.
\end{proof}
We finally need to check that the spatial derivative $\cc_{s}$ decays at some rate in any space-time region where $\uu$ is small. 
\begin{lemma}\label{decayofcsGk}
Let $(\R^{4},g(t))$ be the maximal Ricci flow solution evolving from some $g_{0}\in\Gk$. For any $\hat{k}\in (k,1)$ we have
\[
\sup_{\R^{4}\times [0,\infty)}\,\left(\cc_{s}\uu^{-1 + \hat{k}}\right) < \infty.
\]
\end{lemma}
\begin{proof}
Given $\hat{k} > k$, let $\bar{k}\in (k,\hat{k})$ and $\delta$ be defined so that Lemma \ref{preliminarykeybound} holds. As observed in Remark \ref{remarkbbrm}, from \eqref{sectionalvertical12} we see that for any $t\geq 0$ there exists $A(t) > 0$ such that $\lvert \bb_{s}\cc_{s}\uu^{-1}\rvert \leq A(t)$. Thus, by Lemma \ref{preliminarykeybound} we get
\[
A(t) \geq \lvert \bb_{s}\cc_{s}\uu^{-1}\rvert \geq \left \vert \cc_{s}(-\alpha(t)\uu^{\delta} + 1)\uu^{-1 + \bar{k}}\right \vert.
\]
\noindent Therefore, we find that $\lim_{s_{0}\rightarrow \infty}(\cc_{s}\uu^{-1 + \hat{k}})(s_{0},t) = 0$ for all $t\geq 0$. The evolution equation of $\cc_{s}\uu^{-\hat{k}}$ at any positive maximum is given by 
\begin{align*}
\partial_{t}\left(\cc_{s}\uu^{-\hat{k}}\right)|_{\text{max}} &\leq \frac{\cc_{s}\uu^{-\hat{k}}}{\bb^{2}}\left( \bb_{s}^{2}(\hat{k}^{2}-2) -4\hat{k} + \uu^{2}(-6+4\hat{k}) + (\cc_{s}\uu^{-1})^{2}(\hat{k}^{2} -2\hat{k})\right)\\ &+\frac{1}{\bb^{2}}\left(\cc_{s}\uu^{-\hat{k}}\left(\bb_{s}\cc_{s}\uu^{-1}(-2\hat{k}^{2} + 2\hat{k})\right) + 8\uu^{3-\hat{k}}\bb_{s}\right)
\end{align*}
\noindent Since $\hat{k} < 1$ and $\bb_{s}$ is uniformly bounded by Lemma \ref{firstderivativesboundedGk}, we get
\[
\partial_{t}\left(\cc_{s}\uu^{-\hat{k}}\right)|_{\text{max}} \leq \frac{1}{\bb^{2}}\left(\cc_{s}\uu^{-\hat{k}}\left(-4\hat{k} + (\cc_{s}\uu^{-1})^{2}(\hat{k}^{2} -2\hat{k})+ \bb_{s}\cc_{s}\uu^{-1}(-2\hat{k}^{2} + 2\hat{k})\right) + \alpha \right).
\]
\noindent Finally, if the value of the maximum is large enough, then we find that
\[
(\cc_{s}\uu^{-1})^{2}(\hat{k}^{2} -2\hat{k})+ \bb_{s}\cc_{s}\uu^{-1}(-2\hat{k}^{2} + 2\hat{k}) \leq \hat{k}\cc_{s}\uu^{-1}(-\cc_{s}\uu^{-1} + 2\bb_{s}) < 0.
\]
\noindent Thus, we have shown that $\partial_{t}(\cc_{s}\uu^{-\hat{k}}) < 0$ at any maximum value large enough. That completes the proof.
\end{proof}
We may now complete the section by noting that, as in the asymptotically flat case, for any Ricci flow in $\Gk$ the warping coefficient in the directions orthogonal to the Hopf-fiber grows faster than a paraboloid in $\R^{3}$.
\begin{lemma}\label{keylemmaGk}
Let $(\R^{4},g(t))_{t\geq 0}$ be the maximal Ricci flow solution starting at some $g_{0}\in\Gk$. There exist $\alpha, \lambda > 0$ such that
\[
\bb^{\lambda}(\bb_{s}\uu^{-1} - \log(\bb))\geq -\alpha,
\]
\noindent uniformly in the space-time. Moreover, there exists $\beta > 0$ such that
\[
\inf_{\R^{4}}\,\left(\bb_{s}\uu^{-1}\right)(\cdot,t) \geq \beta,
\]
\noindent for all times $t\geq 1$.% \babbo $g_{0}\in\Gaf$. There exists $\nu > 0$ such that $\bb_{s} \geq \nu\uu$ in the space-time.
\end{lemma}
\begin{proof}
Given $\chi \doteq \bb^{\lambda}(\bb_{s}\uu^{-1} - \log(\bb))$, we note that $\chi$ vanishes at the origin for all times and that, according to Lemma \ref{preliminarykeybound}, $\chi(s_{0},t)\rightarrow \infty$ as $s_{0}\rightarrow \infty$ for all positive times. Thus we can consider the evolution equation of $\chi$ at some negative minimum point $(p_{0},t_{0})$ and arguing as in the proof of Lemma \ref{keyquantity} for the asymptotically flat setting, we deduce that whenever $\lvert \chi_{\text{min}}\rvert$ is large enough, then the evolution equation of $\chi$ satisfies \eqref{equationuseful2}. Namely, we have
\[
\partial_{t}\chi(p_{0},t_{0})|_{\chi_{\text{min}}} \geq \frac{1}{\bb^{2}}\left(\lambda\lvert \chi_{\text{min}}\rvert + \bb^{\lambda}\bb_{s}\uu^{-1}(-4\bb_{s}\cc_{s}\uu^{-1} - 2\uu^{2}) + \bb^{\lambda}\right)
\]
\noindent From $\chi(p_{0},t_{0}) < 0$ we derive 
\[
\bb^{\lambda}(1 - 4\bb_{s}^{2}\cc_{s}\uu^{-2}) \geq \bb^{\lambda}(1 - 4(\log(\bb))^{2}\cc_{s}).
\]
\noindent If we pick $\hat{k} \in (k,1)$ and $\alpha_{\hat{k}}$ so that Lemma \ref{decayofcsGk} holds, then the last term can be bounded as:
\[
\bb^{\lambda}(1 - 4(\log(\bb))^{2}\cc_{s}) \geq \bb^{\lambda}(1 - 4\alpha_{\hat{k}}(\log(\bb))^{2}\uu^{1 - \hat{k}}),
\]
\noindent which is positive whenever $\lvert \chi_{\text{min}}\rvert$ and hence $\uu^{-1}$ are large enough. Therefore, if we let $\lambda \leq 1$, then we obtain:
\[
\partial_{t}\chi(p_{0},t_{0})|_{\chi_{\text{min}}} > 0.
\]
\noindent Once we know that $\chi$ is uniformly bounded from below in the space-time, one can argue exactly as in the positive mass-case of Lemma \ref{lowerboundforCA}.
\end{proof}

\section{Compactness of warped Berger Ricci flows}
In this section we show that a class of complete, bounded curvature warped Berger Ricci flows with monotone coefficients is compact under the pointed Cheeger-Gromov topology. The main step in the argument consists in proving that the Killing vectors generating the SU(2)U(1) symmetry pass to the limit without becoming degenerate. Once we know that the Cheeger-Gromov limit is a warped Berger Ricci flow, the control on the curvature and the monotonicity conditions allow us to prove smooth convergence of the warping functions $\bb$ and $\cc$ up to diffeomorphisms. The next result plays a central role in the convergence argument because it allows to work directly at the level of the metric coefficients. In particular, this will be of importance when classifying ancient solutions to the Ricci flow because whenever the compactness result holds one can always apply maximum principle arguments to geometric quantities by passing to pointed Cheeger-Gromov limits where critical values are indeed achieved in the space-time. %We first state the result before commenting on the assumptions and on the key differences with an analogous result for U(2)-invariant Ricci flows on the blow up of $\mathbb{C}^{2}/\mathbb{Z}_{k}$ proved in \cite{appleton2}.
%In fact, in a few points we only need to simply adapt the arguments in \cite{appleton2} to our different topology. We point out that the change in topology also leads to work with different assumptions to ensure that the symmetries do not become degenerate when passing to the limit. Finally, we observe that depending on the information one has on the actual sequence of warped Berger Ricci flows a few of the assumptions may be weakened. In the following we present a self-contained general form of the compactness property.
%Let me write the more general version of the compactness argument.
\begin{proposition}\label{compactnessresult}
Let $(\R^{4},g_{j}(t),p_{j})_{t\in I}$ be a sequence of pointed complete warped Berger solutions to the Ricci flow with monotone coefficients defined on $I\ni 0$. If 
\begin{align}
\sup_{j}\left(\sup_{\R^{4}\times I}\,\lvert \emph{Rm}_{g_{j}(t)}\rvert_{g_{j}(t)}\right)&< \infty, \label{curvatureboundedCA} \\
\sup_{j}\left(\sup_{\R^{4}\times I}\,(\bb_{j})_{s} + (\cc_{j})_{s}\right) &< \infty, \label{derivativesboundedCA} \\
\liminf_{j\rightarrow \infty}\left(\bb_{j}(p_{j},0)\right) + \liminf_{j\rightarrow\infty}\left((\bb_{j})_{s}(p_{j},0)\right) &> 0, \label{lowerboundsweetCA}\\
\limsup_{j\rightarrow \infty}\,\left(\bb_{j}(p_{j},0)\right) &< \infty, \label{bbjboundedCA} \\
\liminf_{j\rightarrow \infty}\,\left(u_{j}(p_{j},0)\right)&> 0, \label{uujboundedCA}
\end{align}
\noindent then $(\R^{4},g_{j}(t),p_{j})$ subsequentially converges in the pointed Cheeger-Gromov sense to a complete \emph{SU(2)U(1)} invariant Ricci flow solution $(M_{\infty},g_{\infty}(t),p_{\infty})_{t\in I}$ satisfying:
\begin{itemize}
\setlength\itemsep{0.6em}
\item[(i)] $M_{\infty} = \R^{4}$ or $M_{\infty} = \R \times S^{3}$.
\item[(ii)] There exist warping coefficients $\xi_{\infty},\bb_{\infty},\cc_{\infty}$ such that $g_{\infty}(t)$ can be written as
\[
g_{\infty}(t) = \xi_{\infty}(x_{\infty},t)dx^{2}_{\infty} + \bb^{2}_{\infty}(x_{\infty},t)\,(\sigma_{1}^{2} + \sigma_{2}^{2}) + \cc^{2}_{\infty}(x_{\infty},t)\,\sigma_{3}^{2},
\]
\noindent where $x_{\infty}(\cdot) = d_{g_{\infty}(0)}(\origin_{\infty},\cdot)$ if $M_{\infty} = \R^{4}$, and $x_{\infty}(\cdot) = d_{g_{\infty}(0)}(\Sigma_{p_{\infty}},\cdot)$, with $\Sigma_{p_{\infty}}$ the principal orbit passing through $p_{\infty}$, if $M_{\infty} = \R\times S^{3}$.
\item[(iii)] There exist radial functions $s_{j}$ such that $\xi_{j}(s_{j},t)$, $\bb_{j}(s_{j},t)$, $\cc_{j}(s_{j},t)$ converge smoothly on compact sets to $\xi_{\infty}(x_{\infty},t)$, $\bb_{\infty}(x_{\infty},t)$, $\cc_{\infty}(x_{\infty},t)$ respectively.

\end{itemize}
\end{proposition}
\noindent\textbf{Remark on the assumptions.} The uniform bound on the curvature and the monotonicity of the coefficients guarantee that Hamilton's compactness result can be applied to the sequence of solutions. The control on the first order derivatives along with \eqref{bbjboundedCA} ensure that the Killing vectors are bounded in any geodesic ball. Proposition \ref{compactnessresult} has a counterpart for warped Berger Ricci flows on the blow up of $\mathbb{C}^{2}/\mathbb{Z}_{k}$ satisfying a few first and second-order estimates \cite{appleton2}. However, for the topologies analysed in \cite{appleton2} the compactness property for complete solutions is formulated in the case where the Ricci flows are $\kappa$-non-collapsed at some sequence of scales diverging to infinity so that the resulting limit is $\kappa$-non-collapsed (for all scales) \cite[Corollary 8.2]{appleton2}. Such assumption is natural when studying \emph{finite-time singularity models} for the Ricci flow, which arise via a blow-up procedure. In fact, Perelman proved that any finite-time singularity model of the Ricci flow is $\kappa$-non-collapsed. On the other hand, this assumption is not available in our setting for we are interested in proving long-time convergence of the Ricci flow to \emph{infinite-time singularity models} which are, in the case of the Taub-NUT metric, collapsed for some sufficiently large scale. The latter represents a key difference and accounts for the conditions \eqref{lowerboundsweetCA} and \eqref{uujboundedCA}, which ensure that the Killing vectors do not become degenerate when passing to the limit.
\begin{remark} We point out that the structure of the proof below mainly follows from adapting the analysis in Section 5 of \cite{work2} and especially the analogous local result in \cite{appleton2}. Since the compactness property derived in \cite{appleton2} relies on a different set of assumptions and works on the topology of the blow-up of $\mathbb{C}^{2}/\mathbb{Z}_{k}$, we still present a full argument in detail. 
\end{remark}
% Indeed, while in \cite{appleton2} the non-collapsed property implies that the squashing factor cannot become degenerate, in our setting the application of  the compactness result to classify ancient solutions opening faster than a paraboloid relies on the  one needs to verify that the roundness ratio $\uu$ stays non-degenerate along the sequence of points $p_{j}$, while in the non-collapsed case such control is always implied, meaning that 
%The assumption 
%The remaining conditions are meant to make sure that the symmetries do pass to the limit as well. Properties (v), and (vi) imply that the Killing vector fields generating the SU(2)$\times$U(1) action are $C^{0}$-bounded on compact regions. One can then use the Killing equations and the bounds on the curvature to actually prove that there exist limit Killing vectors on $(M_{\infty},g_{\infty}(t))$. From the assumptions (iii), (iv) and (vii) one can finally derive that such limit Killing vectors are non-trivial, so that the full group action is recovered in the limit. 
\begin{proof}
%We point out that the structure of the proof below mainly follows from adapting the analogous local result in \cite{appleton2}. Nonetheless, the compactness result we present here  Since  we work on a different topology and we consider a different class of warped Berger solutions to the Ricci flow, we 
% and indeed in some points the proof follows from adapting the analogous local result in
First, we recall that the monotonicity of the warping functions implies that $(\R^{4},g_{j}(0))$ does not have closed geodesics. Since the curvature is uniformly bounded, we see that $\inf_{j}\text{inj}(g_{j}(0)) > 0$ and we may then apply Hamilton's compactness theorem and extract a subsequence converging in the pointed Cheeger-Gromov sense to a Ricci flow solution $(M_{\infty},g_{\infty}(t),p_{\infty})_{t\in I}$.

In the following we denote by $\Phi_{j}$ the diffeomorphisms given by the Cheeger-Gromov convergence. We also observe that by \eqref{bbjboundedCA} we can rely on the same argument in \cite{work}[Lemma 4.1] to prove that the limit manifold is \emph{simply connected}.
\\Consider the Killing vector fields $\{Y_{1},Y_{2},Y_{3},X_{3}\}$ generating the SU(2)U(1) symmetries. Since we have
\[
\cc^{2}_{j}(\cdot,t)g_{S^{3}}\leq \bb^{2}_{j}(\cdot,t)(\sigma_{1}^{2}+\sigma_{2}^{2}) + \cc^{2}_{j}(\cdot,t)\sigma_{3}^{2} \leq \bb^{2}_{j}(\cdot,t)g_{S^{3}},
\]
\noindent with $g_{S^{3}}$ the bi-invariant constant curvature 1 metric on the 3-sphere, and $\{Y_{i}\}$ are orthonormal with respect to $g_{S^{3}}$, we deduce that 
\begin{equation}\label{nicecontrolbiinvariance}
\cc_{j}(\cdot,t) \leq \lvert Y_{i} \rvert_{j}(\cdot,t) \leq \bb_{j}(\cdot,t),
\end{equation}
\noindent for all $t\in I$ and for $i = 1,2,3$. Let $\nu > 0$ and $q\in B_{g_{j}(0)}(p_{j},\nu)$. From the conditions \eqref{derivativesboundedCA} and \eqref{bbjboundedCA} we derive that there exists $\alpha > 0$ such that
\[
\bb_{j}(q,0) \leq \bb_{j}(p_{j},0) + \left(\sup_{B_{g_{j}(0)}(p_{j},\nu)}(\bb_{j})_{s}\right)\nu \leq \alpha(1 + \nu).
\]
\noindent We may then extend such bounds to other times by using \eqref{curvatureboundedCA}, for given $t\in I$ we have
\[
\lvert \partial_{t}\log \bb_{j}\rvert(\cdot,t) \lvert \leq \lvert\text{Ric}_{j}\rvert_{j}(\cdot,t) \leq \alpha.
\]
\noindent Therefore, for any $\nu > 0$ and $t\in I$ there exists a constant $\alpha = \alpha(t,\nu)$ such that
\begin{equation}\label{c0controlkillingvectors}
\sup_{B_{g_{j}(0)}(p_{j},\nu)}\lvert Y_{i}\rvert_{j}(\cdot,t) \leq \alpha(t,\nu),
\end{equation}
\noindent for all $i = 1,2,3$. The Killing equation also implies
\[
\lvert \nabla^{2}_{j} Y_{i}\rvert_{j}(\cdot,t) \leq \alpha \lvert Y_{i}\rvert_{j}\lvert\text{Rm}_{j}\rvert_{j}(\cdot,t),
\]
\noindent for all times $t\in I$. By the Cheeger-Gromov convergence we deduce that, up to passing to a subsequence, there exist $C^{1}$-limits $\{Y_{i,\infty}\}$ defined on $B_{g_{\infty}(0)}(p_{\infty},1)$. The $C^{1}$-convergence implies that $\{Y_{i,\infty}\}$ are $g_{\infty}(t)$-Killing vector fields. We may then proceed as in \cite{work2}[Lemma 5.8] to derive that $\{Y_{i,\infty}\}$ extend to smooth Killing vectors on $M_{\infty}$. Similar conclusions apply to $X_{3}$, which converges to a $g_{\infty}(t)$-Killing vector $X_{3,\infty}$ up to pulling back by $\Phi_{j}$. 
\\We now show that the Killing vectors are not degenerate. Namely, we have
\vspace{0.1in}
\begin{claim}\label{cohomogen1action}
\emph{SU(2)} acts on $(M_{\infty},g_{\infty}(t))$ with cohomogeneity-\emph{1}.
\end{claim}
\begin{proof}[Proof of Claim \ref{cohomogen1action}]
We first show that the limit Killing vectors are not trivial.

 If there exists $\varepsilon > 0$ such that $\cc_{j}(p_{j},0) \geq \varepsilon$ along a subsequence, then from \eqref{nicecontrolbiinvariance} we see that $Y_{i,\infty}$ do not vanish at $p_{\infty}$. 

We may then assume that $\cc_{j}(p_{j},0)\rightarrow 0$. According to \eqref{uujboundedCA} we also have $\bb_{j}(p_{j},0)\rightarrow 0$. Therefore, from \eqref{lowerboundsweetCA} it follows that $(\bb_{j})_{s}(p_{j},0)\geq 2\beta > 0$, for some constant $\beta$ independent of $j$. Since the curvature is uniformly bounded we see that $\lvert (\bb_{j})_{ss}\rvert(\cdot,t) \leq \alpha$, which implies $(\bb_{j})_{s}(\cdot,0)\geq \beta$ in $B_{g_{j}(0)}(p_{j},r)$ for some $r$ small enough. Similarly, by \eqref{rmo123} and \eqref{uujboundedCA} we get that $\uu_{j}(\cdot,0)$ is uniformly bounded from below in $B_{g_{j}(0)}(p_{j},\tilde{r})$ for some radius sufficiently small. We can then pick $\hat{r} = \min\{r,\tilde{r}\}$ and conclude that there exist points $q_{j}\in B_{g_{j}(0)}(p_{j},\hat{r})$ such that $\cc_{j}(q_{j},0)$ admits a positive lower bound. Therefore we find $q_{\infty}\in B_{g_{\infty}(0)}(p_{\infty},\hat{r})$ such that $Y_{i,\infty}$ are not trivial at $q_{\infty}$.  
%If $\uu_{j}\geq \nu > 0$ in $B_{g_{j}(0)}(p_{j},1)$, then from (iv) we derive that $(\bb_{s})_{j}\geq \tilde{\nu} > 0$ in $B_{g_{j}(0)}(p_{j},1)$, which then implies that there exists a sequence $q_{j}\in S_{g_{j}(0)}(p_{j},1)$ such that $\bb_{j}(q_{j},0) \geq \tilde{\nu}/2$ for all $j$ large enough. In particular, the estimate (vii) implies that $\cc_{j}(q_{j},0)$ must be bounded away from zero as well. Therefore, we find $q_{\infty}\in S_{g_{\infty}(0)}(p_{\infty},1)$ such that $Y_{i,\infty}$ are not trivial at $q_{\infty}$. 
%If instead $\uu_{j}(z_{j},0)$ is converging to zero along a sequence of points $z_{j}\in B_{g_{j}(0)}(p_{j},1)$, then by the rotational symmetry type of estimates (vii) we deduce that $\bb_{j}$ and $\cc_{j}$ are bounded away from zero along the sequence and we can then conclude that there exists some $z_{\infty}\in B_{g_{\infty}(0)}(p_{\infty},1)$ where the Killing vectors $Y_{i,\infty}$ do not vanish. The same argument applies to $X_{3,\infty}$.

Once we know that $\{Y_{i,\infty}\}$ are non-trivial, we may deduce that there exists a non-degenerate copy of $\mathfrak{su}[2]$ in the Lie algebra of Killing vector fields $\mathfrak{iso}(M_{\infty},g_{\infty}(t))$ because the Lie brackets pass to the limit. Since the limit is complete we can integrate the Lie algebra action and obtain that SU(2) acts on $(M_{\infty},g_{\infty}(t))$. Consider $q\in M_{\infty}$ such that $Y_{i,\infty}$ do not vanish at $q$. Suppose that there exist coefficients $\alpha_{i}$ such that the vector $Y_{\infty} = \sum_{i}\alpha_{i}Y_{i,\infty}$ vanishes at $q$. 
A diagonal argument yields
\[
0 = \lvert Y_{\infty}\rvert^{2}_{g_{\infty}(0)}(q) = \lim_{j\rightarrow \infty} \left\vert \sum_{i} \alpha_{i}Y_{i} \right\vert^{2}_{g_{j}(0)}(\Phi_{j}(q)). 
\]
\noindent As observed above, for any spherical vector field $Y$ we have $\lvert Y \rvert_{g_{j}(0)}(p) \geq \cc_{j}(p,0)\lvert Y \rvert_{g_{S^{3}}}$. Since the right-invariant vector field $Y_{i}$ are orthonormal with respect to the round metric on the 3-sphere and $\cc_{j}(\Phi_{j}(q),0)$ is bounded away from zero being the Killing vectors $Y_{i,\infty}$ non-trivial at $q$, we conclude that the limit above is zero if and only if $\alpha_{i} = 0$. Thus, there exists an SU(2)-orbit of codimension 1, which is exactly the claim.
\end{proof}
\noindent Since the limit manifold $M_{\infty}$ is non-compact and simply connected, from the cohomogeneity 1 SU(2) action we derive that either $M_{\infty}$ is foliated by principal orbits, in this case $M_{\infty} = \R\times S^{3}$, or there exists a singular orbit $\Sigma_{\text{sing}}$. From now on we assume that the action admits a singular orbit for the case of the cylinder can be dealt with similarly. As discussed in Section 2.1, we can diagonalize the Ricci flow limit and use the extra-degree of symmetry provided by the Killing vector field $X_{3,\infty}$ to write the solution as% along  write the Ricci flow limit as%hen there are no singul orbits, or $M_{\infty} = \R^{4}$ whenever there exists a singular orbit as follows from (vii). 
%We may then write the limit solution as
%\[
%g_{\infty}(t) = dy_{t}\times dy_{t} + g_{y_{t}},
%\]
%\noindent for some 1-parameter family of left-invariant metrics $g_{y_{t}}$ on $S^{3}$.
%\textbf{As proved in the general setting above}, using the fourth Killing vector $X_{3,\infty}$ we deduce that $(M_{\infty},g_{\infty}(t))$ admits an SU(2)$\times$U(1) cohomogeneity 1 action and we can write the limit solution as
\[
g_{\infty}(t) = \xi_{\infty}(x_{\infty},t)dx^{2}_{\infty} + \bb^{2}_{\infty}(x_{\infty},t)\,(\sigma_{1}\otimes \sigma_{1} + \sigma_{2}\otimes \sigma_{2}) + \cc^{2}_{\infty}(x_{\infty},t)\,\sigma_{3}\otimes \sigma_{3},
\]
\noindent where $\{\sigma_{i}\}$ is the coframe dual to the Milnor frame $\{X_{i,\infty}\}$ of left-invariant vectors induced by the right-invariant Killing vectors $\{Y_{i,\infty}\}$, while $x_{\infty}(\cdot) = d_{g_{\infty}(0)}(\Sigma_{\text{sing}},\cdot)$. In particular, we have $\cc_{\infty} = \lvert X_{3,\infty}\rvert_{\infty}$ in the space-time. Thus $\cc_{j}\rightarrow \cc_{\infty}$ on compact sets and one obtain an analogous conclusion for $\bb_{j}\rightarrow \bb_{\infty}$ once the points $p_{j}$ are chosen of the form $((s_{0})_{j},e)$, with $e$ the identity in SU(2).
 
Let $z_{\infty}\in\Sigma_{\text{sing}}$ and let $z_{j} = \Phi_{j}(z_{\infty})$. If, for a contradiction, there exists $\varepsilon > 0$ such that $d_{g_{j}(0)}(\origin,z_{j}) \geq 2\varepsilon$, then we could find points $\tilde{z}_{j}$ satisfying $d_{g_{j}(0)}(z_{j},\tilde{z}_{j}) = \varepsilon$ with $\cc_{j}(\tilde{z}_{j},0) \leq \cc_{j}(z_{j},0)$. Then by the monotonicity of the warping coefficients it follows that on $M_{\infty}$ there exists a point $\tilde{z}_{\infty}\in S_{g_{\infty}(0)}(\Sigma_{\text{sing}},\varepsilon)$ such that $\cc_{\infty}(\tilde{z}_{\infty},0) \leq 0$, which is a contradiction. Thus both $\bb_{\infty}$ and $\cc_{\infty}$ vanish at $\Sigma_{\text{sing}}$, meaning that $M_{\infty} = \R^{4}$ and hence $\Sigma_{\text{sing}} = \origin_{\infty}$. From the same argument we deduce that the  radial coordinates
\[
s_{j}(\cdot) \doteq d_{g_{j}(0)}(\origin,\Phi_{j}(\cdot)).
\]
\noindent converge to $x_{\infty}$ in $C^{0}$ on compact sets.

%finally prove that we have convergence of the warping functions on compact sets. From now on we assume that $M_{\infty} = \R^{4}$, for the case $M_{\infty} = \R\times S^{3}$ is dealt with similarly. We denote by $\origin_{\infty}$ the trivial singular orbit in $M_{\infty}$ and we let $o_{j} = \Phi_{j}(\origin_{\infty})$. If, for a contradiction, there exists $\varepsilon > 0$ such that $d_{g_{j}(0)}(\origin,o_{j}) \geq 2\varepsilon$, then we could find points $\tilde{o}_{j}$ satisfying $d_{g_{j}(0)}(o_{j},\tilde{o}_{j}) = \varepsilon$ with $\cc_{j}(\tilde{o}_{j},0) \leq \cc_{j}(o_{j},0)$. Then by the monotonicity conditions in (iii), on the limit manifold $M_{\infty}$ there would be a point $\tilde{\origin}_{\infty}\in S_{g_{\infty}(0)}(\origin_{\infty},\varepsilon)$ such that $\cc_{\infty}(\tilde{\origin},0) \leq 0$, which is a contradiction. We may then define the radial coordinates:
%\[
%s_{j}(\cdot) \doteq d_{g_{j}(0)}(\origin,\Phi_{j}(\cdot)).
%\]
%\noindent According to the Cheeger-Gromov convergence $s_{j}$ converges to $x_{\infty}$ in $C^{0}$ on compact sets. 

We know that $\bb_{j}$ and $\cc_{j}$ converge to $\bb_{\infty}$ and $\cc_{\infty}$ respectively in $C^{0}$ on compact sets. Consider $0 < \delta < D$. Once again by \eqref{derivativesboundedCA} we see that $\bb_{j}(\cdot,0)$ is uniformly bounded from above on $B_{g_{j}(0)}(\origin,D)$. The cohomogeneity-1 action implies that $\bb_{\infty}$ is bounded away from zero in $B_{g_{\infty}(0)}(\origin_{\infty},D)\setminus B_{g_{\infty}(0)}(\origin_{\infty},\delta)$ thus yielding 
\[
\inf_{B_{g_{j}(0)}(\origin,D)\setminus B_{g_{j}(0)}(\origin,\delta)}\bb_{j}(\cdot,0)\geq \alpha(\delta,D) > 0.
\]
\noindent A similar estimate holds for $\cc_{j}$ as well. The latter bounds, along with \eqref{curvatureboundedCA} and Shi's derivative estimates yield
\[
\sup_{B_{g_{\infty}(0)}(\origin_{\infty},D)\setminus B_{g_{\infty}(0)}(\origin_{\infty},\delta)}\left\vert \nabla^{k}_{g_{\infty}(0)}s_{j}\right\vert \leq \alpha(k,\delta,D) < \infty,
\]
\noindent for any positive integer $k$. Therefore $s_{j}\rightarrow x_{\infty}$ smoothly on compact sets away from the origin. By a similar argument the $C^{0}$-convergence of $\cc_{j}\circ s_{j}$ to $\cc_{\infty}$ and $\bb_{j}\circ s_{j}$ to $\bb_{\infty}$ respectively are in fact smooth on compact sets away from the origin. One can finally check that $\xi_{j}\circ s_{j}$ converges smoothly on compact sets away from the origin to $\xi_{\infty} = \lvert \partial_{x_{\infty}}\rvert_{g_{\infty}}$. The boundary conditions at the origin and the uniform bounds on the curvature allow to extend the convergence at the singular orbit $\origin_{\infty}$ as well.
\end{proof}
\begin{remark}
According to the compactness property, whenever a family of warped Berger Ricci flows satisfies the assumptions as in the statement, then the pointed Cheeger-Gromov limit is attained by parametrizing the radial distance function.% and hence one can only work at the level of convergence of the warping functions as in \cite{appleton2}.
\end{remark}
We dedicate the end of this section to proving that as a consequence of the rotational symmetry type of estimates in Lemma \ref{firstorderestimates}, Corollary \ref{rotationalsymmetryboundsGk}, the scale-invariant lower bounds for the spatial derivative $\bb_{s}$ in Lemmas \ref{lowerboundforCA}, \ref{keylemmaGk} and the compactness result in Proposition \ref{compactnessresult} any Ricci flow solution considered so far has curvature uniformly bounded in the space-time.
\begin{proposition}\label{curvatureboundedintime}
If $(\R^{4},g(t))_{t\geq 0}$ is the maximal Ricci flow solution evolving from some $g_{0}$ belonging to either $\Gk$ or $\Gaf$, then 
\[
\limsup_{t\nearrow\infty}\left(\sup_{\R^{4}}\lvert \emph{Rm}_{g(t)}\rvert_{g(t)} \right) < \infty.
\]
\end{proposition}
\begin{proof} In the following we provide the details for the case $g_{0}\in\Gk$. The proof in the asymptotically flat setting only requires to replace the analysis of Section 4 with its counterpart in Section 3. Assume for a contradiction that the curvature becomes unbounded. The solution then develops a Type-II(b) singularity. Following \cite{IRF}[Chapter 8] we deduce that there exists a space-time sequence $(p_{j},t_{j})$, with $t_{j}\rightarrow \infty$, such that for all $\varepsilon > 0$ the Ricci flows $g_{j}(t) \doteq \lambda_{j}g(t_{j} + \lambda_{j}^{-1}t)$, where $\lambda_{j}=\lvert \text{Rm}\rvert(p_{j},t_{j})$, satisfy
\begin{equation}\label{adhoccurvaturebound}
\sup_{\R^{4}\times I}\lvert \text{Rm}_{g_{j}(t)}\rvert_{g_{j}(t)}\leq 1 + \varepsilon,
\end{equation}
\noindent for all $j\geq j_{0}(\varepsilon, I)$, with $I$ some interval. In particular, if the curvature diverges along some sequence of times, then we can choose space time points $(p_{j},t_{j})$ such that $\lambda_{j}\nearrow \infty$. We now check that we can apply the compactness result to the sequence of Ricci flows on a given interval $I\ni 0$. 
\\By Lemma \ref{consistencymonotone} the sequence of solutions has monotone coefficients. From \eqref{adhoccurvaturebound} and Lemma \ref{firstderivativesboundedGk} we also see that \eqref{curvatureboundedCA} and \eqref{derivativesboundedCA} respectively are satisfied. From Corollary \ref{curvaturecontrolledbycgk} it follows that $\cc(p_{j},t_{j})\rightarrow 0$ because $\lambda_{j}\rightarrow \infty$. Thus, we can use Corollary \ref{rotationalsymmetryboundsGk} to deduce that $\bb(p_{j},t_{j}) \rightarrow 0$ and hence that
\[
\bb^{2}_{j}(p_{j},0) = \lambda_{j}\bb^{2}(p_{j},t_{j}) = \lambda_{j}\cc^{2}(p_{j},t_{j})\uu^{-2}(p_{j},t_{j}) \leq \alpha(1 + \alpha\bb)^{2}(p_{j},t_{j}) \leq \alpha,
\]
\noindent for some $\alpha > 0$. Since the roundness ratio is scale invariant and $\cc(p_{j},t_{j})$ is converging to zero, we can again apply Corollary \ref{rotationalsymmetryboundsGk} to derive that $\uu(p_{j},t_{j}) \rightarrow 1$. Finally, Lemma \ref{keylemmaGk} implies that $\bb_{s}(p_{j},t_{j}) \geq \beta/2$ for $j$ large enough. 
\\Therefore, the assumptions in Proposition \ref{compactnessresult} are satisfied and we can hence apply a diagonal argument and pick a subsequence converging to an ancient Ricci flow solution $(M_{\infty},g_{\infty}(t),p_{\infty})_{t\leq 0}$ of the form
\[
g_{\infty}(t) = \xi_{\infty}(x_{\infty},t)dx^{2}_{\infty} + \bb^{2}_{\infty}(x_{\infty},t)\,(\sigma_{1}\otimes \sigma_{1} + \sigma_{2}\otimes \sigma_{2}) + \cc^{2}_{\infty}(x_{\infty},t)\,\sigma_{3}\otimes \sigma_{3}.
\]
\noindent The convergence of the warping coefficients and Corollary \ref{rotationalsymmetryboundsGk} yield:
\begin{align*}
\frac{1}{\cc_{\infty}}\left(1 - \uu_{\infty}\right)(q,t) &= \lim_{j\rightarrow \infty}\frac{1}{\cc_{j}}\left(1 - \uu_{j}\right)(s_{j}(q), t) \\ &= \lim_{j\rightarrow \infty}\frac{1}{\sqrt{\lambda_{j}}\cc}\left(1 - \uu\right)(s_{j}(q), t_{j} + \lambda_{j}^{-1}t)\leq \lim_{j\rightarrow\infty}\frac{\alpha}{\sqrt{\lambda_{j}}} = 0,
\end{align*}
\noindent for any $q\in M_{\infty}$ and $t\leq 0$. Therefore $(M_{\infty},g_{\infty}(t))$ is a $\kappa$-non collapsed - being the limit of a blow-up sequence - rotationally symmetric ancient solution to the Ricci flow and hence a $\kappa$-solution by \cite{zhang}. In particular, from Lemma \ref{keylemmaGk} we get
\[
(\bb_{\infty})_{s}(q,t) = \lim_{j\rightarrow \infty}(\bb_{s}\uu^{-1})(s_{j}(q),t_{j} + \lambda_{j}^{-1}t) \geq \beta > 0,
\]
\noindent where we have used that by the rotational symmetry of the limit the scale-invariant quantity $\uu(s_{j}(\cdot),t)$ is converging to 1. We conclude that the limit ancient Ricci flow is a non-compact $\kappa$-solution with positive asymptotic volume ratio. A rigidity result of Perelman \cite[Proposition 11.4]{pseudolocality} implies that $g_{\infty}(t)$ needs to be flat, which contradicts the choice of the factors $\lambda_{j}$.
% ancient solution to the Ricci flow positive asymptotic volume ratio.  
%Since by the monotonicity condition the Ricci flow does not contain closed geodesics, we see that the sequence $(\R^{4},g_{j}(t),p_{j})$ is non-collapsed. 

\end{proof}

\section{Ancient solutions opening faster than a paraboloid}
In this section we prove that the only complete warped Berger ancient solution with monotone coefficients, curvature uniformly bounded in the space-time, bounded Hopf-fiber and opening faster than a paraboloid in the directions orthogonal to the Hopf-fiber is the Taub-NUT metric. The main idea consists in showing that for any ancient solution belonging to the class just described the warping coefficient $\bb$ is actually a linear function of the distance in any region where the squashing factor $\uu$ is small. In other words, the ancient solution behaves exactly as the Taub-NUT metric along the $S^{2}$-directions whenever $\bb$ is large. Once such control is available, one can consider scale-invariant first-order quantities derived from the hyperk\"ahler odes \eqref{ODE1}, \eqref{ODE2}, and prove that they have a sign on the ancient solution, the aim being to finally show that \eqref{ODE2} is in fact satisfied in the space-time.

It is worth outlining a strategy we often use in the following which was adopted in \cite{appleton2} to prove a uniqueness result for the Eguchi-Hanson metric, with a substantial difference given by the assumption of $\kappa$-non-collapsedness as we describe below. 

Suppose that we are given a geometric scale-invariant quantity $\mathcal{L}$ and that we want to prove that $\mathcal{L} \geq 0$ in the space-time $\R^{4}\times (-\infty,0]$. Once we know that $\mathcal{L}$ is uniformly bounded and that, say, $\partial_{t}\mathcal{L} > 0$ at any negative minimum, one might try to apply a maximum principle argument to the evolution equation of $\mathcal{L}$. In general though, the infimum of $\mathcal{L}$ may not be achieved. In this case, one could consider a space-time sequence $(p_{j},t_{j})$ such that $\mathcal{L}(p_{j},t_{j}) \rightarrow \inf \mathcal{L} < 0$ and define the Ricci flow sequence $g_{j}(t) = g(t_{j} + t)$, centred at $p_{j}$. \emph{If} the compactness result in Proposition \ref{compactnessresult} holds, then on the limit ancient Ricci flow the analogous quantity $\mathcal{L}_{\infty}$ achieves its (negative) infimum in the space-time, thus allowing to rely on maximum principle arguments to obtain the contradiction. 

This is exactly the strategy used in \cite{appleton2}. Since in \cite{appleton2} the ancient solutions analysed are all \emph{non}-collapsed, the roundness ratio is \emph{a priori} controlled from below away from the singular orbit uniformly in time and hence the compactness property can be applied for any sequence of Ricci flows as above. On the contrary, in our \emph{collapsed} setting we always need to verify that the squashing factor $\uu$ stays away from zero along the space-time sequence $(p_{j},t_{j})$ used to approximate the infimum of $\mathcal{L}$ so that the compactness result can indeed be used. 

The latter represents the main difficulty when analysing the collapsed case and the requirement on the ancient solution to open up faster than a paraboloid along the $S^{2}$-directions allows us to bypass this issue. In fact, the application of the compactness result in Proposition \ref{compactnessresult} yields that for any ancient solution opening faster than a paraboloid the hyperk\"ahler quantity $\oded$ given in \eqref{ODE2} is nonnegative. On the gradient steady soliton found by Appleton instead we find that $\oded$ approaches its infimum -2 at spatial infinity on any time-slice, thus along space-time sequences where the squashing factor $\uu$ become degenerate. 

In order to ease the notations, we give the following:
\vspace{0.1in}
\begin{definition}\label{definitionancientpro} Let $m > 0$. The class $\mathcal{A}$ consists of all complete, warped Berger ancient solutions to the Ricci flow with monotone coefficients and curvature uniformly bounded in the space-time, satisfying
\begin{align}
\inf_{\R^{4}\times (-\infty,0]}\,\,\frac{\bb_{s}\uu^{-1}}{f(\uu^{-1})} &> 0,  \label{conditionforrigidity}
\\ \sup_{\R^{4}\times (-\infty,0]}\cc = m^{-1}, \label{boundedhopffberdefn}
\end{align}
\noindent for some continuous positive function $f$ such that $f(z)\rightarrow\infty$ as $z\rightarrow\infty$.
\end{definition}
\begin{remark} We again point out that \eqref{conditionforrigidity} means that the warping coefficient $\bb$ opens faster than a paraboloid in $\R^{3}$ on any time-slice. In particular, the volume of geodesic balls $B_{g(t)}(\origin,r)$ grows faster than $r^{2}$. However, \emph{a priori} there is no upper bound for the volume growth.
\end{remark}
We start by proving that the first order derivatives are uniformly bounded. In fact, from the following estimate we also derive that $\cc_{s}$ decays at some rate in any space-time region where the squashing factor is small.  
\begin{lemma}\label{firstderivativesancientbounded} If $(\R^{4},g(t))_{t\leq 0}$ is an ancient solution in $\mathcal{A}$, then 
\[
2\bb_{s} + \cc_{s}\uu^{-1} - 4 \leq 0.
\]
\end{lemma}
\begin{proof} Let $h$ denote the quantity $2\bb_{s} + \cc_{s}\uu^{-1} - 4$. By the boundary conditions we see that $h(\origin,t) = -1$. Thus, if $h$ is positive somewhere in the space-time, then there exist $p\in \R^{4}$ and $t\leq 0$ such that $h(p,t) > 0$ and $h_{s}(p,t) > 0$. The latter condition implies 
\[
\left(2\bb_{ss} + \cc_{ss}\uu^{-1} + \cc_{s}\left(\frac{\bb_{s}}{\cc} - \frac{\cc_{s}\uu^{-1}}{\cc}\right)\right)(p,t) > 0.
\]
\noindent Therefore, the scalar curvature \eqref{scalarcurvature} is bounded from above as follows:
\[
R(p,t) < \frac{2}{\bb^{2}}\left(-\bb_{s}\cc_{s}\uu^{-1} - (\cc_{s}\uu^{-1})^{2} + 4 -\uu^{2} -\bb_{s}^{2} \right)(p,t).
\]
\noindent Since $h (p,t) > 0$, we get
\[
R(p,t) < \frac{2}{\bb^{2}}\left(- 4 -\frac{3}{4}(\cc_{s}\uu^{-1})^{2} + 4 -\uu^{2}\right)(p,t) < 0.
\]
\noindent However, according to \cite{chen2009} any complete ancient solution to the Ricci flow has nonnegative scalar curvature. We conclude that $h \leq 0$ in the space-time. 
\end{proof}
Since we aim to prove that $\oded \geq 0$ for any ancient solution in $\mathcal{A}$, let us consider the evolution equation of $\oded = \bb_{s} + \uu - 2$ at any negative minimum:
\begin{align*}
\partial_{t}\,\oded|_{\text{min} < 0} &\geq \frac{1}{\bb^{2}}\left(-(\cc_{s}\uu^{-1})^{2}(2 + \oded) + \cc_{s}(8 + 4\oded) \right) \\ &+ \frac{1}{\bb^{2}}\left(-\oded^{3} - 6\oded^{2} -8\oded - 3\oded\uu^{2} - 6\uu^{2} \right).
\end{align*}
\noindent We simplify the evolution equation by introducing $z = \oded + 2 > 0$. Then we obtain
\begin{equation}\label{evolutionodednice}
\partial_{t}\,\oded|_{\text{min} < 0} \geq -\frac{z}{\bb^{2}}\left((\cc_{s}\uu^{-1})^{2} - 4\cc_{s} + z^{2} - 4 + 3\uu^{2}\right).
\end{equation}
\noindent In order to prove that the $\cc_{s}$-quadratic is always negative along an ancient solution in $\mathcal{A}$, one needs to control $\cc_{s}$ in terms of $\bb_{s}$ and $\uu$ in a precise way everywhere in the space-time. To this aim, we first show that $\cc_{s}$ decays to zero at the same rate given by the hyperk\"ahler quantity \eqref{ODE1} for any solution in $\mathcal{A}$.
\begin{lemma}\label{lemmaforARF1}
If $(\R^{4},g(t))_{t\leq 0}$ is an ancient solution in $\mathcal{A}$, then 
\[
\cc_{s} - 2\uu^{2} \leq 0.
\]
\end{lemma}
\begin{proof}
Let $\psi \doteq \cc_{s} - 2\uu^{2}$ and let $\Psi_{\infty} \doteq \sup_{\R^{4}\times (-\infty,0]}\psi$. From Lemma \ref{firstderivativesancientbounded} it follows that $\Psi_{\infty}$ is bounded and we can hence assume for a contradiction that $\Psi_{\infty} > 0$. By direct computation we check that the evolution equation of $\psi$ at any positive maximum is given by
\begin{align*}
\partial_{t}\,\psi|_{\text{max}} \leq \frac{1}{\bb^{2}}\left(-\cc_{s}\left(6\uu^{2} + 2\bb_{s}^{2}\right) + 8\bb_{s}\uu^{3} + \uu^{2}\left(8\bb_{s}^{2} - 8\bb_{s}\cc_{s}\uu^{-1} -16(1 - \uu^{2})\right)\right).
\end{align*}
\noindent Since $\cc_{s} > 2\uu^{2}$ at any positive maximum of $\psi$, we find
\begin{align*}
\partial_{t}\,\psi|_{\text{max}} & < \frac{1}{\bb^{2}}\left(4\bb_{s}^{2}\uu^{2} - 8\bb_{s}\uu^{3} - 16\uu^{2} + 4\uu^{4}\right) \\ &\leq \frac{4\uu^{2}}{\bb^{2}}\left(\bb_{s}^{2} -2\bb_{s}\uu - 4 + \uu^{2}\right).
\end{align*}
\noindent Finally, we note that the quadratic on the right hand side is always negative because $0 \leq \bb_{s} \leq 2$ by Lemma \ref{firstderivativesancientbounded}. Therefore, we have shown that $\partial_{t}\psi < 0$ at any positive maximum. Let now $(p_{j},t_{j})$ be a space-time sequence satisfying
$\psi(p_{j},t_{j}) \rightarrow \Psi_{\infty}$. From Lemma \ref{firstderivativesancientbounded} we derive
\[
4 \geq \cc_{s}\uu^{-1}(p_{j},t_{j}) \geq \frac{\Psi_{\infty}}{2}\uu^{-1}(p_{j},t_{j}),
\]
\noindent for any $j$ large enough. Thus $\uu(p_{j},t_{j}) \geq \Psi_{\infty}/8$ for $j$ large enough. Since $\cc$ is uniformly bounded in the space-time, the latter also yields $\bb(p_{j},t_{j}) \leq 8m^{-1}/\Psi_{\infty}$. Moreover, by the uniform lower bound for $\uu$ and \eqref{conditionforrigidity} we obtain
$\bb_{s}(p_{j},t_{j}) \geq \varepsilon$, for some $\varepsilon > 0$. We can then apply the compactness result in Proposition \ref{compactnessresult} to the sequence $(\R^{4},g_{j}(t),p_{j})_{t\leq 0}$, with $g_{j}(t) \doteq g(t_{j} + t)$, and deduce that there exists a warped Berger ancient solution $(M_{\infty},g_{\infty}(t),p_{\infty})_{t\leq 0}$, with $M_{\infty} = \R^{4}$ or $M_{\infty} = \R\times S^{3}$, satisfying:
\[
\psi_{\infty}(p_{\infty},0) = ((\cc_{\infty})_{s} - 2\uu_{\infty}^{2})(p_{\infty},0) = \sup_{M_{\infty}\times (-\infty,0]}\,\psi_{\infty} = \Psi_{\infty} > 0.
\] 
\noindent However, from the previous calculations we see that 
\[
\partial_{t}\,\psi_{\infty}|_{\Psi_{\infty}} < 0,
\]
\noindent hence arriving to a contradiction.
% By direct computation we see that the last polynomial is negative whenever $\bb_{s}\in (y_{1},y_{2})$ with 
%\begin{align*}
%y_{1} &= \frac{28\frac{\cc}{\bb} - \sqrt{576 + 640\frac{\cc^{2}}{\bb^{2}}}}{12} \leq \frac{28\frac{\cc}{\bb} - \sqrt{784\frac{\cc^{2}}{\bb^{2}}}}{12} \leq 0 
%\\ y_{2} &= \frac{28\frac{\cc}{\bb} + \sqrt{576 + 640\frac{\cc^{2}}{\bb^{2}}}}{12} > \frac{\sqrt{576}}{12} = 2.
%\end{align*}
%\noindent Therefore the right hand side of the evolution equation is always negative whenever $0 \leq \bb_{s}(\cdot,t) \leq 2$.
\end{proof}
While the previous Lemma gives us a precise upper bound for $\cc_{s}$ in terms of $\uu$, in order to control \eqref{evolutionodednice} at any negative minimum we also need a lower bound for $\cc_{s}\uu^{-1}$. Thus, we consider the difference between $\oded$ and $\ode\uu^{-1}$.
\begin{lemma}\label{lemmaforARF2}
If $(\R^{4},g(t))_{t\leq 0}$ is an ancient solution in $\mathcal{A}$, then 
\[
\oded - \ode\uu^{-1} = \bb_{s} -\cc_{s}\uu^{-1} -2(1 - \uu) \leq 0.
\]
\end{lemma}
\begin{proof}
Let $\phi$ denote $\oded - \ode\uu^{-1}$. By Lemma \ref{firstderivativesancientbounded} we get that $\phi$ is uniformly bounded in the space-time. Suppose for a contradiction that $\Phi_{\infty}\doteq \sup_{\R^{4}\times (-\infty,0]}\phi$ is positive. Given a space-time sequence $(p_{j},t_{j})$ such that $\phi(p_{j},t_{j})\rightarrow \Phi_{\infty}$, from Lemma \ref{firstderivativesancientbounded} we find
\begin{align*}
\frac{\Phi_{\infty}}{2} \leq \phi(p_{j},t_{j}) &= \left(\bb_{s} + \frac{\cc_{s}\uu^{-1}}{2} - 2 -\frac{3}{2}\cc_{s}\uu^{-1} + 2\uu\right)(p_{j},t_{j}) \\ &\leq \left(-\frac{3}{2}\cc_{s}\uu^{-1} + 2\uu\right)(p_{j},t_{j})
\end{align*}
\noindent Thus $\uu(p_{j},t_{j})$ is uniformly bounded along the sequence. We may then use \eqref{conditionforrigidity} as in the proof of Lemma \ref{lemmaforARF1} to deduce that the compactness result applies to the sequence of ancient Ricci flows translated by times $t_{j}$ and centred at $p_{j}$. In particular, there exists a limit warped Berger ancient Ricci flow $(M_{\infty},g_{\infty}(t),p_{\infty})_{t\leq 0}$ such that 
\[
\phi_{\infty}(p_{\infty},0) = \left((\bb_{\infty})_{s} -(\cc_{\infty})_{s}\uu_{\infty}^{-1} -2(1 - \uu_{\infty})\right)(p_{\infty},0)  = \sup_{M_{\infty}\times (-\infty,0]}\,\phi_{\infty} = \Phi_{\infty} > 0.
\]
\noindent Thus, it remains to check that $\phi_{\infty}$ cannot achieve a positive supremum along a warped Berger ancient solution as in $\mathcal{A}$. In fact, we show that for any complete, warped Berger ancient Ricci flow with monotone coefficients $\phi_{\infty}$ never attains its positive supremum in the space-time.
\\We compute the evolution equation of $\phi_{\infty}$ at a positive maximum and we drop the $\infty$-subscript from the notation:
\begin{align*}
\partial_{t}\phi|_{\text{max}} &\leq \frac{1}{\bb^{2}}\left(\bb_{s}\left(4 - \bb_{s}^{2} - (\cc_{s}\uu^{-1})^{2} - 14\uu^{2}\right)\right) \\  &+ \frac{1}{\bb^{2}}\left(\cc_{s}\uu^{-1}\left(\bb_{s}^{2}+ (\cc_{s}\uu^{-1})^{2} + 4 + 6\uu^{2}\right)\right) \\
&+ \frac{1}{\bb^{2}}\left(2\uu\left(-3\bb_{s}^{2} -(\cc_{s}\uu^{-1})^{2} + 4\bb_{s}\cc_{s}\uu^{-1} + 4(1-\uu^{2})\right)\right).
\end{align*}
\noindent At any positive maximum of $\phi$ we can bound the evolution equation by
\begin{align*}
\partial_{t}\phi|_{\text{max}}&\leq \frac{1}{\bb^{2}}\left(4(1-\uu)\left(-\left(\cc_{s}\uu^{-1} - \uu\right)^{2}\right)\right) \\ & + \frac{\phi}{\bb^{2}}\left(-\phi^{2} -\phi(6 + 2\cc_{s}\uu^{-1}) - 2(\cc_{s}\uu^{-1})^{2} -8\cc_{s}\uu^{-1} + 4\cc_{s} -8 -2\uu^{2} \right) < 0.
\end{align*}
\noindent Therefore, given a warped Berger ancient solution with monotone coefficients the geometric quantity $\phi$ cannot achieve its supremum in the space-time. This completes the proof.
\end{proof}
We may now go back to the evolution equation of $\oded$ at a negative minimum \eqref{evolutionodednice}. The roots of the $\cc_{s}$-quadratic are
\[
y_{\pm} = \uu^{2}\left(2 \pm \sqrt{1 + \uu^{-2}(4 - z^{2})}\right).
\] 
\noindent From Lemma \ref{lemmaforARF1} we immediately derive that
\[
\cc_{s} \leq 2\uu^{2} < y_{+},
\]
\noindent for any Ricci flow in $\mathcal{A}$. According to Lemma \ref{lemmaforARF2}, in order to prove that $\cc_{s} > y_{-}$ everywhere in the space-time, it suffices to show that
\begin{equation}\label{quadratic}
y_{-} < \bb_{s}\uu -2\uu(1-\uu).
\end{equation}
\noindent The latter is equivalent to 
\[
1 + \uu^{-1}(2-z) < \sqrt{1 + \uu^{-2}(4 - z^{2})}.
\]
\noindent After taking the square of the equation and rearranging the terms, we see that \eqref{quadratic} holds if and only if
\[
2\uu^{-1}(2 - z)(1 - z\uu^{-1}) < 0,
\]
\noindent which is indeed satisfied because by definition $z\uu^{-1} = (\bb_{s} + \uu)\uu^{-1} > 1$ and at any negative minimum of $\oded$ we have $z < 2$. Therefore, we conclude that $\cc_{s}\in (y_{-},y_{+})$ in the space-time. 

To sum up, we have shown that:
\begin{lemma}\label{lemmastationaryoded}
Given an ancient solution to the Ricci flow in $\mathcal{A}$, the evolution equation of $\oded$ at any negative minimum satisfies
\[
\partial_{t}\,\oded|_{\text{min} < 0} > 0.
\] 
\end{lemma}
The final and most difficult step consists in proving that, up to passing to a subsequence, the hyperk\"ahler quantity $\oded$ always attains its infimum on any ancient solution in $\mathcal{A}$. We emphasize that any conclusion achieved so far does extend to the steady soliton found by Appleton in \cite{appleton2}. Explicitly, one may check that in order to apply the compactness result for proving Lemmas \ref{lemmaforARF1} and \ref{lemmaforARF2} it suffices to require $\bb_{s}\uu^{-1}$ to be uniformly bounded from below in the space-time, which holds along the soliton because it opens as fast as a paraboloid at spatial infinity. However, along the soliton the infimum of $\oded$  is $-2$ and is never attained in the space-time. Indeed, along any space-time sequence $(p_{j},t_{j})$ satisfying $\oded(p_{j},t_{j})\rightarrow -2$ the roundness ratio $\uu$  becomes degenerate. Thus, we see once again that in the collapsed setting the main issue consists in verifying that geometric quantities do attain their infimum (supremum) in regions of the space-time where the squashing factor stays positive. This is where the condition \eqref{conditionforrigidity}, with $f$ diverging to infinity when $\uu^{-1}\rightarrow \infty$, plays a role via a sort of approximation method to show that, in fact, for any solution in $\mathcal{A}$ we have
\[
\lim_{\uu^{-1}\rightarrow \infty}\oded = 0.
\]
\noindent Equivalently, below we prove that for any ancient solution in $\mathcal{A}$ the warping coefficient $\bb$ in the directions orthogonal to the Hopf-fiber grows linearly in space, meaning that the volume of geodesic balls $B_{g(t)}(\origin,r)$ behaves like $r^{3}$ for any radius $r$ large enough and on any time-slice. 
\begin{proposition}\label{keypropositionancient}
If $(\R^{4},g(t))_{t\leq 0}$ is an ancient solution to the Ricci flow in $\mathcal{A}$, then there exists $\delta > 0$ satisfying
\[
\inf_{\R^{4}\times (-\infty, 0]}\,\uu^{-\delta}(\bb_{s} - 2) > -\infty.
\]
\end{proposition}
\begin{proof} For any $\delta,\epsilon > 0$ we define 
\[
\omega_{\delta,\epsilon} = \uu^{-\delta}\left(\epsilon\bb_{s}\uu^{-1} + \bb_{s} - 2\right), 
\]
\noindent and 
\[
\Omega_{\delta,\epsilon} = \inf_{\R^{4}\times (-\infty,0]}\,\omega_{\delta,\epsilon}.
\]
\noindent In the following we always take $\Omega_{\delta,\epsilon}$ to be \emph{negative}. From \eqref{conditionforrigidity} we see that there exists $\beta > 0$ such that
\[
\omega_{\delta,\epsilon} \geq \uu^{-\delta}\left(\epsilon(\beta f(\uu^{-1})) + \bb_{s} - 2\right),
\]
\noindent with $f(\uu^{-1})\rightarrow \infty$ as $\uu^{-1}\rightarrow \infty$. Thus, given $\delta$ and $\epsilon$ positive constants, the quantity $\Omega_{\delta,\epsilon}$ is finite. Given a negative minimum point for $\omega_{\delta,\epsilon}$, the evolution equation becomes
\begin{align*}
\partial_{t}\omega_{\delta,\epsilon}|_{\text{min} < 0} &\geq \frac{\omega_{\delta,\epsilon}}{\bb^{2}}\left( \bb_{s}^{2}(\delta^{2} + 2\delta +2\delta\frac{\epsilon\uu^{-1}}{1 + \epsilon\uu^{-1}}) + (\cc_{s}\uu^{-1})^{2}(\delta^{2} + 2\delta\frac{\epsilon\uu^{-1}}{1 + \epsilon\uu^{-1}})\right)\\ & + \frac{\omega_{\delta,\epsilon}}{\bb^{2}}\left(\bb_{s}\cc_{s}\uu^{-1}(-2\delta^{2} -2\delta -4\delta\frac{\epsilon\uu^{-1}}{1 + \epsilon\uu^{-1}}) -4\delta(1-\uu^{2})\right) \\ &+ \frac{\epsilon\bb_{s}\uu^{-1-\delta}}{\bb^{2}}\left(2\bb_{s}^{2}\frac{\epsilon\uu^{-1}}{1 + \epsilon\uu^{-1}} + (\cc_{s}\uu^{-1})^{2}(-2 + 2\frac{\epsilon\uu^{-1}}{1 + \epsilon\uu^{-1}})\right) \\ &+\frac{\epsilon\bb_{s}\uu^{-1-\delta}}{\bb^{2}}(-4\frac{\epsilon\uu^{-1}}{1 + \epsilon\uu^{-1}}\bb_{s}\cc_{s}\uu^{-1} -2\uu^{2}) \\& +\frac{1}{\bb^{2}}\left(4\epsilon\cc_{s}\uu^{-\delta} + 4\cc_{s}\uu^{1-\delta} + \bb_{s}\uu^{-\delta}(4 - \bb_{s}^{2} -(\cc_{s}\uu^{-1})^{2} -6\uu^{2})\right).
\end{align*}
\noindent Since $\bb_{s}$ and $\cc_{s}$ are nonnegative and $\omega_{\delta,\epsilon} < 0$, we may bound the right hand side by:
\begin{align*}
\partial_{t}\omega_{\delta,\epsilon}|_{\text{min} < 0} &\geq \frac{\delta\omega_{\delta,\epsilon}}{\bb^{2}}\left( \bb_{s}^{2}(\delta + 2 +2\frac{\epsilon\uu^{-1}}{1 + \epsilon\uu^{-1}}) + (\cc_{s}\uu^{-1})^{2}(\delta + 2\frac{\epsilon\uu^{-1}}{1 + \epsilon\uu^{-1}}) -4(1-\uu^{2})\right)\\ &+ \frac{\epsilon\bb_{s}\uu^{-1-\delta}}{\bb^{2}}\left(-2(\cc_{s}\uu^{-1})^{2} -4\frac{\epsilon\uu^{-1}}{1 + \epsilon\uu^{-1}}\bb_{s}\cc_{s}\uu^{-1} -2\uu^{2}\right) \\& +\frac{1}{\bb^{2}}\left(\bb_{s}\uu^{-\delta}(4 - \bb_{s}^{2} -(\cc_{s}\uu^{-1})^{2} -6\uu^{2})\right).
\end{align*}
\noindent We note that 
\[
\bb_{s}\uu^{-\delta}(4-\bb_{s}^{2}) = (\bb_{s}^{2} + 2\bb_{s})(-\omega_{\delta,\epsilon} +\epsilon\bb_{s}\uu^{-1-\delta}) \geq -(\bb_{s}^{2} + 2\bb_{s})\omega_{\delta,\epsilon}.
\]
\noindent Since by Lemma \ref{firstderivativesancientbounded} $\bb_{s}\leq 2$, if we take $\delta$ positive such that $2-\delta^{2}-4\delta \geq 0$, then we can write
\begin{align*}
\partial_{t}\omega_{\delta,\epsilon}|_{\text{min} < 0} &\geq \frac{\delta\lvert\omega_{\delta,\epsilon}\rvert}{\bb^{2}}\left((\cc_{s}\uu^{-1})^{2}(-\delta - 2\frac{\epsilon\uu^{-1}}{1 + \epsilon\uu^{-1}}) + 4(1-\uu^{2})\right)\\ &+ \frac{\epsilon\bb_{s}\uu^{-1-\delta}}{\bb^{2}}\left(-2(\cc_{s}\uu^{-1})^{2} -4\frac{\epsilon\uu^{-1}}{1 + \epsilon\uu^{-1}}\bb_{s}\cc_{s}\uu^{-1} -2\uu^{2}\right) \\& +\frac{1}{\bb^{2}}\left(-\bb_{s}\uu^{-\delta}((\cc_{s}\uu^{-1})^{2} + 6\uu^{2})\right).
\end{align*}
\noindent By Lemma \ref{lemmaforARF1} we see that $\lvert \cc_{s}\uu^{-1}\rvert \leq 2\uu$. Moreover, whenever $\omega_{\delta,\epsilon} < 0$ we have $\epsilon\bb_{s}\uu^{-1} < 2$. Since by the condition we set previously $\delta < 1$, we can bound the evolution equation by
\begin{align*}
\partial_{t}\omega_{\delta,\epsilon}|_{\text{min} < 0} &\geq \frac{\delta\lvert\omega_{\delta,\epsilon}\rvert}{\bb^{2}}\left(-12 \uu^{2} + 4(1-\uu^{2})\right)\\ &+ \frac{2\uu^{-\delta}}{\bb^{2}}\left(-8\uu^{2} -16\uu -2\uu^{2}\right) \\& +\frac{1}{\bb^{2}}\left(-2\uu^{-\delta}(4\uu^{2} + 6\uu^{2})\right).
\end{align*}
\noindent Finally, we observe that $\lvert \omega_{\delta,\epsilon}\rvert \leq 2\uu^{-\delta}$ at any negative minimum. Therefore, from the evolution equation we deduce that if $\lvert \omega_{\delta,\epsilon} \rvert = \Omega_{\delta}$, for some $\Omega_{\delta}$ large enough and \emph{independent} of $\epsilon$, then 
\[
\partial_{t}\omega_{\delta,\epsilon}|_{\text{min} = -\Omega_{\delta} < 0} \geq \frac{1}{\bb^{2}}\left(\delta\Omega_{\delta} - 1\right) > 0.
\]
\noindent Let $(p_{j},t_{j})$ be a space-time sequence satisfying $\omega_{\delta,\epsilon}(p_{j},t_{j})\rightarrow \Omega_{\delta,\epsilon} < 0$. According to \eqref{conditionforrigidity} the squashing factor is bounded away from zero by some positive quantity depending on $\epsilon$ along the given sequence. From \eqref{boundedhopffberdefn} we then derive that $\bb(p_{j},t_{j}) \leq \alpha(\epsilon) < \infty$. Again by the constraint in \eqref{conditionforrigidity} the spatial derivative $\bb_{s}$ has a uniform lower bound along the sequence because $\uu$ is non-degenerate. Since the ancient solution belongs to $\mathcal{A}$ we deduce that we may apply the compactness result in Proposition \ref{compactnessresult} to the sequence $(\R^{4},g(t_{j} + t),p_{j})_{t\leq 0}$ and conclude that there exists a warped Berger ancient solution $(M_{\infty},g_{\infty}(t),p_{\infty})_{t\leq 0}$ such that
\[
(\omega_{\delta,\epsilon})_{\infty}(p_{\infty},0) = \inf_{M_{\infty}\times (-\infty,0]}\,(\omega_{\delta,\epsilon})_{\infty} = \Omega_{\delta,\epsilon} < 0.
\]
\noindent From the previous analysis we derive that there exists $\Omega_{\delta} > 0$ such that 
\[
\partial_{t}(\omega_{\delta,\epsilon})_{\infty}(p_{\infty},0) > 0,
\]
\noindent when $(\omega_{\delta,\epsilon})_{\infty}(p_{\infty},0) < -\Omega_{\delta}$.

Thus, fixed $\delta$ as above, we see that $\Omega_{\delta,\epsilon} \geq -\lvert \Omega_{\delta}\rvert$, for all $\epsilon > 0$. We then let $\epsilon\searrow 0$ and conclude that there exists $\delta > 0$ such that
\[
\omega_{\delta,0} = \uu^{-\delta}(\bb_{s} - 2) \geq -\lvert \Omega_{\delta}\rvert > -\infty.
\]
\end{proof}
We now have all the ingredients to prove Theorem \ref{maintheoremancient}. 
%Namely, we show the following:
%\begin{theoremm}\label{rigiditytheoremsection6}
%If $(\R^{4},g(t))_{t\leq 0}$ is an ancient solution to the Ricci flow in $\mathcal{A}$, then $(\R^{4},g(t))$ is the Taub-NUT metric of mass $m$.
%\end{theoremm}
\begin{proof}[Proof of Theorem \ref{maintheoremancient}] By Proposition \ref{keypropositionancient} we know that there exist $\delta, \Omega_{\delta} > 0$ such that
\begin{equation}\label{lowerboundoded}
\oded \geq \uu -\Omega_{\delta}\,\uu^{\delta}.
\end{equation}
\noindent Let us set the notation $\inf\,\oded \doteq \mathcal{J}_{2}$ and assume for a contradiction that $\mathcal{J}_{2} < 0$. Given $(p_{j},t_{j})$ such that $\oded(p_{j},t_{j})\rightarrow \mathcal{J}_{2}$, , by \eqref{lowerboundoded} we see that 
\[
\uu(p_{j},t_{j})^{\delta} \geq \frac{\lvert \mathcal{J}_{2}\rvert}{2\Omega_{\delta}},
\] 
\noindent for all $j$ large enough. We may then argue as for the proof of Proposition \ref{keypropositionancient} and derive that there exists a limit warped Berger solution $(M_{\infty},g_{\infty}(t),p_{\infty})_{t\leq 0}$ such that
\[
(\oded)_{\infty}(p_{\infty},0) = \inf_{M_{\infty}\times (-\infty,0]}(\oded)_{\infty} = \mathcal{J}_{2} < 0.
\]
\noindent However, Lemma \ref{lemmastationaryoded} implies that $\partial_{t}(\oded)_{\infty}(p_{\infty},0) > 0$. Therefore, for any ancient solution in $\mathcal{A}$ we have $\oded \geq 0$.
\begin{claim}\label{claimricciflatnessj2}
Let $(\R^{4},g(t))_{t\leq 0}$ be a complete, bounded curvature warped Berger ancient solution to the Ricci flow. Assume that there exist $r > 0$, $t_{0}\leq 0$ such that
\[
\oded(\cdot,t_{0})|_{B_{g(t_{0})}(\origin,r)} \geq 0.
\]
\noindent Then $g(t) \equiv g$ is Ricci-flat.
\end{claim}
\begin{proof}[Proof of Claim \ref{claimricciflatnessj2}] 
Since the curvature is bounded, we may apply l'H\^opital rule and find that the scalar curvature at the origin is given by:
\[
R(\origin,t) = -4(\cc_{sss}(\origin,t) + 2\bb_{sss}(\origin,t)).
\]
\noindent In particular, we derive that
\[
(\uu)_{ss}(\origin,t) = \frac{1}{3}\left(\cc_{sss}(\origin,t) - \bb_{sss}(\origin,t)\right),
\]
\noindent Therefore, we get
\[
(\oded)_{s}(\origin,t) = 0, \,\,\,\,\, (\oded)_{ss}(\origin,t) = \frac{1}{3}\left(2\bb_{sss}(\origin,t) + \cc_{sss}(\origin,t)\right) = -\frac{R(\origin,t)}{12}.
\]
\noindent From the assumption we deduce that the origin must be a local minimum for $\oded(\cdot,t_{0})$, meaning that $(\oded)_{ss}(\origin,t_{0}) \geq 0$ and hence $R(\origin,t_{0}) \leq 0$. By Chen \cite{chen2009} and a standard application of the strong maximum principle to the Ricci flow we conclude that the ancient solution is in fact a stationary Ricci flat metric.
% uniqueness result we know that any ancient solution has nonnegative scalar curvature; we may then apply the maximum principle and conclude that the Ricci tensor vanishes everywhere in the space-time.
\end{proof}
By Claim \ref{claimricciflatnessj2} any ancient solution in $\mathcal{A}$ is Ricci-flat. Accordingly, we drop the time-dependence and we simply write $g$. Suppose for a contradiction that $\oded > 0$ somewhere. From the boundary conditions we get that there exists $p\in\R^{4}$ such that $\oded(p) > 0$ and $(\oded)_{s}(p) > 0$. Thus
\[
\left(\bb_{ss} + \frac{\cc_{s}}{\bb} -\bb_{s}\frac{\uu}{\bb}\right)(p) > 0,
\] 
\noindent and $(\bb_{s} + \uu)(p) > 2$. We then obtain
\begin{align*}
\text{Ric}\left(X_{1},X_{1}\right)(p) &= (-\bb\bb_{ss} -\bb_{s}\cc_{s}\uu^{-1} -\bb_{s}^{2} -2\uu^{2} + 4)(p) \\ & < (\cc_{s} -\bb_{s}\uu -\bb_{s}\cc_{s}\uu^{-1} -\bb_{s}^{2} - 2\uu^{2} + 4)(p) \\ & < (2\cc_{s}(1 - \uu^{-1}) + 2\uu(1-\uu))(p) < 0,  
\end{align*}
\noindent where the last inequality follows from Lemma \ref{lemmaforARF2}. Therefore $\oded = 0$ everywhere in the space-time. Again, by Lemma \ref{lemmaforARF2} we see that $\ode \geq 0$ and a similar argument yields $\ode = 0$. Since the Hopf-fiber is uniformly bounded for any ancient solution in $\mathcal{A}$ and the differential equations \eqref{ODE1} and \eqref{ODE2} are satisfied in the space-time, we can conclude that $g$ is the Taub-NUT metric of mass $m$, with $m$ given by \eqref{boundedhopffberdefn}.
\end{proof}
% if $\Omega_{\delta,\epsilon} < 0$ with $\lvert\Omega_{\delta,\epsilon}\rvert > \Omega_{\delta}$, with $\Omega_{\delta}
%\begin{proof}[Proof of Theorem \ref{maintheoremancient}]
%We first prove that the first order derivatives are uniformly bounded. In fact, from the following estimate we also derive that $\cc_{s}$ decays at some rate in any space-time region where the squashing factor is small. 
%\begin{claim} $2\bb_{s} + \cc_{s}\uu^{-1} - 4 \leq 0$ in $\R^{4}\times (-\infty,1)$.
%\end{claim}
%\begin{proof}[Proof of Claim 

%\end{proof}

%\end{proof}
 
\section{Long-time behaviour of the Ricci flow}
In this section we apply the compactness property in Proposition \ref{compactnessresult} and the rigidity result in Theorem \ref{maintheoremancient} to study the long-time behaviour of Ricci flow solutions in $\Gk$ and in $\Gaf$. In particular, we show that any solution in $\Gk$ encounters a Type-II(b) singularity in infinite-time, which is modelled by the Taub-NUT metric in a precise way. Namely, any solution in $\Gk$ converges to $\nut$ in the Cheeger-Gromov sense.

We recall that an immortal Ricci flow solution $(M,g(t))_{t\geq 0}$ converges to a stationary Ricci-flat metric $(M_{\infty},g_{\infty})$ in the pointed Cheeger-Gromov sense in infinite time if there exist $p\in M$ and $p_{\infty}\in M_{\infty}$ such that for any sequence $t_{j}\nearrow \infty$ the pointed sequence $(M,g(t_{j}+t),p)$ converges to $(M_{\infty},g_{\infty},p_{\infty})$ in the Cheeger-Gromov sense. We point out that for this notion of convergence we do \emph{not} rescale the immortal solution, so that if $g_{\infty}$ is \emph{non}-flat, then $g(t)$ develops a Type-II(b) singularity in infinite-time.

\subsection{The positive mass case.} In the following we focus the attention on Ricci flow solutions starting in $\Gk$. We recall that any metric in $\Gaf$ with positive-mass belongs to $\mathcal{G}_{0}$. We first prove that the mass of any Ricci flow solution in $\Gk$ is in fact preserved in any region where $\bb$ is large, uniformly in time. 
\begin{lemma}\label{lemmapositivemasscase}
 Let $(\R^{4},g(t))_{t\geq 0}$ be the maximal Ricci flow solution evolving from some $g_{0}\in\Gk$ with mass $m_{g_{0}}$. There exists $\nu > 0$ such that for all $\gamma < (m_{g_{0}})^{-1}$ we have
\[
\inf_{\R^{4}\times [0, + \infty)}\,\bb^{\nu}(\cc - \gamma) > -\infty.
\]
\end{lemma}
\begin{proof} Given $\nu > 0$ and $\gamma < (m_{g_{0}})^{-1}$ we let $f = \bb^{\nu}(\cc - \gamma)$. From Lemma \ref{conservationmassGk} we derive that $f(s,t) \rightarrow \infty$ at spatial infinity on any time-slice. Assume that there exist $p_{0}\in\R^{4}$ and $t_{0} > 0$ such that $f$ attains a negative minimum at $(p_{0},t_{0})$. The evolution equation of $f$ is
\[
\partial_{t}f(p_{0},t_{0}) \geq \frac{1}{\bb^{2}}\left(f\left(\nu^{2}\bb_{s}^{2}(1 - \frac{c - \gamma}{c}) + 2\nu\uu^{2} -4\nu\right) -2\uu^{2}\bb^{\nu}\cc\right).
\]
\noindent If $\lvert f(p_{0},t_{0})\rvert$ is large, then $\bb(p_{0},t_{0})$ is large too. Thus, from the rotational symmetry type of bounds in Corollary \ref{rotationalsymmetryboundsGk} we see that 
\[
\frac{\lvert c - \gamma\rvert}{\cc}(p_{0},t_{0}) \leq 1 + \frac{m_{g_{0}}^{-1}}{\cc}(p_{0},t_{0}) \leq \alpha < \infty.
\]
\noindent Therefore, we obtain
\[
\partial_{t}f(p_{0},t_{0}) \geq \frac{\lvert f \rvert}{\bb^{2}}\left(-(1 +\alpha)\nu^{2}\bb_{s}^{2} - \frac{2}{\lvert f \rvert}\uu^{2}\bb^{\nu}\gamma + 2\nu\right).
\]
\noindent Since by Lemma \ref{firstderivativesboundedGk} the derivative $\bb_{s}$ is uniformly bounded in time, we may pick $\nu > 0$ small enough such that
\[
\partial_{t}f(p_{0},t_{0}) \geq \frac{\lvert f \rvert}{\bb^{2}}\left(\nu - 2\frac{\uu (m_{g_{0}})^{-2}}{\lvert f \rvert}\right).
\] 
\noindent Finally, once we let $\lvert f \rvert$ be sufficiently large depending on the choice of $\nu$ and on the value of $m_{g_{0}}$, we conclude that $\partial_{t}f(p_{0},t_{0}) > 0$, which completes the proof.
\end{proof}
We may now prove that any Ricci flow in $\Gk$ converges to $\nut$ in infinite-time.
\begin{proof}[Proof of Theorem \ref{maintheoremconvergence}] Let $t_{j}\nearrow \infty$ and consider the pointed sequence of Ricci flow solutions $(\R^{4},g_{j}(t),\origin)_{t\in[-t_{j},0]}$, with $g_{j}(t) = g(t_{j} + t)$. According to Proposition \ref{curvatureboundedintime}, the curvature is uniformly bounded along the sequence. Moreover, the first order derivatives are controlled in the space-time by Lemma \ref{firstderivativesboundedGk}. From the boundary conditions we also derive that conditions \eqref{lowerboundsweetCA}, \eqref{bbjboundedCA} and \eqref{uujboundedCA} are satisfied by the sequence given above. Therefore, after a diagonal argument we deduce that $(\R^{4},g_{j}(t),\origin)$ converges to an ancient solution $(\R^{4},g_{\infty}(t),\origin)_{t\leq 0}$ as in Proposition \ref{compactnessresult}. In particular, $(\R^{4},g_{\infty}(t))_{t\leq 0}$ is a complete, warped Berger ancient solution with monotone coefficients and curvature uniformly bounded in the space-time. Moreover, from the convergence of the warping coefficients given by Proposition \ref{compactnessresult} we find that $\cc_{\infty} \leq m_{g_{0}}^{-1}$. By Lemma \ref{keylemmaGk} and the bound on the Hopf-fiber we know that $\bb_{\infty}$ diverges at spatial infinity on any time-slice. According to Lemma \ref{keylemmaGk} and Lemma \ref{lemmapositivemasscase} we may pick $\lambda > 0$ small enough and $\alpha > 0$ such that
\[
(\bb_{\infty})_{s}\uu_{\infty}^{-1}\geq \log(\bb_{\infty}) - \frac{\alpha}{\bb_{\infty}^{\lambda}}
\]
\noindent and
\[
\cc_{\infty} \geq \frac{m_{g_{0}}^{-1}}{2} - \frac{\alpha}{\bb_{\infty}^{\lambda}}.
\]
\noindent Let $V_{\lambda}$ be the space time region where $\bb_{\infty}^{\lambda}\geq 4(\alpha + 1)(m_{g_{0}} + 1)$. Thus, in $V_{\lambda}$ we get
\[
(\bb_{\infty})_{s}\uu_{\infty}^{-1}\geq \log(\uu_{\infty}^{-1}) + \log(\cc_{\infty}) - \frac{1}{4} \geq\log(\uu_{\infty}^{-1}) - \log(m_{g_{0}}) - \log(4) - \frac{1}{4}.
\]
\noindent Since by Lemma \ref{keylemmaGk} we have $(\bb_{\infty})_{s}\uu_{\infty}^{-1} \geq \beta$, for some $\beta > 0$, we conclude that there exists a continuous function $f:[1, \infty) \rightarrow \R_{> 0}$, with $f(z)\rightarrow \infty$ as $z\rightarrow\infty$, such that
\[
(\bb_{s})_{\infty}\uu_{\infty}^{-1} \geq f(\uu_{\infty}^{-1})
\]
\noindent in $\R^{4}\times (-\infty,0]$. Therefore, the limit ancient Ricci flow $(\R^{4},g_{\infty}(t))$ belongs to $\mathcal{A}$ (see Definition \ref{definitionancientpro}). By the rigidity property in Theorem \ref{maintheoremancient} we see that $g_{\infty}(t)$ is the Taub-NUT metric $\nut$ of mass \emph{exactly} $m_{g_{0}}$ as follows from Lemma \ref{lemmapositivemasscase}.
\end{proof}
\begin{remark} We note again that the argument above works for any asymptotically flat Ricci flow with positive mass, hence proving (i) of Theorem \ref{maintheoremconvergenceALF} as well.
\end{remark}
We also get:
\begin{proof}[Proof of Corollary \ref{maincorollary}] Let $g_{0}\in\mathcal{G}_{0}$ be as in Lemma \ref{existencemetricsec}. We can then apply Corollary \ref{solutionimmortal} and Theorem \ref{maintheoremconvergence} to derive that the maximal complete, bounded curvature Ricci flow solution starting at $g_{0}$ is immortal and converges to $\nut$ in the pointed Cheeger-Gromov sense as $t\nearrow \infty$.
\end{proof}
% Since we proved in Section 2.5 that there exists $g_{0}\in\Gk$ with $\text{sec}(g_{0}) \geq 0$, one can immediately deduce from Theorem \ref{maintheoremconvergence} that there exists an immortal warped Berger Ricci flow starting at some initial metric with nonnegative sectional curvature and converging to the Ricci-flat Taub-NUT metric in infinite-time, which is exactly the statement reported in Corollary \ref{maincorollary}.
\subsection{The zero mass case.} In order to prove that any asymptotically flat Ricci flow with Euclidean volume growth, or equivalently zero mass, encounters a Type-III singularity in infinite-time, we first show that the roundness ratio converges to 1 uniformly in any space-time region where $\bb$ is large.
\begin{lemma}\label{lemma1zeromasscase} If $(\R^{4},g(t))_{t\geq 0}$ is the maximal Ricci flow solution starting at some $g_{0}\in\Gaf$ with zero mass, then there exists $\nu > 0$ such that
\[
\sup_{\R^{4}\times [0, \infty)}\,\bb^{\nu}(1 - \uu) < \infty.
\]
\end{lemma}
\begin{proof} Let $\epsilon > 0$ satisfy $\sup_{\R^{4}}(\,d_{g_{0}}(\origin,\cdot))^{2 + \epsilon}\lvert \text{Rm}\rvert_{g_{0}}(\cdot) < \infty$. If we pick $0 < \nu < \epsilon$, then we may apply l'H\^opital formula to $\bb^{\nu}(1-\uu)$ and use \eqref{rmo123} to derive that $\bb^{\nu}(1-\uu)$ converges to zero at spatial infinity. Since by Lemma \ref{ALFpreserved} the decay of the curvature persists in time, such conclusion holds along the solution. At any positive maximum point we have
\[
\partial_{t}(\bb^{\nu}(1-\uu))|_{\text{max}} \leq \frac{\bb^{\nu}(1-\uu)}{\bb^{2}}\left(\nu^{2}\bb_{s}^{2}\uu^{-1} -4\nu + 2\nu\uu^{2} -4\uu(1+\uu) \right).
\]
\noindent We may take $\nu < 1$ and hence write
\[
\partial_{t}(\bb^{\nu}(1-\uu))|_{\text{max}} \leq \frac{\bb^{\nu}(1-\uu)}{\bb^{2}}\left(\nu^{2}\bb_{s}^{2}\uu^{-1} -4\nu -4\uu\right).
\]
\noindent We note that from Lemma \ref{characterizationALF} it follows that $\uu(\cdot,0) \geq \varepsilon > 0$ in the zero-mass case. Thus, we may apply (iii) in Lemma \ref{consistencymonotone} and Lemma \ref{firstorderestimates} to write
\[
\partial_{t}(\bb^{\nu}(1-\uu))|_{\text{max}} \leq \frac{\bb^{\nu}(1-\uu)}{\bb^{2}}\left(\nu^{2}\bb_{s}^{2}\varepsilon^{-1} -4\nu -4\varepsilon\right) \leq \frac{\bb^{\nu}(1-\uu)}{\bb^{2}}\left(\alpha\nu^{2}\varepsilon^{-1} -4\nu -4\varepsilon\right),
\] 
\noindent which is negative whenever $\nu$ is sufficiently small depending on $\varepsilon$ and $\alpha$.
\end{proof}
As a consequence of the previous result, we prove that the squashing factor $\uu$ converges to 1 as time goes to infinity.
\begin{corollary}\label{corollaryzeromasscase} If $(\R^{4},g(t))_{t\geq 0}$ is the maximal Ricci flow solution starting at some $g_{0}\in\Gaf$ with zero mass, then there exist $\beta > 0$ and $\delta > 0$ such that
\[
(1 - \uu)(\cdot,t) \leq \frac{\beta}{\left(1 + \beta t\right)^{\delta}},
\]
\noindent for all times $t\geq 0$.
\end{corollary}
\begin{proof} Consider the function $h = 1 - \uu$. In the zero mass case we have $h(\origin,t) = 0$ and $h(s,t)\rightarrow 0$ as $s\rightarrow \infty$, for all positive times. At any positive maximum we may compute that:
\[
\partial_{t}(1- \uu)|_{\text{max}} \leq \frac{1}{\bb^{2}}\left(-4\uu(1-\uu^{2})\right).
\]
\noindent From Lemma \ref{lemma1zeromasscase} we see that $\bb \leq \alpha(1-\uu)^{-\nu}$, yielding
\[
\partial_{t}(1- \uu)|_{\text{max}} \leq \left(-\frac{(1-\uu)^{\frac{2}{\nu}}}{\alpha}\right)4\uu(1-\uu^{2})
\]
\noindent Since the roundness ratio is uniformly bounded from below by Lemma \ref{consistencymonotone}, we find that there exists $\varepsilon > 0$ such that
\[
\partial_{t}(1- \uu)|_{\text{max}} \leq -\left(\frac{4}{\alpha}\varepsilon(1 + \varepsilon)\right)(1-\uu)^{\frac{2}{\nu} + 1} \leq -\beta(1-\uu)^{\frac{2}{\nu} + 1}.
\]
\noindent We may then apply the maximum principle and integrate the previous relation to obtain
\[
(1 - \uu)_{\text{max}}(t) \leq \frac{\beta}{(1 + \beta t)^{\frac{\nu}{2}}},
\]
\noindent up to choosing $\beta$ large enough.
\end{proof}
We finally prove that any immortal Ricci flow in $\Gaf$ with zero mass encounters a Type-III singularity at infinite time.
\begin{proof}[Proof of (ii) in Theorem \ref{maintheoremconvergenceALF}]
Given an immortal Ricci flow solution $(\R^{4},g(t))_{t\geq 0}$ starting at $g_{0}\in\Gaf$ with $m_{g_{0}} = 0$, let us assume for a contradiction that there is a Type-II(b) singularity at infinite time. One can then argue as for the proof of Proposition \ref{curvatureboundedintime} and deduce that there exists a space-time sequence $(p_{j},t_{j})$, with $t_{j}\nearrow \infty$, such that the pointed sequence $(\R^{4},g_{j}(t),p_{j})$, defined by $g_{j}(t) = \lambda_{j}g(t_{j}+t/\lambda_{j})$, where $\lambda_{j}=\lvert \text{Rm}\rvert(p_{j},t_{j})$, converges to an eternal Ricci flow solution $(M_{\infty},g_{\infty}(t),p_{\infty})$. By Lemma \ref{characterizationALF} and Lemma \ref{ALFpreserved} there exists $\varepsilon > 0$ such that 
\[
\uu(\cdot,t) \geq \varepsilon > 0,
\]
\noindent for all times $t\geq 0$. We may then apply Lemma \ref{lowerboundforCA} and derive that \eqref{lowerboundsweetCA}, \eqref{uujboundedCA} are satisfied along the Ricci flow sequence. From Corollary \ref{corollarycccurvature} we also get
\[
\bb_{j}(p_{j},0) = \sqrt{\lambda_{j}}\bb(p_{j},t_{j}) = \sqrt{(\cc^{2}\lvert\text{Rm}\rvert)(p_{j},t_{j})}\uu^{-1}(p_{j},t_{j}) \leq \alpha\varepsilon^{-1},
\]
\noindent for some $\alpha > 0$. Therefore, the compactness result in Proposition \ref{compactnessresult} holds and we may hence use Corollary \ref{corollaryzeromasscase} to deduce that $(M_{\infty},g_{\infty}(t),p_{\infty})_{t\in \R}$ is an eternal warped Berger solution to the Ricci flow such that
\[
(1 - \uu_{\infty})(q,t) = \lim_{j\rightarrow\infty}(1 - \uu)(s_{j}(q),t_{j} + t) = 0,
\]
\noindent for all $q\in M_{\infty}$ and for all $t\in \R$. Thus $g_{\infty}(t)$ is rotationally symmetric. Since by Lemma \ref{lowerboundforCA} we see that $g_{\infty}(t)$ has Euclidean volume growth, we can follow the proof of Proposition \ref{curvatureboundedintime} to conclude that $g_{\infty}(t)$ is flat, therefore contradicting the choice of the rescaling factors.

We have shown that for any Ricci flow solution starting at some $g_{0}\in\Gaf$ with zero mass, the singularity forming at infinite time is of Type-III. Given $t_{j}\nearrow \infty$, the Ricci flow sequence $(\R^{4},g_{j}(t),\origin)_{t\in[-t_{j},0]}$ satisfies the assumptions in Proposition \ref{compactnessresult} and hence converges to a warped Berger solution $(\R^{4},g_{\infty}(t),p_{\infty})$. From the Type-III condition we see that $\lvert \text{Rm}_{\infty}\rvert_{\infty}(\cdot,t) \equiv 0$, meaning that $(M_{\infty},g_{\infty})$ is the Euclidean space.
\end{proof}

\bibliographystyle{alpha}
\bibliography{referencesconvergence}

\newcommand{\etalchar}[1]{$^{#1}$}
\begin{thebibliography}{CCG{\etalchar{+}}08}

\bibitem[AK75]{classichom}
Dmitri Alekseevskii and Boris~N. Kimelfeld.
\newblock \emph{Structure of homogeneous {R}iemannian spaces with zero {R}icci
  curvature}.
\newblock {\em \emph{Funktional. Anal. i Prilov Zen.}}, \textbf{9}(2):5--11,
  1975.

\bibitem[App18]{appleton}
Alexander Appleton.
\newblock \emph{A family of non-collapsed steady {R}icci solitons in even
  dimensions greater or equal to four}.
\newblock {\em \emph{arXiv preprint arXiv:1708.00161}}, 2018.

\bibitem[App19]{appleton2}
Alexander Appleton.
\newblock \emph{Eguchi-{H}anson singularities in {U}(2)-invariant {R}icci
  flow}.
\newblock {\em \emph{arXiv preprint arXiv:1903.09936 }}, 2019.

\bibitem[Bam18]{bamler}
Richard~H. Bamler.
\newblock \emph{Long-time behavior of $3$–dimensional {R}icci flow:
  introduction}.
\newblock {\em \emph{Geom. Topol.}}, \textbf{22}(2):757--774, 2018.

\bibitem[BK16]{bettiol}
Renato~G. Bettiol and Anusha~M. Krishnan.
\newblock \emph{Four-dimensional cohomogeneity one {R}icci flow and nonnegative
  sectional curvature}.
\newblock {\em \emph{arXiv preprint arXiv:1606.00778}}, 2016.

\bibitem[B{\"o}h15]{bohmIII}
Christoph B{\"o}hm.
\newblock \emph{On the long time behavior of homogeneous {R}icci flows}.
\newblock {\em \emph{Comm. Math. Helv}}, \textbf{90}(3):543--571, 2015.

\bibitem[CCG{\etalchar{+}}08]{ricciflowtechniques2}
Bennett Chow, S~Chu, David Glickenstein, Christine Guenther, James Isenberg,
  Tom Ivey, Dan Knopf, Peng Lu, Feng Luo, and Lei Ni.
\newblock \emph{The {R}icci Flow: Techniques and Applications Part {II}:
  Analytic Aspects}.
\newblock {\em \emph{Mathematical Surveys and Monographs}}, 144, 2008.

\bibitem[Che09]{chen2009}
Bing-Long Chen.
\newblock \emph{Strong uniqueness of the {R}icci flow}.
\newblock {\em \emph{J. Differential Geom.}}, \textbf{82}(2):363--382, 06 2009.

\bibitem[CK04]{IRF}
Bennett Chow and Dan Knopf.
\newblock {\em The {R}icci flow: an introduction}, volume~1.
\newblock American Mathematical Soc., 2004.

\bibitem[CZ06]{uniqueness}
Bing-Long Chen and Xi-Ping Zhu.
\newblock \emph{Uniqueness of the {R}icci flow on complete noncompact
  manifolds}.
\newblock {\em \emph{J. Differential Geom.}}, \textbf{74}(1):119--154, 2006.

\bibitem[DG19]{work}
Francesco Di~Giovanni.
\newblock \emph{Rotationally symmetric {R}icci flow on $\mathbb{R}^{n+1}$}.
\newblock {\em \emph{arXiv preprint arXiv:1904.09555, to appear in Adv. Math.
  (2021)}}, 2019.

\bibitem[DG20]{work2}
Francesco Di~Giovanni.
\newblock \emph{Ricci flow of warped {B}erger metrics on $\mathbb{R}^{4}$}.
\newblock {\em \emph{Calc.Var.}}, \textbf{59}(162), 2020.

\bibitem[DK20]{alix}
Alix Deruelle and Klaus Kr{\"o}ncke.
\newblock \emph{Stability of {ALE} {R}icci-flat manifolds under {R}icci flow}.
\newblock {\em \emph{J. Geom. Anal.}}, pages 1--42, 2020.

\bibitem[FLS17]{jason}
Joel Fine, Jason~D. Lotay, and Michael Singer.
\newblock \emph{The space of hyperk\"ahler metrics on 4-manifold with
  boundary}.
\newblock {\em \emph{Forum Math. Sigma}}, \textbf{5}:e6, 2017.

\bibitem[Ham82]{threemanifolds}
Richard~S. Hamilton.
\newblock \emph{Three-manifolds with positive {R}icci curvature}.
\newblock {\em \emph{J. Differential Geom.}}, \textbf{17}(2):255--306, 1982.

\bibitem[Ham95]{formationsingularities}
Richard~S. Hamilton.
\newblock \emph{The formation of singularities in the {R}icci flow.}
\newblock {\em \emph{Surveys in Differential Geometry}}, \textbf{2}:7--136,
  1995.

\bibitem[Haw77]{hawking}
Stephen~W. Hawking.
\newblock \emph{Gravitational instantons}.
\newblock {\em \emph{Phys. Lett.}}, \textbf{60}A(2):81--83, 1977.

\bibitem[HM14]{haslhofer}
Robert Haslhofer and Reto M{\"u}ller.
\newblock \emph{Dynamical stability and instability of {R}icci-flat metrics}.
\newblock {\em \emph{Math. Ann.}}, \textbf{360}(1-2):547--553, 2014.

\bibitem[HSW07]{numerics}
Gustav Holzegel, Thomas Schmelzer, and Claude Warnick.
\newblock \emph{Ricci flows connecting {T}aub-{B}olt and {T}aub-{NUT} metrics}.
\newblock {\em \emph{Class. Quantum Gravity}}, \textbf{24}(24):6201--6217,
  2007.

\bibitem[IK{\v{S}}16]{IKS1}
James Isenberg, Dan Knopf, and Nata{\v{s}}a {\v{S}}e{\v{s}}um.
\newblock \emph{{R}icci flow neckpinches without rotational symmetry}.
\newblock {\em \emph{Comm. Part. Diff. Eq.}}, \textbf{41}(12):1860--1894, 2016.

\bibitem[IK{\v{S}}19]{IKS2}
James Isenberg, Dan Knopf, and Nata{\v{s}}a {\v{S}}e{\v{s}}um.
\newblock \emph{Non-{K}\"{a}hler {R}icci flow singularities that converge to
  {K}\"{a}hler-{R}icci solitons}.
\newblock {\em \emph{arXiv preprint arXiv:1703.02918v5}}, 2019.

\bibitem[Ive94]{Ivey}
Thomas~A. Ivey.
\newblock \emph{The ricci flow on radially symmetric $\mathbb{R}^{3}$}.
\newblock {\em \emph{Comm. Part. Diff. Eq.}}, \textbf{19}(9-10):1481--1500,
  1994.

\bibitem[L{\v{S}}14]{lottsesum}
John Lott and Nata{\v{s}}a {\v{S}}e{\v{s}}um.
\newblock \emph{Ricci flow on three-dimensional manifolds with symmetry}.
\newblock {\em \emph{Comm. Math. Helv}}, \textbf{89}(1):1--32, 2014.

\bibitem[LZ16]{Lott}
John Lott and Zhou Zhang.
\newblock \emph{Ricci flow on quasiprojective manifolds \textrm{2}}.
\newblock {\em \emph{J. Eur. Math. Soc.}}, \textbf{18}(8):1813--1854, 2016.

\bibitem[Mar19]{marxen}
Tobias Marxen.
\newblock \emph{Convergence of {R}icci {F}low on a {C}lass of {W}arped
  {P}roduct {M}etrics}.
\newblock {\em \emph{J. Geom. Anal.}}, pages 1--35, 2019.

\bibitem[Min10]{minerbe}
Vincent Minerbe.
\newblock \emph{On the asymptotic geometry of gravitational instantons}.
\newblock {\em \emph{Ann. Sci. de l'Ecole Norm. Superieure}}, Ser. 4,
  43(6):883--924, 2010.

\bibitem[OW07]{woolgar}
Todd Oliynyk and Eric Woolgar.
\newblock \emph{Rotationally symmetric {R}icci flow on asymptotically flat
  manifolds.}
\newblock {\em \emph{Comm. Anal. Geom.}}, \textbf{15}(3):535--568, 2007.

\bibitem[Per02]{pseudolocality}
Grisha Perelman.
\newblock \emph{The entropy formula for the {R}icci flow and its geometric
  applications}.
\newblock {\em \emph{arXiv preprint math/0211159}}, 2002.

\bibitem[{\v{S}}e{\v{s}}06]{sesumstability}
Nata{\v{s}}a {\v{S}}e{\v{s}}um.
\newblock \emph{Linear and dynamical stability of {R}icci-flat metrics}.
\newblock {\em \emph{Duke Math. J.}}, \textbf{133}(1):1--26, 2006.

\bibitem[Shi89]{shi}
Wan-Xiong Shi.
\newblock \emph{Deforming the metric on complete Riemannian manifolds}.
\newblock {\em \emph{J. Differential Geom.}}, \textbf{30}(1):223--301, 1989.

\bibitem[Unn96]{unnebrink}
Stefan Unnebrink.
\newblock \emph{Asymptotically flat 4-manifolds}.
\newblock {\em \emph{Differential Geom. Appl.}}, \textbf{6}(3):271--274, 1996.

\bibitem[Zha08]{zhang}
Zhu-Hong Zhang.
\newblock \emph{Gradient Shrinking Solitons with Vanishing {W}eyl Tensor}.
\newblock {\em \emph{Pacific J. Math.}}, \textbf{242}, 2008.

\end{thebibliography}
\end{document}